\documentclass[11pt,  reqno]{amsart}
\usepackage[margin=1.2in,marginparwidth=1.5cm, marginparsep=0.5cm]{geometry}

\usepackage{amsmath,amssymb,amscd,amsthm,amsxtra, esint}

\usepackage{booktabs} 
\usepackage{microtype}
\usepackage{amssymb}
\usepackage{mathrsfs}

\usepackage{upgreek} 

\usepackage{color}
\usepackage[implicit=true]{hyperref}

\usepackage{xsavebox}

\usepackage{cases}

\allowdisplaybreaks[2]

\usepackage{tikz}

\usepackage{marginnote}
\usepackage{scalerel}

\usetikzlibrary{shapes.misc}
\usetikzlibrary{shapes.symbols}

\usetikzlibrary{decorations}
\usetikzlibrary{decorations.markings}

\newcommand{\Q}{\mathbb{Q}}

\newcommand{\proj}{\Pi}

\newcommand{\A}{\mathcal{A}}
\newcommand{\1}{\hspace{0.2mm}\textup{I}\hspace{0.2mm}}
\newcommand{\II}{\text{I \hspace{-2.8mm} I} }
\newcommand{\III}{\text{I \hspace{-2.9mm} I \hspace{-2.9mm} I}}
\newcommand{\IV}{\text{I \hspace{-2.9mm} V}}

\newcommand{\5}{\text{V}}

\newcommand{\6}{\text{VI}}

\newcommand{\dia}{\diamond}
\newcommand{\ZZ}{\mathfrak{Z}}

\newcommand{\nuu}{\vec{\nu}}
\newcommand{\rhoo}{\vec{\rho}}
\newcommand{\muu}{\vec{\mu}}

\newtheorem{theorem}{Theorem} [section]

\newtheorem{lemma}[theorem]{Lemma}
\newtheorem{proposition}[theorem]{Proposition}
\newtheorem{remark}[theorem]{Remark}

\DeclareMathOperator{\Law}{Law}
\DeclareMathOperator{\Id}{Id}

\newcommand{\W}{\mathcal{W}}
\newcommand{\U}{\mathcal{U}}

\newcommand{\dr}{\theta}
\newcommand{\Dr}{\Theta}
\newcommand{\Ha}{\mathbb{H}_a}
\newcommand{\noi}{\noindent}
\newcommand{\Z}{\mathbb{Z}}
\newcommand{\R}{\mathbb{R}}

\newcommand{\T}{\mathbb{T}}
\newcommand{\N}{\mathbb{N}}

\let\Re=\undefined\DeclareMathOperator*{\Re}{Re}
\let\Im=\undefined\DeclareMathOperator*{\Im}{Im}

\let\P= \undefined
\newcommand{\P}{\mathbf{P}}
\newcommand{\PP}{\mathbb{P}}

\newcommand{\F}{\mathcal{F}}
\newcommand{\NN}{\mathcal{N}}

\newcommand{\nb}{\nabla}

\newcommand{\al}{\alpha}
\newcommand{\be}{\beta}
\newcommand{\dl}{\delta}
\newcommand{\Dl}{\Delta}
\newcommand{\eps}{\varepsilon}

\newcommand{\g}{\gamma}

\newcommand{\ld}{\lambda}
\newcommand{\Ld}{\Lambda}
\newcommand{\s}{\sigma}

\newcommand{\ft}{\widehat}
\newcommand{\wt}{\widetilde}
\newcommand{\cj}{\overline}

\newcommand{\dt}{\partial_t}

\newcommand{\ta}{\theta}

\newcommand{\jb}[1]
{\langle #1 \rangle}

\renewcommand{\l}{\ell}

\renewcommand{\H}{\mathcal{H}}

\newcommand{\les}{\lesssim}
\newcommand{\ges}{\gtrsim}

\newcommand{\ind}{\mathbf 1}

\newcommand{\GG}{\mathcal{G}}

\newcommand{\E}{\mathbb{E}}

\renewcommand{\o}{\omega}
\renewcommand{\O}{\Omega}

\numberwithin{equation}{section}
\numberwithin{theorem}{section}

\newcommand{\too}{\longrightarrow}

\newcommand{\M}{\mathcal{M}}

\usepackage{tikz}

\colorlet{symbols}{black}
\colorlet{testcolor}{green!60!black}

\def\1{\mathbf{{1}}}

\usetikzlibrary{shapes.misc}
\usetikzlibrary{shapes.symbols}
\usetikzlibrary{snakes}
\usetikzlibrary{decorations}
\usetikzlibrary{decorations.markings}
\definecolor{dblue}{rgb}{0.1, 0.1, 0.9}

\tikzset{
	root/.style={circle,fill=testcolor,inner sep=0pt, minimum size=2mm},		
	dot/.style={circle,fill=black,draw=black, solid,inner sep=0pt,minimum size=0.75mm},
	bdot/.style={circle,fill=blue,draw=dblue, solid,inner sep=0pt,minimum size=0.75mm},
		}
\colorlet{symbols}{blue!90!black}

\makeatletter
\def\DeclareSymbol#1#2#3{\expandafter\gdef\csname MH@symb@#1\endcsname{\tikz[baseline=#2,scale=0.15]{#3}}%
\expandafter\gdef\csname MH@symb@#1s\endcsname{\scalebox{0.6}{\tikz[baseline=#2,scale=0.15]{#3}}}}
\def\<#1>{\csname MH@symb@#1\endcsname}
\makeatother

\DeclareSymbol{sunset}{-3}{\draw (0,0) circle [x radius = 1.5, y radius = 1]; \draw (-1.5,0) node[dot] {} -- (1.5,0) node[dot] {};}
\DeclareSymbol{tadpole}{-3}{\draw (0,0) circle [x radius = 0.75, y radius = 1]; \draw (0,-1) node[dot] {};}

\DeclareSymbol{1}{0}{\draw[white] (-.4,0) -- (.4,0); \draw (0,0)  -- (0,1.2) node[dot] {};}
\DeclareSymbol{2}{0}{\draw (-0.5,1.2) node[dot] {} -- (0,0) -- (0.5,1.2) node[dot] {};}
\DeclareSymbol{3}{0}{\draw (0,0) -- (0,1.2) node[dot] {}; \draw (-.7,1) node[dot] {} -- (0,0) -- (.7,1) node[dot] {};}
\DeclareSymbol{4}{0}{\draw (-0.36,1.2) node[dot] {} -- (0,0) -- (0.36,1.2) node[dot] {}; \draw (-1,1) node[dot] {} -- (0,0) -- (1,1) node[dot] {};}
\DeclareSymbol{30}{-3}{\draw (0,0) -- (0,-1); \draw (0,0) -- (0,1.2) node[dot] {}; \draw (-.7,1) node[dot] {} -- (0,0) -- (.7,1) node[dot] {};}
\DeclareSymbol{31}{-3}{\draw (0,0) -- (0,-1) -- (1,0) node[dot] {}; \draw (0,0) -- (0,1.2) node[dot] {}; \draw (-.7,1) node[dot] {} -- (0,0) -- (.7,1) node[dot] {};}
\DeclareSymbol{32}{-3}{\draw (0,0) -- (0,-1) -- (1,0) node[dot] {}; \draw (0,0) -- (0,-1) -- (-1,0) node[dot] {}; \draw (0,0) -- (0,1.2) node[dot] {}; \draw (-.7,1) node[dot] {} -- (0,0) -- (.7,1) node[dot] {};}
\DeclareSymbol{20}{-3}{\draw (0,0) -- (0,-1);\draw (-.7,1) node[dot] {} -- (0,0) -- (.7,1) node[dot] {};}
\DeclareSymbol{22}{-3}{\draw (0,0.3) -- (0,-1) -- (1,0) node[dot] {}; \draw (0,0.3) -- (0,-1) -- (-1,0) node[dot] {};\draw (-.7,1) node[dot] {} -- (0,0.3) -- (.7,1) node[dot] {};}
\DeclareSymbol{202}{-3}{\draw (-.6625,0) -- (0,-1) -- (.6625,0)  {}; \draw (-1.1,1) node[dot] {} -- (-.6625,0) -- (-0.365,1) node[dot] {}; \draw (0.365,1) node[dot] {} -- (.6625,0) -- (1.1,1) node[dot] {};}

\makeatletter
\def\DeclareSymbol#1#2#3{\expandafter\gdef\csname MH@symb@#1\endcsname{\tikz[baseline=#2,scale=0.15]{#3}}}
\def\<#1>{\csname MH@symb@#1\endcsname}
\makeatother

\DeclareSymbol{X}{-2.4}{\node[dot] {};}
\DeclareSymbol{1}{0}{\draw[white] (-.4,0) -- (.4,0); \draw (0,0)  -- (0,1.2) node[dot] {};}
\DeclareSymbol{2}{0}{\draw (-0.5,1.2) node[dot] {} -- (0,0) -- (0.5,1.2) node[dot] {};}

\DeclareSymbol{sunset}{-3}{\draw (0,0) circle [x radius = 1.5, y radius = 1]; \draw (-1.5,0) node[dot] {} -- (1.5,0) node[dot] {};}
\DeclareSymbol{tadpole}{-3}{\draw (0,0) circle [x radius = 0.75, y radius = 1]; \draw (0,-1) node[dot] {};}

\DeclareMathOperator{\sgn}{sgn}

\DeclareSymbol{T0}{-2.7}
 { \draw (0,0) node[ddot]{};}

\DeclareSymbol{T1}{0}
 {  
 \draw (0,0) node[dot]{} -- (0,4) node[ddot] {}; 
 \draw (-3,0) node[dot] {} -- (0,4)node[ddot] {} -- (3,0) node[dot] {};
\draw[line width=1pt] (0, 4)node[ddot, label=above:$r_1$]{} 
.. controls(3,6) and (5,6) ..
(8, 4)node[ddot, label=above:$r_2$]{}
(4, 12) node[label ] {$\underline{j = 1}$};
 }

\DeclareSymbol{T2}{0}
 {\draw (0,0) node[dot]{} -- (0,4) node[ddot] {}; 
 \draw (-3,0) node[dot] {} -- (0,4)node[ddot] {} -- (3,0) node[dot] {};
 \draw (0,-4) node[dot]{} -- (0,0) node[dot] {}; 
 \draw (-3,-4) node[dot] {} -- (0,0)node[dot] {} -- (3,-4) node[dot] {};
\draw[line width=1pt] (0, 4)node[ddot, label=above:$r_1$]{} 
.. controls(3,6) and (5,6) ..
(8, 4)node[ddot, label=above:$r_2$]{}
(4, 12) node[label ] {$\underline{j = 2}$};
 }

\DeclareSymbol{T3}{0}
 {\draw (0,0) node[dot]{} -- (0,4) node[ddot] {}; 
 \draw (-3,0) node[dot] {} -- (0,4)node[ddot] {} -- (3,0) node[dot] {};
 \draw (0,-4) node[dot]{} -- (0,0) node[dot] {}; 
 \draw (-3,-4) node[dot] {} -- (0,0)node[dot] {} -- (3,-4) node[dot] {};
\draw (10,0) node[dot]{} -- (10,4) node[ddot] {}; 
 \draw (7,0) node[dot] {} -- (10,4)node[ddot] {} -- (13,0) node[dot] {};
\draw[line width=1pt] (0, 4)node[ddot, label=above:$r_1$]{} 
.. controls(4,6) and (6,6) ..
(10, 4)node[ddot, label=above:$r_2$]{}
(5, 12) node[label ] {$\underline{j = 3}$};
 }

\tikzstyle{dot1} = [ draw=  gray!00, 
 rectangle, rounded corners, fill=gray!00,  inner sep=1pt, inner ysep=3pt]

\tikzstyle{dot2} = [ draw=  black, 
ellipse, fill=gray!00,  inner sep=1pt, inner ysep=3pt]

\tikzstyle{dot3} = [ draw=  gray!00, 
ellipse, fill=gray!00,  inner sep=1pt, inner ysep=3pt]

\DeclareSymbol{T4}{0}
 {\draw (0,0) node[dot1]{$\phantom{n_2^{(1)} = n^{(2)}}$} -- (0,15) node[dot1] {$\phantom{n^{(1)}}$}; 
 \draw (-13,0) node[dot1] {$n_1^{(1)}$} -- (0,15)node[ddot] {$\phantom{n^{(1)}}$} 
 -- (13,0) node[dot1] {$n_3^{(1)}$};
 \draw (0,-15) node[dot1]{$n_2^{(2)}$} -- (0,0) node[dot1] {$\phantom{n_2^{(1)} = n^{(2)}}$}; 
 \draw (-13,-15) node[dot1] {$n_1^{(2)}$} -- (0,0)node[dot1] {$n_2^{(1)} = n^{(2)}$} -- (13,-15) node[dot1] {$n_3^{(2)}$};
\draw (40,0) node[dot1]{{$n_2^{(3)}$}} -- (40,15) node[dot3]  {$\phantom{n^{(1)} = n^{(3)}}$} ; 
 \draw (27,0) node[dot1] {$n_1^{(3)}$} -- (40,15)node[dot3]  {$\phantom{n^{(1)} = n^{(3)}}$}  -- (53,0) node[dot1] {$n_3^{(3)}$};
\draw[line width=1pt] (0, 15)node[ddot, label=above:$r_1$]{$n^{(1)}$} 
.. controls(10,23) and (27,23) ..
(40, 15)node[dot2, label=above:$r_2$] {$n^{(1)} = n^{(3)}$} ;
 }

\begin{document}

\baselineskip = 14pt

\title[Gibbs measures for Schr\"odinger-wave system]
{phase transition of singular Gibbs measures for three-dimensional Schr\"odinger-wave system  }

\author[K.~Seong]
{Kihoon Seong}

\address{Kihoon Seong\\
Department of Mathematics\\
Cornell University\\ 
310 Malott Hall\\ 
Cornell University\\
Ithaca\\ New York 14853\\ 
USA and Max–Planck Institute for Mathematics in the Sciences, Leipzig, Germany}

\email{kihoonseong@cornell.edu}

\subjclass[2020]{60H30, 81T08, 35Q55, 35L70}

\keywords{Phase transition; singular Gibbs measure; Schr\"odinger-wave system;  invariant measure.}

\begin{abstract}
We study singular Gibbs measures for the three-dimensional Schr\"odinger-wave system, also known as the Yukawa system. Our primary result is the phase transition between weak and strong coupling cases, a phenomenon absent in one- and two-dimensional cases. Therefore, the three-dimensional model turns out to be critical, exhibiting the phase transition. In the weak coupling case, the Gibbs measure can be normalized as a probability measure and is shown to be singular with respect to the Gaussian free field. Conversely, in the strong coupling case, the Gibbs measure cannot be constructed as a probability measure. In particular, the finite-dimensional truncated Gibbs measures have no weak limit in an appropriate space, even up to a subsequence.
\end{abstract}

\maketitle

\tableofcontents

\section{Introduction}
\label{SEC:1}

\subsection{Phase transition of singular Gibbs measures}
\label{SUBSEC:intro}

We study the phase transition phenomenon of the Gibbs measure associated with the Schr\"odinger-wave systems, also known as the Yukawa system on $\T^3 = (\R/2\pi\Z)^3$. This measure is a probability measure $\rhoo$ on the space $\vec{\mathcal D}'(\T^3)$\footnote{Here, $\vec{\mathcal D}'(\T^3)=\mathcal{D}'(\T^3) \times \mathcal{D}'(\T^3) \times \mathcal{D}'(\T^3) $} of Schwartz distributions, formally written as
\begin{align}
d\rhoo (u,w, \dt w) = Z^{-1} e^{-H(u,w, \dt w)} \prod_{x\in \T^3} du(x)  dw(x)  d(\dt w)(x).
\label{gibbs}
\end{align}

\noi
Here $Z$ is the partition function, the base measure in \eqref{gibbs} is the (non-existent) Lebesgue measure on fields $(u,w, \dt w)$, and $H$ is the Hamiltonian as follows
\begin{align}
H(u,w,\dt w) =\frac 12 \int_{\T^3}|\nb u|^2 \, dx
+\frac 14 \int_{\T^3}| \nb w|^2 \, dx+\frac 14\int_{\T^3}| \dt w |^2 \,dx+\frac {\ld}2 \int_{\T^3} |u|^2w \,dx 
\label{Ham}
\end{align}

\noi
where the coupling constant\footnote{Since $|u|^2w$ is not sign-definite, the sign of $\ld$ does not play any role. } $\ld \in \R\setminus \{0\}$ measures the strength
of the interaction potential. The interaction potential, represented by the last term in \eqref{Ham}, corresponds to the Yukawa interaction between the two fields $u$ and $w$. Thanks to the conservation of the Hamiltonian $H(u, w, \dt w)$ and the invariance\footnote{Liouville's theorem} of volume measures along the Hamiltonian dynamics \eqref{SW1}, the Gibbs measure \eqref{gibbs} can be formally understood as an invariant measure for the following Hamiltonian PDE
\begin{align}
\begin{cases}
i \dt u +\Dl u = \ld uw\\
\dt^2 w -\Dl w = \ld |u|^2.\\
\end{cases}
\label{SW1}
\end{align}

\noi
This is the so-called Schr\"odinger-wave system on $\T^3$, often referred to as the Yukawa system. See \cite{GhS0, GhS1} for more physical explanations about the model.

Inspired by the statistical approach to nonlinear Hamiltonian dynamics, the construction and invariance of Gibbs measures for Hamiltonian PDEs have attracted considerable attention in recent years. For a detailed account on this subject, see \cite{LRS, BO94, BO96, BS, BT1, BT2, ChaKrik, LMW, Chatt, CFL, DNY2, DNY3, OOT1, OOT2, OST1, BDNY} and references therein. This statistical perspective on nonlinear Hamiltonian PDEs, originating from the seminal works of Lebowitz, Rose, and Speer \cite{LRS}, Bourgain \cite{BO94}, and Brydges and Slade \cite{BS}, was inspired by the developments in the probabilistic approach to Euclidean quantum field theory (EQFT) during the 1970s and 1980s. In particular, the Gibbs measure \eqref{gibbs} is closely related to the focusing Hartree $\Phi^4_3$-measure, which has been extensively studied in constructive quantum field theory. For a more detailed explanation of the relation with the Hartree $\Phi^4_3$-measure, see Subsection \ref{Subsub:Hartree}.

Compared to the construction of the $\Phi^4$-measure (i.e.~$\Phi^4$-EQFT), the primary challenge in constructing the Gibbs measure \eqref{gibbs} arises from the focusing nature of the interaction potential. Specifically, the potential energy $\frac {\ld}{2} \int |u|^2w dx$ is unbounded from below for any $\ld \in \R\setminus\{0\}$. This unboundedness implies that $H(u,w,\dt w)$ is not bounded below, preventing $e^{-H(u,w, \dt w)}$ from being integrable with respect to the volume measure. Consequently, the problem differs significantly from the defocusing case where the Hamiltonian is bounded from below. While the construction of defocusing Gibbs measures is well understood due to strong interest in constructive Euclidean quantum field theory, the construction of focusing Gibbs measures, first studied by Lebowitz, Rose, and Speer \cite{LRS} and Brydges and Slade \cite{BS}, has not been fully explored.

Indeed, even in the one-dimensional case $\T$, Lebowitz, Rose, and Speer \cite{LRS} observed that the Gibbs measure \eqref{gibbs} cannot be constructed as a probability measure due to the focusing nature of the potential energy
\begin{align*}
Z= \int e^{- \frac \ld 2  \int_{\T}    |u|^2 w \, dx} d\muu(u,w, \dt w)=\infty 
\end{align*}

\noi 
for any $\ld \in \R\setminus\{0\}$, where $\vec \mu=\mu \otimes \mu \otimes \mu_0$, $\mu$ is the Gaussian free field and $\mu_0$ is the white noise measure. To overcome this issue, Lebowitz, Rose, and Speer proposed in \cite{LRS} adding a mass cut-off, specifically replacing the density for the Gibbs measure with\footnote{Here $Z$ denotes different normalizing constants that may differ from one line to line.}
\begin{align}
d\rhoo (u,w, \dt w) &= Z^{-1} e^{-H(u,w,\dt w)} \ind_{ \{\int_{\T}  | u|^2    dx \le K\}}\prod_{x\in \T} du(x)  dw(x)  d(\dt w)(x) \notag \\
&=Z^{-1} e^{- \frac \ld 2  \int_{\T}    |u|^2 w \, dx} \ind_{ \{|\int_{\T} | u|^2    dx| \le K\}} d\muu(u,w, \dt w)
\label{GibbsY}
\end{align}

\noindent
for some parameter $K > 0$, in order to recover the integrability of the density with respect to the Gaussian field $\muu$. The introduction of the mass cutoff is justified by the fact that mass is a conserved quantity for the Hamiltonian PDE \eqref{SW1}, making \eqref{GibbsY} an invariant measure under the flow of \eqref{SW1}. After the seminal work by Lebowitz, Rose, and Speer \cite{LRS} on the construction of focusing Gibbs measures \eqref{GibbsY} with an $L^2$-cutoff, Bourgain \cite{BO94b} proved that these Gibbs measures are invariant under the Hamiltonian PDE \eqref{SW1} on $\T$.

In contrast, in the two-dimensional case $\T^2$, the free field is no longer supported on a space of functions. Thus, it becomes necessary to interpret samples as distributions. In fact, it can be promptly shown that $\int_{\T^2} |u|^2 w  dx$ is almost surely infinite under the Gaussian field $\muu$. To address this issue, one must subtract suitable ``counter-terms" from a regularized version of the interaction potential, a procedure known as renormalization, to compensate for divergences. In \cite{Seong}, the author applied proper renormalization to the interaction potential $\int_{\T^2} |u|^2w  dx$ and the $L^2$-cutoff, and constructed the Gibbs measure
\begin{align}
d\rhoo(u,w,\dt w)=Z^{-1}  e^{ -\frac \ld2 \int_{\T^2}   \; :  |u|^2w : \;   \,dx }   \ind_{\{\int_{\T^2} \; :|u|^2: \;  dx \, \leq K\}}   d\muu(u,w, \dt w)
\label{GibbsYY}
\end{align}

\noindent
for any $\lambda \in \R \setminus \{0\}$ and $K > 0$, where $  :  |u|^2w : $ and $ :|u|^2: $ denote the standard Wick renormalization (see below). After the Gibbs measure \eqref{GibbsYY} on $\T^2$ was constructed, the author \cite{Seong} proved that the Gibbs measure \eqref{GibbsYY} is invariant under the flow of Hamiltonian PDE \eqref{SW1} on $\T^2$.

Our paper continues the study of the focusing Gibbs measure \eqref{gibbs} in the three-dimensional case $\T^3$, specifically examining the phenomenon of phase transition between weak and strong coupling cases. We first state our main result in a somewhat informal manner. See Theorem \ref{THM:Gibbs} for the precise statement.

\begin{theorem}\label{THM:Pha0}
The following phase transition holds for the Gibbs measure $\rhoo$ in \eqref{gibbs}.
\begin{itemize}
\item[\textup{(i)}]\textup{(weak coupling).}
Let $0 < |\ld|\ll 1$ and $K>0$. Then, by introducing a further counter-term beyond Wick renormalizations, the Gibbs measure $\rhoo$ \eqref{gibbs}
exists as a probability measure. In this case, the resulting Gibbs measure $\rhoo$ and the Gaussian field $\muu$  are mutually singular.

\item[\textup{(ii)}] \textup{(strong coupling).}
When $|\ld| \gg1$, the Gibbs measure cannot be constructed as a probability measure for any $K> 0$. Furthermore, the finite-dimensional truncated Gibbs measures  \textup{(}see \eqref{truGibbsN} below\textup{)}  do not have a weak limit in a natural space, even up to a subsequence.
\end{itemize}
\end{theorem}

The Gibbs measures \eqref{GibbsY} and \eqref{GibbsYY} in the one- and two-dimensional cases can be constructed as probability measures for any $\lambda \in \R \setminus \{0\}$ and any $K > 0$. Hence, our results show that the three-dimensional model is critical, exhibiting a phase transition that is absent in the one- and two-dimensional cases.

Compared to the two-dimensional case, the three-dimensional case is more challenging due to further divergences, known as ultraviolet (i.e.~small-scale) divergences. Indeed, we show that the interaction potential, even with Wick renormalization, fails to converge
\begin{align*}
\E_{\muu}\Bigg[\bigg|   \frac 1 2  \int_{\T^3}    :\!  |u|^2 w \! : \,dx 
 \bigg|^2 \Bigg]=\infty.
\end{align*}

\noi
Therefore, the main difficulty is to handle the additional ultraviolet divergences beyond Wick renormalization as follows: 
\begin{align}
d\rhoo(u,w, \dt w ) = Z^{-1}  
e^{ -\frac \ld2 \int_{\T^3}   \; :  |u|^2w : \;   \,dx -\infty } \ind_{ \{|\int_{\T^3} : | u|^2 :   dx| \le K\}}  d \muu(u,w, \dt w),
\label{GibbsW}
\end{align}

\noi
where the term $- \infty$ denotes another (non-Wick) renormalization (see~\eqref{addcount} below).  While the Gibbs measures \eqref{GibbsY} and \eqref{GibbsYY} in the one- and two-dimensional cases are absolutely continuous with respect to the Gaussian field $\muu$, the situation is different in the three-dimensional case. Due to subtracting the additional counter term along with Wick renormalization, the Gibbs measure $\rhoo$ \eqref{GibbsW} and the Gaussian field $\muu$ are mutually singular, which introduces additional difficulties in every aspect of proving the main result. In particular, the singularity of the Gibbs measure causes further challenges in proving non-convergence in the strong coupling case.

\subsection{Main results}
\label{SUBSEC:renorm}

In this subsection, we introduce the main results. Before presenting the main result, we provide a renormalization procedure required to construct the Gibbs measure in \eqref{GibbsW}.

Given $ s \in \R$, let $\mu_s$ denote a Gaussian measure with the Cameron-Martin space $H^s(\T^3)$,   formally defined by
\begin{align}
d\mu_s= Z_s^{-1} e^{-\frac 12 \| u\|_{{H}^{s}}^2} du
& =  Z_s^{-1} \prod_{n \in \Z^3} e^{-\frac 12 \jb{n}^{2s} |\ft u(n)|^2}  d\ft u(n) , 
\label{fracGFF}
\end{align}

\noi
where $\jb{\,\cdot\,} = (1+|\,\cdot\,|^2)^\frac{1}{2}$.
When $ s= 1$, 
the Gaussian measure $\mu_s$ corresponds to 
 the massive Gaussian free field, 
 while it corresponds to  the white noise measure $\mu_0$ when $s = 0$.
For simplicity, we set 
\begin{align}
\mu = \mu_1
\qquad \text{and}
\qquad 
\muu = \mu \otimes \mu.
\label{gauss1}
\end{align}

\noi
Define the index sets $\Ld$ and $\Ld_0$ by 
\begin{align*}
\Ld = \bigcup_{j=0}^{2} \Z^j\times \N \times \{ 0 \}^{2-j}
\qquad \text{and}\qquad \Ld_0 = \Ld \cup\{(0, 0, 0)\}
\end{align*}

\noi
such that $\Z^3 = \Ld \cup (-\Ld) \cup \{(0, 0, 0)\}$.
Then, 
let 
$\{ g_n \}_{n \in \Ld_0}$ and $\{ h_n \}_{n \in \Ld_0}$ be sequences of mutually independent standard complex-valued\footnote{This means that $g_0,h_0\sim\NN_\R(0,1)$
and  $\Re g_n, \Im g_n, \Re h_n, \Im h_n \sim \NN_\R(0,\tfrac12)$
for $n \ne 0$.} Gaussian random variables and 
set $g_{-n} := \cj{g_n}$ and $h_{-n} := \cj{h_n}$ for $n \in \Ld_0$\footnote{Although we discuss the Gibbs measure construction in the real-valued setting, the results also apply to the complex-valued setting. In the complex-valued setting, the Wick renormalization can be defined using the Laguerre polynomial.}.
We now define random distributions $u= u^\o$ and $w = w^\o$ by 
the following  random Fourier series\footnote{By convention, 
we endow $\T^3$ with the normalized Lebesgue measure $dx_{\T^3}= (2\pi)^{-3} dx$.}:
\begin{equation} 
u^\o = \sum_{n \in \Z^3 } \frac{ g_n(\o)}{\jb{n}} e^{i n\cdot x}
 \qquad
\text{and}
\qquad
w^\o = \sum_{n\in \Z^3} \frac{ h_n(\o)}{\jb{n}} e^{i n\cdot x}. 
\label{rFou}
\end{equation}

\noi
By denoting the law of a random variable $X$ by $\Law(X)$  (with respect to the underlying probability measure $\PP$), we have
\begin{align*}
\Law (u, w) = \muu = \mu \otimes \mu
\end{align*}

\noi
for $(u, w)$ in \eqref{rFou}.
Note that  $\Law (u, w) = \muu$ is supported on $\vec H^{s}(\T^3): = H^{s}(\T^3)\times H^{s}(\T^3)$ for $s < -\frac 12$ but not for $s \geq -\frac 12$ and more generally in $B^{s}_{p,q}(\T^3) \times B^{s}_{p,q}(\T^3)$
for any $1 \le p,q \le \infty$ and $s < -\frac 12$.

\begin{remark}\rm
(i) For technical considerations, we employ a massive Gaussian free field as our reference measure by introducing an identity ``mass" term into the covariance $(1-\Delta)^{-1}$. This avoids the degeneracy of the zeroth Fourier mode. This approach is based on considering the Gibbs measure
\begin{align}
d\rhoo (u,w, \dt w) &= Z^{-1} e^{-H(u,w,\dt w)-M(u)} \ind_{ \{\int_{\T}  | u|^2    dx \le K\}}\prod_{x\in \T} du(x)  dw(x)  d(\dt w)(x).
\label{GibbsMu}
\end{align}

\noi
Since the $L^2$-norm of the Schr\"odinger component $M(u)  = \int |u|^2 dx$ is known to be conserved, \eqref{GibbsMu} is an invariant measure for \eqref{SW1}. Regarding the wave dynamics, because the zeroth frequency is not controlled due to the lack of conservation of the $L^2$-mass of $w$, we work with the massive case, that is, $\dt^2w+(1-\Dl)w$, in this paper.

\smallskip

\noi
(ii) We note that the Gibbs measure $\rhoo$ on $ (u,w, \dt w)$, formally defined in \eqref{gibbs}, decouples into the Gibbs measure $\rhoo_0$ on the component $(u,w)$ and the white noise $ \mu_0$ on the component $\dt w$.
Therefore, once the Gibbs measure $\rhoo_0$  on $(u,w)$ is established,  
we can construct the Gibbs measure $\rhoo$ on $(u,w,\dt w)$ by setting 
\begin{align*}
\rhoo(u,w,\dt w) 
= \rhoo_0 \otimes \mu_{0}(u,w, \dt w). 
\end{align*}

\noi
Thus, in the following, we focus only on the construction of the (renormalized) Gibbs measure $\rhoo_0$ on $(u,w)$.
\end{remark}

Given $N \in \N$, we define the  frequency projector $\pi_N = \pi_N^\text{cube}$  onto
the  frequencies $\{|n| \le N \}$
\begin{align}
\pi_N \phi = \sum_{|n| \le N} \ft \phi(n) e^{in\cdot x}.
\label{pi}
\end{align}

\noi
We set $\phi_N:=\pi_N \phi$. Since $\muu$ is supported on the space of distributions, we encounter the following ultraviolet (small-scale) problems. For any $(u,w)$ under the free field $\muu$ and each fixed $x\in \T^3$, $u_N$ and $w_N$ are mean-zero Gaussian random fields with 
\begin{align}
\E_\mu\big[|u_N(x)|^2\big] &= \sum _{\substack{n \in \Z^3 }} \frac{\chi^2_N(n)}{\jb{n}^2}=\<tadpole>_{S,N}
\sim N \too \infty, \label{tadpole1}\\
\E_\mu\big[w_N^2(x)\big] &= \sum _{\substack{n \in \Z^3 }} \frac{\chi^2_N(n)}{\jb{n}^2}=\<tadpole>_{W,N}
\sim N \too \infty, \label{tadpole}
\end{align}

\noi
as $N \to \infty$. Note that $\<tadpole>_{S,N}$ and $\<tadpole>_{W,N}$ are independent of $x\in \T^3$ due to the stationarity of the Gaussian free field $\mu$. We define the Wick products $:\! u_N^2 \!: $ and $:\! u_N^2w_N \!: $ as follows \footnote{The pairing $2\E_{\muu}[u_N(x)w_N(x)]=0$ does not make a divergent contribution due to the independence.}
\begin{align}
:\! |u_N|^2 \!: \, &= u_N^2 - \<tadpole>_{S,N}\label{Wick01}\\
:\! |u_N|^2 w_N \!: \, &= u_N^2w_N -\<tadpole>_{S,N} w_N=:\! u_N^2 \!:  w_N. \label{Wick02}
\end{align}

\noi
This implies that the contracted graphs causing divergent contributions are excluded during the Wick renormalization process. Therefore, we study the subsequent renormalized interaction potential
\begin{align}
\H_N(u,w):=\frac \ld 2  \int_{\T^3}    :\!  |u_N|^2 w_N \! : \,dx.
\label{wickpoten}
\end{align}

\noi
In contrast to the two-dimensional case, the interaction potential \eqref{wickpoten}, even with Wick renormalization, fails to converge (see Subsection \ref{SUBSEC:Killdiv})
\begin{align}
\E_{\muu }\Bigg[\bigg|   \frac \ld 2  \int_{\T^3}    :\!  |u_N|^2 w_N \! : \,dx 
\bigg|^2 \Bigg]\sim \ld^2\<sunset>_N\to \infty
\label{furth01}
\end{align}

\noi
as $N\to \infty$. The additional divergence arises from the ``sunset" diagram, which appears in the perturbation theory of $\Phi^4$ quantum field theory, given by
\begin{align} 
\<sunset>_N=
\sum_{\substack{n_1 + n_2 + n_3 = 0\\ |n_1|,|n_2|,|n_3|\le N}} \frac{1}{\langle n_1 \rangle^2 \langle n_2 \rangle^2 \langle n_3 \rangle^2}\sim \log N\to \infty
\label{sunset0}
\end{align}	

\noi
as $N\to \infty$.  Due to these additional divergences, we introduce further scale-dependent renormalization constants $\dl_{N,\ld}$ to the (Wick) renormalized interaction potential \eqref{wickpoten} as follows
\begin{align} 
\begin{split}
\H_N^\dia(u,w):=\frac \ld 2  \int_{\T^3}    :\!  |u_N|^2 w_N \! : \,dx+\dl_{N,\ld},
\end{split}
\label{addcount}
\end{align}

\noi
where $\dl_{N,\ld}\sim \ld^2\<sunset>_N$ is an additional diverging constant (as $N \to \infty$)  defined in \eqref{log0} below. To isolate these additional divergences in the interaction potential \eqref{wickpoten}, we use the stochastic variational approach (Lemma \ref{LEM:BoueDupu}). The remaining divergences are then eliminated by the additional counterterm. See Subsection \ref{SUBSEC:Killdiv}.

Based on the additional renormalization in \eqref{addcount}, we define the truncated  Gibbs measure $\rhoo_N$ as follows
\begin{align}
d\rhoo_N (u,w) = Z_N^{-1} e^{- \H_N^\dia(u,w)} \ind_{\{\int_{\T^3} \; :|u_N|^2: \;  dx \, \leq K\}}  d\muu(u,w),
\label{truGibbsN}
\end{align}

\noi
where the partition function $Z_N$ is given by
\begin{align}
Z_N = \int e^{- \H_N^\dia(u,w)} \ind_{\{\int_{\T^3} \; :|u_N|^2: \;  dx \, \leq K\}}  d\muu(u,w).
\label{LAUMP0}
\end{align}

\noi
Then, we are now ready to present the main result of this work.

\begin{theorem} \label{THM:Gibbs}
There exist $\ld_1 \geq \ld_0 > 0$ such that 
the following statements hold.
\smallskip

\begin{itemize}
\item[\textup{(i)}] 
\textup{(weak coupling)}. 
Let $0 < |\ld| < \ld_0$ and $K>0$.  Then, the sequence $\{\rhoo_N\}$ of truncated Gibbs measures in \eqref{truGibbsN} is tight. Hence, there exists a subsequence $\{\rhoo_{N_k}\}$ that converges weakly to a probability measure, formally written as
\begin{align}
d\rhoo(u,w) = Z^{-1}  
e^{ -\frac \ld2 \int_{\T^3}   \; :  |u|^2w : \;   \,dx -\infty } \ind_{ \{|\int_{\T^3} : | u|^2 :   dx| \le K\}}  d \muu(u,w).
\label{GibPARIS008}
\end{align}

\noi
The resulting Gibbs measure $\rhoo$ and the base Gaussian free field $\muu$ are mutually singular. In particular, the Gibbs measure $\rhoo$ is absolutely continuous with respect to the shifted measure 
\begin{align*}
\nuu=\Law_\PP (\GG^{(1)} +\GG^{(2)} + \GG^{(5)} ),
\end{align*}

\noi
where $\GG^{(1)}$ is a random field with $\Law_\PP(\GG^{(1)}) =\muu$ and $\GG^{(2)}, \GG^{(5)}$ are some appropriate quadratic and quintic order Wiener chaos in \textup{Proposition \ref{PROP:shift}}.

\smallskip
\item[\textup{(ii)}] \textup{(strong coupling)}.
Let  $|\ld| > \ld_1$ and $K>0$. Then, the Gibbs measure cannot be constructed as a probability measure in the following sense. The sequence $\{\rhoo_N\}$ of truncated Gibbs measures in \eqref{truGibbsN} does not converge to any weak limit, even up to a subsequence, as probability measures on  the Besov-type  space $\vec \A(\T^3) \supset \vec{\mathcal{C}}^{-\frac 34}(\T^3)$.  See Subsection \ref{SEC:non} and \eqref{vecA} for the definition of the space $\vec \A(\T^3)$.

\end{itemize}
\end{theorem}

As explained above, the interesting behavior of the three-dimensional Gibbs measure \eqref{GibPARIS008} is that it shows a phase transition according to the strength 
$|\ld|$ of the coupling constant. This behavior does not appear in the one- and two-dimensional Gibbs measures \eqref{GibbsY} and \eqref{GibbsYY}. In particular, compared to the one- and two-dimensional cases, where the Gibbs measures \eqref{GibbsY} and \eqref{GibbsYY} are absolutely continuous with respect to the Gaussian free field for any $\ld \in \R\setminus\{0\}$ and any $K>0$, the mutual singularity of the Gibbs measure and the Gaussian free field is a new phenomenon. Furthermore, in our main theorem \ref{THM:Gibbs}, the phase transition occurs based on the size of the coupling constant $\ld$, regardless of the $L^2$-cutoff size $K>0$. For a study of the phase transition in terms of the $L^2$-cutoff size $K>0$, see \cite{OST1}, where they studied focusing $\Phi^6_1$-measure. 

In the weak coupling case $0<|\ld| \ll 1$, we address this problem by using the variational approach recently developed by Barashkov and Gubinelli \cite{BG} to prove ultraviolet stability for the $\Phi^4_3$-measure. In this variational formulation, we first isolate the divergent contributions in the interaction potential and then eliminate the ultraviolet divergences using the renormalization techniques. Additionally, there are extra terms that should be carefully controlled due to the focusing nature of the interaction potential under consideration. To prove the singularity of the Gibbs measure and the existence of shifted measures with respect to which the Gibbs measure is absolutely continuous, we use the approach introduced in \cite{BG, BG2, OOT1, OOT2}, exploiting the variational formula.

Next, we turn to the strong coupling case $|\ld | \gg 1$. As previously mentioned, the singularity of the Gibbs measure $\rhoo$ with respect to the Gaussian free field $\muu$ complicates the formulation of the non-normalizability statement in Theorem \ref{THM:Gibbs}(ii) because there is no density for the Gibbs measure. If we were in a situation where the truncated density $ e^{- \mathcal{H}_N^\dia  (u,w)}$ converges to the limiting density under the Gaussian free field, as in the one-and two-dimensional cases, it would be possible to define a $\s$-finite version of the Gibbs measure $\rhoo$ using the density as follows
\begin{align}
d\rhoo(u,w)=e^{- \H^\dia  (u,w)} \ind_{ \big\{|\int_{\T^3} \;: | u|^2 :\;   dx| \le K\big\}} d\muu(u,w).
\label{smeau}
\end{align}

\noi
Then, it suffices to prove 
\begin{align}
\sup_{N \in \N} \E_{\muu} \Big[ e^{- \H_N^\dia  (u,w)} \ind_{ \big\{|\int_{\T^3} \; : | u|^2 : \;   dx| \le K\big\}} \Big] = \infty, 
\label{podiv}
\end{align}

\noi
because \eqref{podiv} implies that there is no normalization constant that can make the limiting measure \eqref{smeau} into a probability measure. In the current situation, however, the interaction potential $\H_N^\dia(u,w)$ in \eqref{addcount} and the corresponding density $e^{-\H_N^\dia(u,w)}$ do not converge to any limit under the Gaussian free field $\muu$. Hence, even though we prove a statement of the form \eqref{podiv}, we can still choose a sequence of constants (partition functions) $Z_N$ such that the measures $Z_N^{-1} e^{-\H_N^\dia(u,w)} \ind_{ \{|\int_{\T^3} : | u_N|^2 :   dx| \le K\}}d\muu$ have a weak limit as $N\to \infty$. This corresponds to the case in Theorem \ref{THM:Gibbs}\,(i) when $0<|\ld|\ll 1$. The non-convergence of the truncated Gibbs measures stated in Theorem \ref{THM:Gibbs}\,(ii) shows that this can not occur for the Gibbs measure $\rhoo$ in \eqref{GibPARIS008} when the coupling effect is sufficiently strong $|\ld|\gg 1$. Due to the singularity of the Gibbs measure with respect to the base Gaussian field $\muu$, we must be careful in stating the non-normalizability result. For a more precise formulation, see the beginning of Section \ref{SEC:non}.

The proof of Theorem \ref{THM:Gibbs}\,(ii) is based on the recent works by Oh, Okamoto, Tolomeo \cite{OOT1,OOT2}, and Oh, Tolomeo, and the author \cite{OSeoT}. However, we note that in our setting, the argument introduced in \cite{OOT1, OOT2} only proves the non-construction of the singular Gibbs measure for large $K \gg 1$ because those works considered generalized grand-canonical Gibbs measures with a taming by $L^2$-norm (see \eqref{GibbsH}).  In particular, \cite{OSeoT} introduced a refined argument to prove the non-construction of focusing Gibbs measures with $L^2$-cutoff in a scenario where the truncated density $ e^{- \mathcal{H}_N^\dia  (u,w)}$ converges to a limiting object as $N\to \infty$ under the free field, that is, when the Gibbs measure is absolutely continuous with respect to the Gaussian free field. More precisely, in the variational formulation of the partition function, a profile that causes the partition function to blow up was constructed based on the Gaussian process, taking advantage of the situation where the Gibbs measure is absolutely continuous with respect to the Gaussian free field. However, in our case, the singularity of the Gibbs measure prevents us from exploiting the Gaussian structure. Consequently, the desired profile causing the blow-up is no longer an independent Gaussian process for each Fourier mode. This lack of explicit independence complicates the proof of the non-construction of the Gibbs measure. Therefore, instead of relying on the Gaussian structure, we construct a more general blowing-up profile to prove the non-construction of the singular Gibbs measure when $|\ld| \gg 1$. See also the explanation in Remark \ref{REM:diff0}.

\subsection{Motivation and comments on the literature}

\subsubsection{Focusing Gibbs measures}\label{SUBSUBSEC:foc}
In the seminal work \cite{LRS}, Lebowitz, Rose, and Speer started studying the one-dimensional focusing Gibbs measure with an $L^2$-cutoff, of the form
\begin{align}
d\rho(\phi) = Z^{-1}e^{\frac{\s}{p} \int_{\T } \phi^{p} dx  } \ind_{\{\int_{\T} \phi^2 \, dx\, \leq K\}}  d \mu(\phi)
\label{mic}
\end{align}

\noi 
where (i) $p \in 2\N+1$ with $\s \in \R \setminus\{0\}$ or (ii) $p \in 2\N+2$ with $\s>0$. For the (non)-construction of the focusing $\Phi^p$-measure on $\T$, see \cite{LRS, BO94, OST1}. In \cite{CFL}, Carlen, Fr\"ohlich, and Lebowitz also studied the construction of the generalized grand-canonical Gibbs measure given by
\begin{align}
d\rho(\phi) = Z^{-1}
e^{\frac{\s}{p} \int_{\T} \phi^p dx  -A\Big( \int_{\T} \phi^2 dx   \Big)^\g  }   d \mu(\phi)
\label{gra0}
\end{align}

\noi
for specific values of $\sigma$, $p$, $A$, and $\gamma$. Since then, Brydges and Slade \cite{BS} studied the non-construction of the focusing $\Phi^4$-measure in the two-dimensional setting $\T^2$. In \cite{BO97a}, Bourgain stated,  “It seems worthwhile to investigate this aspect [the (non-)normalizability issue of the focusing Gibbs measures] more as a continuation of [\,\cite{LRS, BS}\,]." Regarding other focusing Gibbs measures associated to nonlinear Hamiltonian PDEs, for example, see \cite{BO94b, TzBO, OOT1, OOT2, Bring, RSTW, Seong, RSoh1, RSoh2, SSoe}.

\subsubsection{Relation with the Hartree $\Phi^4_3$-measure} \label{Subsub:Hartree}
We investigate the relation of Gibbs measures between the Schr\"odinger-wave systems on $\T^3$
\begin{align}
\begin{cases}
i \dt u +\Dl u = \ld uw\\
c^{-2}\dt^2 w- \Dl w = \ld |u|^2
\label{Zak}
\end{cases}
\end{align}

\noi
and the focusing Hartree NLS with the Coulomb potential\footnote{It is of particular physical relevance.}
\begin{align}
i \dt u +\Dl u =  \ld^2 \big(V * |u|^2 \big)u
\label{Hart}
\end{align}

\noi
where $V(x)=\frac 1{|x|}$ is a convolution potential with $\ft V(n)=|n|^{-2}$. Notice that
by sending the wave speed $c$ in \eqref{Zak} to $\infty$, 
the Schr\"odinger-wave systems converge, at a formal level, to \eqref{Hart}.

In  \cite{BO97a}, Bourgain first constructed the 
focusing Gibbs measure with a Hartree-type interaction, known as the Hartree $\Phi^4_3$-measure for \eqref{Hart}, endowed with a Wick-ordered $L^2$-cutoff
\begin{align*}
d\rho(u) = Z^{-1} e^{\frac {\ld^2}4 \int_{\T^3} (V*:|u|^2:)\,  :|u|^2 :\, dx}\ind_{\{\int_{\T^3} :\,|u|^2: \, dx\leq K\}} d\mu(u),
\end{align*}

\noi
where $\ft V(n) = |n|^{-\be} $ is  the potential of order $\be>2$. In \cite{OOT1}, they continued the study of the focusing Hartree $\Phi^4_3$-measure with the Coulomb potential (i.e.~$\be=2$) in the generalized grand-canonical formulation:
\begin{align}
d\rho(u) = Z^{-1}
 \exp \bigg\{ \frac {\ld^2}4 \int_{\T^3} (V \ast :\! u^2 \!:) :\! u^2 \!: \, dx - A
\bigg|\int_{\T^3} :\! u^2 \! : \, dx\bigg|^3\bigg\} d \mu(u)
\label{GibbsH}
\end{align}

\noi
and  established a phase transition: when $\be = 2$ (i.e.~Coulomb potential), 
the  focusing Hartree $\Phi^4_3$-measure is constructible for $0 < \ld^2 \ll 1$ and $A=A(\ld)\gg 1$, while it is not for $\ld^2 \gg 1$.
The result shows the critical nature of \eqref{Hart} like \eqref{Zak}. Notice that the focusing Hartree $\Phi^4_3$-model with $\be=2$ and the Gibbs measure \eqref{GibbsW} for \eqref{Zak} share common features. They are both critical in terms of the phase transition, depending on  the size of  the coupling constant~$\ld$. We, however, point out that there is a crucial difference. While the focusing Hartree $\Phi^4_3$-measure  with Coulomb potential is absolutely continuous with  respect to the base Gaussian free field $\mu$, the Gibbs measure $\rhoo$ \eqref{GibbsW} for \eqref{Zak} is singular with respect to the base Gaussian free field $\muu$.
As we already pointed out, this singularity of the Gibbs measure $\rhoo$ makes an additional difficulty in proving the non-normalizability, which requires a significant effort.

\subsubsection{Dynamical problems on the ensemble} 
In spatial dimensions $d = 1, 2$ (i.e.~\eqref{GibbsY} and \eqref{GibbsYY}), the Gibbs measure $\rhoo$ is absolutely continuous with respect to the Gaussian field $\muu$, and hence the samples of the Gibbs measure behave like a random Fourier series with independent coefficients as in \eqref{rFou}. The independence of the Fourier coefficients, and the simple structure of \eqref{rFou} are  essential ingredients in studying the dynamics on the statistical ensemble (see \cite{BO94b, Seong}). Unfortunately, the Gibbs measure $\rhoo$ on $\T^3$ is singular with respect to the Gaussian field $\muu$ and so the free field is no longer a reference measure. Note that in Theorem \ref{THM:Gibbs}, we constructed a reference measure $\nuu$ for the Gibbs measure $\rhoo$, which serves a similar purpose as the Gaussian field for the Gibbs measure on $\T$ or $\T^2$. The samples of $\nuu$  are given by explicit Gaussian chaos.  
However, we note that the samples of $\nuu$  have dependent Fourier coefficients, which presents a challenge for studying the well-posedness problem with Gibbsian data. In the weak coupling regime, where the Gibbs measure can be constructed, we plan to study the well-posedness problem on the statistical ensemble $\rhoo$ (i.e.~singular Gibbs measure $\rhoo$ on $\T^3$) and invariance of the measure under the flow. This will certainly require a sophisticated approach, such as the paracontrolled approach \cite{GIP, GKO, BR2, BDNY} in a hyperbolic setting with random tensor estimates \cite{DNY3, BDNY}. In particular, since we are considering the system \eqref{SW1} (not a single equation), we expect significant effort will be needed due to the complicated situations arising from Schrödinger-wave interactions (i.e., different phases).

\subsection{Organization of the paper} 
In Section \ref{SEC:Not}, we introduce some notations and preliminary lemmas for 
both deterministic and probabilistic techniques. In Section \ref{SEC:Construc}, we first present the stochastic variational formulation of the partition function and isolate and kill divergences beyond Wick renormalization. Next, we prove the uniform exponential integrability and tightness of the truncated Gibbs measures  (Subsection \ref{SUBSEC:uniform}). Once we construct the Gibbs measure for the weakly coupled case (i.e.~$0<|\ld|\ll 1$), we prove the singularity of the Gibbs measure with respect to the Gaussian field (Subsection \ref{SUBSEC:singular}) and construct a shifted measure with respect to which the Gibbs measure $\rhoo$ is absolutely continuous (Subsection \ref{SUBSEC:AC}). In Section \ref{SEC:non}, we prove the non-normalizability of the Gibbs measure in the strongly coupled case (i.e.~$|\ld|\gg 1$). We first construct a reference measure (Subsection \ref{SUBSEC:refm}) and then construct a $\s$-finite version of Gibbs measures with respect to the reference measure. After that, we prove the non-normalizability of the $\s$-finite measure (Subsection \ref{SUBSEC:Non}), and thus the non-convergence of the truncated Gibbs measures $\{\rhoo_N\}$ comes as a corollary (Subsection \ref{SUBSEC:nonconv}).

\section{Notations and basic lemmas}
\label{SEC:Not}

When addressing regularities of functions and distributions, we  use $\eps > 0$ to denote a small constant. We usually  suppress the dependence on such $\eta > 0$ in estimates. For $a, b > 0$, $a\lesssim b$  means that
there exists $C>0$ such that $a \leq Cb$. By $a\sim b$, we mean that $a\lesssim b$ and $b \lesssim a$. Regarding space-time functions, we use the following short-hand notation $L^q_TL^r_x$ = $L^q([0, T]; L^r(\T^2))$, etc.

\subsection{Function spaces}
\label{SUBSEC:21}

Let $s \in \R$ and $1 \leq p \leq \infty$.
We define the $L^2$-based Sobolev space $H^s(\T^d)$
as follows:
\begin{align*}
\| f \|_{H^s} = \| \jb{n}^s \ft f (n) \|_{\l^2_n}.
\end{align*}

\noi
We also define the $L^p$-based Sobolev space $W^{s, p}(\T^d)$
by the norm:
\begin{align*}
\| f \|_{W^{s, p}} = \big\| \F^{-1} [\jb{n}^s \ft f(n)] \big\|_{L^p}.
\end{align*}

\noi
When $p = 2$, we have $H^s(\T^d) = W^{s, 2}(\T^d)$. We also set $\vec H(\T^3):=H(\T^3)\times H(\T^3)$ and $\vec W^{s, p}(\T^3)=W^{s, p}(\T^3) \times W^{s, p}(\T^3)$.

Let $\phi:\R \to [0, 1]$ be a smooth  bump function supported on $[-\frac{8}{5}, \frac{8}{5}]$ 
and $\phi\equiv 1$ on $\big[-\frac 54, \frac 54\big]$.
For $\xi \in \R^d$, we set $\varphi_0(\xi) = \phi(|\xi|)$
and 
\begin{align}
\varphi_{j}(\xi) = \phi\big(\tfrac{|\xi|}{2^j}\big)-\phi\big(\tfrac{|\xi|}{2^{j-1}}\big)
\label{phi1}
\end{align}

\noi
for $j \in \N$.
Then, for $j \in \Z_{\geq 0} := \N \cup\{0\}$, 
we define  the Littlewood-Paley projector  $\P_j$ 
as the Fourier multiplier operator with a symbol $\varphi_j$.
Note that we have 
\begin{align*}
\sum_{j = 0}^\infty \varphi_j (\xi) = 1
\end{align*}

\noi
 for each $\xi \in \R^d$. 
Thus, 
we have
\begin{align*}
f = \sum_{j = 0}^\infty \P_j f.
\end{align*}

\noi

We next recall the basic properties of the Besov spaces $B^s_{p, q}(\T^d)$
defined by the norm:
\begin{align*}
\| u \|_{B^s_{p,q}} = \Big\| 2^{s j} \| \P_{j} u \|_{L^p_x} \Big\|_{\l^q_j(\Z_{\geq 0})}.
\end{align*}

\noi
We denote the H\"older-Besov space by  $\mathcal{C}^s (\T^d)= B^s_{\infty,\infty}(\T^d)$ and $\vec{\mathcal{C}}^s (\T^d):=\mathcal{C}^s (\T^d)\times \mathcal{C}^s (\T^d)$.
Note that  the parameter $s$ measures differentiability and $p$ measures integrability. In particular, $H^s (\T^d) = B^s_{2,2}(\T^d)$
and for $s > 0$ and not an integer, $\mathcal{C}^s(\T^d)$ coincides with the classical H\"older spaces $C^s(\T^d)$; see  \cite{Graf}.

We recall the following basic estimates in Besov spaces, see \cite{BCD}, for example.

\begin{lemma}\label{LEM:KCKON0}
The following estimates hold.

\noi
\textup{(i) (interpolation)} 
Let $s, s_1, s_2 \in \R$ and $p, p_1, p_2 \in (1,\infty)$
such that $s = \ta s_1 + (1-\ta) s_2$ and $\frac 1p = \frac \ta{p_1} + \frac{1-\ta}{p_2}$
for some $0< \ta < 1$.
Then, we have
\begin{equation}
\| u \|_{W^{s,  p}} \les \| u \|_{W^{s_1, p_1}}^\ta \| u \|_{W^{s_2, p_2}}^{1-\ta}.
\label{INT0P}
\end{equation}

\noi
\textup{(ii) (embeddings)}
Let $s_1, s_2 \in \R$ and $p_1, p_2, q_1, q_2 \in [1,\infty]$.
Then, we have
\begin{align} 
\begin{split}
\| u \|_{B^{s_1}_{p_1,q_1}} 
&\les \| u \|_{B^{s_2}_{p_2, q_2}} 
\qquad \text{for $s_1 \leq s_2$, $p_1 \leq p_2$,  and $q_1 \geq q_2$},  \\
\| u \|_{B^{s_1}_{p_1,q_1}} 
&\les \| u \|_{B^{s_2}_{p_1, \infty}}
\qquad \text{for $s_1 < s_2$},\\
\| u \|_{B^0_{p_1, \infty}}
 &  \les  \| u \|_{L^{p_1}}
 \les \| u \|_{B^0_{p_1, 1}}.
\end{split}
\label{embed}
\end{align}

%


\smallskip

\noi
\textup{(iii) (Besov embedding)}
Let $1\leq p_2 \leq p_1 \leq \infty$, $q \in [1,\infty]$,  and  $s_2 \ge s_1 + d \big(\frac{1}{p_2} - \frac{1}{p_1}\big)$. Then, we have
\begin{equation*}
 \| u \|_{B^{s_1}_{p_1,q}} \les \| u \|_{B^{s_2}_{p_2,q}}.
\end{equation*}

\smallskip

\noi
\textup{(iv) (duality)}
Let $s \in \mathbb{R}$
and  $p, p', q, q' \in [1,\infty]$ such that $\frac1p + \frac1{p'} = \frac1q + \frac1{q'} = 1$. Then, we have
\begin{align}
\bigg| \int_{\T^d}  uv \, dx \bigg|
\le \| u \|_{B^{s}_{p,q}} \| v \|_{B^{-s}_{p',q'}},
\label{dual}
\end{align}

\noi
where $\int_{\T^d} u v \, dx$ denotes  the duality pairing between $B^{s}_{p,q}(\T^d)$ and $B^{-s}_{p',q'}(\T^d)$.

\smallskip
	
\noi		
\textup{(v) (fractional Leibniz rule)} 
Let $p, p_1, p_2, p_3, p_4 \in [1,\infty]$ such that 
$\frac1{p_1} + \frac1{p_2} 
= \frac1{p_3} + \frac1{p_4} = \frac 1p$. 
Then, for every $s>0$, we have
\begin{equation}
\| uv \|_{B^{s}_{p,q}} \les  \| u \|_{B^{s}_{p_1,q}}\| v \|_{L^{p_2}} + \| u \|_{L^{p_3}} \| v \|_{B^s_{p_4,q}} .
\label{prod}
\end{equation}

\end{lemma}

We also recall the following product estimate from \cite{GKO}.

\begin{lemma}\label{LEM:KKS0}
Let $0\leq s\leq 1$.

\smallskip

\noi
\textup{(i)}
Let  $1<p_j,q_j,r<\infty$, $j=1,2$ such that $\frac{1}{r}=\frac{1}{p_j}+\frac{1}{q_j}$.
Then, we have 
\begin{align*}
\|\jb{\nb}^s(fg)\|_{L^r(\T^3)}\lesssim\| \jb{\nb}^s f\|_{L^{p_1}(\T^3)} \|g\|_{L^{q_1}(\T^3)}+ \|f\|_{L^{p_2}(\T^3)} 
\|  \jb{\nb}^s g\|_{L^{q_2}(\T^3)}.
\end{align*}

\smallskip

\noi
\textup{(ii)}
Let   $1<p,q,r<\infty$ such that $s \geq   3\big(\frac{1}{p}+\frac{1}{q}-\frac{1}{r}\big)$.
Then, we have
\begin{align*}
\|\jb{\nb}^{-s}(fg)\|_{L^r(\T^3)}
\les\| \jb{\nb}^{-s} f\|_{L^{p}(\T^3)} \| \jb{\nb}^{s} g\|_{L^{q}(\T^3)} .
\end{align*}

\end{lemma}

\subsection{On discrete convolutions}

We next recall the following basic lemma on a discrete convolution.

\begin{lemma}\label{LEM:COKSS0}
Let $d \geq 1$. If $\al, \be \in \R$ satisfy
\begin{align*}
\al+ \be > d  \qquad \text{and}\qquad  \al < d, 
\end{align*}

\noi 
then we have
\begin{align*}
\sum_{n = n_1 + n_2} \frac{1}{\jb{n_1}^\al \jb{n_2}^\be}
\les \jb{n}^{- \al + \ld}
\end{align*}

\noi
for any $n \in \Z^d$, where $\ld = 
\max( d- \be, 0)$ when $\be \ne d$ and $\ld = \eps$ when $\be = d$ for any $\eps > 0$.

\end{lemma}

For the proof of Lemma \ref{LEM:COKSS0}, see, for example,  \cite[Lemma 4.2]{GTV} and \cite[Lemma 4.1]{MWX}.

\subsection{Tools from stochastic analysis}

We conclude this section by recalling some lemmas
from stochastic analysis.
See \cite{Bog, Shige} for basic definitions.
Let $(H, B, \mu)$ be an abstract Wiener space, that is, $\mu$ is a Gaussian measure on a separable Banach space $B$, and $H \subset B$ is its Cameron-Martin space.
Given  a complete orthonormal system $\{e_j \}_{ j \in \N} \subset B^*$ of $H^* = H$, 
we  define a polynomial chaos of order
$k$ to be an element of the form $\prod_{j = 1}^\infty H_{k_j}(\jb{x, e_j})$, 
where $x \in B$, $k_j \ne 0$ for only finitely many $j$'s, $k= \sum_{j = 1}^\infty k_j$, 
$H_{k_j}$ is the Hermite polynomial of degree $k_j$, 
and $\jb{\cdot, \cdot} = \vphantom{|}_B \jb{\cdot, \cdot}_{B^*}$ denotes the $B$--$B^*$ duality pairing.
We then 
denote the closure  of 
polynomial chaoses of order $k$ 
under $L^2(B, \mu)$ by $\mathcal{H}_k$.
The element in $\H_k$ 
is called homogeneous  Wiener chaos of order $k$.
We also set
\[ \H_{\leq k} = \bigoplus_{j = 0}^k \H_j\]

\noi
 for $k \in \N$.

Let $L = \Dl -x \cdot \nabla$ be 
 the Ornstein-Uhlenbeck operator. Then, 
it is known that 
any element in $\mathcal H_k$ 
is an eigenfunction of $L$ with eigenvalue $-k$.
Then, as a consequence
of the  hypercontractivity of the Ornstein-Uhlenbeck
semigroup $U(t) = e^{tL}$ due to Nelson \cite{Nelson2}, 
we have the following Wiener chaos estimate
\cite[Theorem~I.22]{Simon}.

\begin{lemma}\label{LEM:hyp}
Let $k \in \N$.
Then, we have
\begin{equation*}
\|X \|_{L^p(\O)} \leq (p-1)^\frac{k}{2} \|X\|_{L^2(\O)}
 \end{equation*}
 
\noi
for any $p \geq 2$ and any $X \in \H_{\leq k}$.

\end{lemma}

We recall the following orthogonality
relation for the Hermite polynomials. 
See  \cite[Lemma 1.1.1]{Nua}.

\begin{lemma}\label{LEM:Wick2}
Let $f$ and $g$ be jointly Gaussian random variables with mean zero 
and variances $\s_f$
and $\s_g$.
Then, we have 
\begin{align*}
\E\big[ H_k(f; \s_f) H_\l(g; \s_g)\big] = \dl_{k\l} k! \big\{\E[ f g] \big\}^k, 
\end{align*}

\noi
where $H_k (x,\s)$ denotes the Hermite polynomial of degree $k$ with variance parameter $\s$.

\end{lemma}

We recall the following Wick's theorem. See Proposition I.2 in \cite{Simon}.
\begin{lemma}\label{LEM:Wick}	
Let $g_1, \dots, g_{2n}$ be \textup{(}not necessarily distinct\textup{)} jointly Gaussian random variables.
Then, we have
\[ \E\big[ g_1 \cdots g_{2n}\big]
= \sum  \prod_{k = 1}^n \E\big[g_{i_k} g_{j_k} \big], 
\]

\noi
where the sum is over all partitions of $\{1, \dots, 2 n\}$
into disjoint pairs $(i_k, j_k)$.
\end{lemma}

Lastly, we recall the Prokhorov theorem from \cite{Bill99}.

\begin{lemma}[Prokhorov theorem]
\label{LEM:Pro}
If a sequence of probability measures 
on a metric space $\M$ is tight, then
it is relatively compact. 
If in addition, $\M$ is separable and complete, then relative compactness is 
equivalent to tightness. 
\end{lemma}

\section{Construction of Gibbs measures in the weak coupling regime}
\label{SEC:Construc}

In this section, we construct the singular Gibbs measure in the weak coupling regime $0<|\ld|\ll 1$. Subsequently, we show that the resulting Gibbs measure $\rhoo$ and the Gaussian field $\muu$ are mutually singular, and establish the reference measure $\nuu$ with respect to which the Gibbs measure $\rhoo$ is absolutely continuous

\subsection{Bou\'e-Dupuis variational formalism for the Gibbs measure}
\label{SUBSEC:var}

In this subsection, we introduce the main framework for analyzing expectations under the Gaussian field $\muu$. The framework was introduced by Barashkov and Gubinelli \cite{BG} to control the ultraviolet stability of the $\Phi^4_3$-measure.

Let $\vec W(t)$ be the cylindrical Wiener process on $\vec L^2(\T^3)$ with respect to the underlying probability measure $\PP$, that is,
\begin{align*}
\vec{W}(t) = (W_S(t),W_W(t))= \bigg(\sum_{n \in \Z^3} B_S^n(t) e^{in\cdot x}, \sum_{n \in \Z^3} B_W^n(t) e^{in\cdot x} \bigg),
\end{align*}

\noi
where
$\{ B_S^n \}_{n \in \Z^3}$ and $\{ B_W^n \}_{n \in \Z^3}$  are defined by  $B_j^n(t) = \jb{\xi_j, \ind_{[0, t]} \cdot e^{in\cdot x}}_{ x, t}$ for $j\in \{S,W\}$. Here, $\jb{\cdot, \cdot}_{x, t}$ denotes 
the duality pairing on $\T^3\times \R$, and $\xi_S$ and $\xi_W$ are independent (Gaussian) space-time white noise on $\T^3\times \R_{+}$. We  then define a centered Gaussian process $\vec {\<1>}(t)=(\<1>_S(t), \<1>_W(t) ) $ by 
\begin{align}
\vec {\<1>}(t)=(\<1>_S(t), \<1>_W(t) )= 
&= \jb{\nabla}^{-1}\vec W(t)=\bigg(\sum_{n \in \Z^3} \frac{B_S^n(t)}{\jb{n}} e^{in\cdot x}, \sum_{n \in \Z^3} \frac{B_W^n(t)}{\jb{n}} e^{in\cdot x} \bigg).
\label{CenGauss}
\end{align}

\noi
Then, we get $\Law (\vec {\<1>}(1)) = \muu$.  By setting  $\vec{\<1>}_N(t) = \pi_N\vec {\<1>}(t) $, 
we have   $\Law (\vec{\<1>}_N(1)) = (\pi_N)_\#\muu$. In particular, we observe 
\begin{align*}
\E \Big[|\vec{\<1>}_N(t) |^2\Big] = \<tadpole>_{S,N}(t)+\<tadpole>_{W,N}(t)=2\sum_{|n|\le N}\frac t{\jb{n}^2}\sim tN
\end{align*}

\noi
as $N\to \infty$, where we recall from \eqref{tadpole1} and \eqref{tadpole} that $\<tadpole>_{S,N}(1)= \E_\mu \big[u_N(x)^2\big]$ and $\<tadpole>_{W,N}(1)= \E_\mu \big[w_N(x)^2\big]$. The second and third Wick powers of $\<1>_{S,N}$ and $\<1>_{W,N}$ are the space-stationary stochastic processes $\<2>_S(t), \<3>_{S,W}(t)$ defined by 
\begin{align}
\<2>_{S,N}(t)
&=\<1>_{S,N}^2(t) - \<tadpole>_{S,N}(t), \label{cher18}\\
\<3>_{S,W,N}(t)
&=\<1>_{S,N}^2(t) \<1>_{W,N}(t) -\<tadpole>_{S,N}(t) \<1>_{W,N}(t)\label{cher19}.
\end{align}

\noi
In particular, $\<2>_{S,N}(1)$ and $\<3>_{S,W,N}(1)$ are equal in law to $:\! u_N^2 \!: \,$ and $:\! u_N^2 w_N \!: \, $ in \eqref{Wick01} and \eqref{Wick02}.

Next, let $\vec {\mathbb{H}}_a $ denote the space of drifts, which are the progressively measurable processes belonging to
$L^2([0,1]; \vec L^2(\T^3))$, $\PP$-almost surely. 
For later use, we also define
$\vec {\mathbb{H}}_a^1$
to be  the space of drifts, which are the progressively measurable processes 
 belonging to
$L^2([0,1]; \vec H^1(\T^3))$, $\PP$-almost surely. 
In other words, we have
\begin{align}
\vec {\mathbb{H}}_a^1 = \jb{\nb}^{-1} \vec {\mathbb{H}}_a 
\label{Ha}
\end{align}

\noi
We next recall the  Bou\'e-Dupuis variational formula \cite{BD, Ust}; in particular, see Theorem 7 in~\cite{Ust}. See also Theorem 2 in \cite{BG}.

\begin{lemma}\label{LEM:BoueDupu}
Let $\vec{\<1>}$ be as in \eqref{CenGauss}.
Fix $N \in \N$. Suppose that  $F:\vec C^\infty(\T^3)  \to \R$
is measurable such that $\E\big[|F(\pi_N \vec{\<1>}(1))|^p\big] < \infty$
and $\E\big[|e^{-F(\pi_N \vec{\<1>}(1))}|^q \big] < \infty$ for some $1 < p, q < \infty$ with $\frac 1p + \frac 1q = 1$.
Then, we have
\begin{align*}
- \log \E\Big[e^{-F(\pi_N \vec{\<1>}(1))}\Big]
= \inf_{\vec{\dr} \in \vec{\mathbb{H}}_{a} }
\E\bigg[ F(\pi_N \vec{\<1>}(1) + \pi_N \vec{I}(\vec \dr)(1)) + \frac{1}{2} \int_0^1 \| \vec{\dr}(t) \|_{\vec L^2_x }^2 dt \bigg], 
\end{align*}

\noi
where  $\vec{I}(\vec{\dr})=(I(\dr_S),I(\dr_W) )$ is  defined by 
\begin{equation}
\vec{I}(\vec{\dr})(t) =  \big(I(\dr_S)(t), I(\dr_W)(t) \big)   =\bigg( \int_0^t \jb{\nb}^{-1} \dr_S(t') dt', \int_0^t \jb{\nb}^{-1} \dr_W(t') dt' \bigg).
\label{IdrSW}
\end{equation}

\end{lemma}

We next study a lemma on pathwise regularity estimates  of $(\<1>_{S}, \<1>_{W}, \<2>_{S}, \<2>_{W}, \<1>_{S}\<1>_{W}, \<3>_{S,W} )$ and $\vec {I}(\vec \dr)(1)$. For the convenience of the presentation, we define $\Xi(t):=(\<1>_{S}(t), \<1>_{W}(t), \<2>_{S}(t), \<2>_{W}(t), \<1>_{S}\<1>_{W}(t))$.

\begin{lemma}  \label{LEM:Cor00}

\textup{(i)} 
For any finite $p \ge 2$, $t\in [0,1]$, and $\eps>0$,
$\pi_N\Xi(t)$ converges to $\Xi(t)$ in $L^p(\O; \vec{\mathcal{C}}^{-\frac 12-\eps}(\T^3)\times \vec{\mathcal{C}}^{-1-\eps}(\T^3))$ as $N\to \infty$ 
and also almost surely in $\vec{\mathcal{C}}^{-\frac 12-\eps}(\T^3)\times \vec{\mathcal{C}}^{-1-\eps}(\T^3)$.
Moreover, we have 
\begin{equation}
\begin{split}
\E \Big[& \|\<1>_{S,N}(t)\|_{\mathcal{C}^{-\frac 12 - \eps}}^p+  \| \<1>_{W,N}(t) \|_{\mathcal{C}^{-\frac 12 - \eps}}^p+\|\<2>_{S,N}(t)\|_{\mathcal{C}^{-1 - \eps}}^p+ \| (\<1>_{S,N}\<1>_{W,N})(t)   \|_{\mathcal{C}^{-1 - \eps}}^p
\Big] \les  p   <\infty, 
\end{split}
\label{QS0}
\end{equation}

\noi
uniformly in $N \in \N \cup\{\infty\}$\footnote{When $N=\infty$, it can be understood as the identity operator.} and $t \in [0, 1]$.
We also have 
\begin{align}
\E 
\Big[ \|\<2>_{S,N}(t) \|_{H^{-1}}^2
\Big]
\sim  t^2 \log N, \label{logtdiv0}\\
\E 
\Big[ \| (\<1>_{S,N}\<1>_{W,N})(t)   \|_{H^{-1}}^2
\Big]
\sim  t^2 \log N
\label{logtdiv}
\end{align}

\noi
for any   $t \in [0, 1]$.

\smallskip

\noi
\textup{(ii)} 
For any $N \in \N$, we have 
\begin{align}
\E \bigg[ \int_{\T^3} \<3>_{S,W,N}(1) dx \bigg]
&= 0, \label{mean0}\\
\E \Bigg[ \bigg|\int_{\T^3} \<3>_{S,W,N}(1) dx \bigg|^2 \Bigg]&\sim \log N. \label{logNN0}
\end{align}

\smallskip

\noi
\textup{(iii)}
The drift term $\vec \dr\in \vec{\mathbb{H}}_a $ has the regularity
of the Cameron-Martin space, that is, for any $\vec \dr=(\dr_S, \dr_W) \in \vec{\mathbb{H}}_a $, we have
\begin{align}
\| \vec I(\vec \dr)(1) \|_{\vec H^{1}}^2 \leq \int_0^1 \| \vec \dr(t) \|_{\vec L^2}^2dt.
\label{CCS5}
\end{align}
\end{lemma}

\begin{remark}\rm
The estimate \eqref{logNN0} shows that $\<3>_{S,W,N}(1)$ diverges in variance as $N \to \infty$, even with Wick renormalization. Consequently, as presented in \eqref{furth01}, the interaction potential \eqref{wickpoten} fails to converge.
\end{remark}

\begin{proof}
The estimate \eqref{QS0} follows
from the Wiener chaos estimate (Lemma \ref{LEM:hyp}). See, for example,  \cite{GKO}.

We now prove \eqref{logtdiv0} and \eqref{logtdiv}. By recalling the definition of $\<2>_{S,N}(t)=H_2(\<1>_{S,N}(t); \<tadpole>_{S,N}(t) )$, where $H_k (x,\s)$ denotes the Hermite polynomial of degree $k$ with variance parameter $\s$, and proceeding with Lemma \ref{LEM:Wick2}, we have
\begin{align}
\begin{split}
\E & \Big[ \|\<2>_{S,N}(t)\|_{H^{-1}}^2
\Big]\\
& = \sum_{n \in \Z^3}\frac{1}{\jb{n}^2} \int_{\T^3_x \times \T^3_y}
 \E\Big[  H_2(\<1>_{S,N}(x, t); \<tadpole>_{S,N}(t)) H_2(\<1>_{S,N}(y, t); \<tadpole>_{S,N}(t))\Big] e_n(y - x)dx dy\\
& =\sum_{n \in \Z^3}\frac{t^2}{\jb{n}^2}
\sum_{\substack{n_1, n_2\in \Z^3\\ |n_1|, |n_2| \le N }}
\frac{1}{\jb{n_1}^2\jb{n_2}^2}
\int_{\T^3_x \times \T^3_y}
e_{n_1 + n_2 - n}(x-y) dx dy\\
& =\sum_{n \in \Z^3}\frac{t^2}{\jb{n}^2}
\sum_{\substack{n = n_1 + n_2\\ |n_1|, |n_2| \le N }}
\frac{1}{\jb{n_1}^2\jb{n_2}^2}, 
\end{split}
\label{ZU3}
\end{align}

\noi
where $e_n(z):=e^{in\cdot z}$.
By applying Lemma~\ref{LEM:COKSS0} to~\eqref{ZU3}, the upper bound in \eqref{logtdiv0} is obtained.
To get the lower bound, we first consider the contribution from $|n| \le \frac{2}{3}N$
and $\frac 14 |n| \leq |n_1| \leq \frac 12 |n|$, which implies $|n_2| \sim |n|$ and $|n_j| \le N$, $j = 1, 2$).
Then, it follows from \eqref{ZU3} that we get 
\begin{align*}
\E \Big[ \|\<2>_{S,N}(t) \|_{H^{-1}}^2 \Big]
\ges  \sum_{\substack{n \in \Z^3\\|n| \le \frac 23 N}}
\frac{t^2}{\jb{n}^3} \sim t^2 \log N.
\end{align*}

\noi
This proves the lower bound in \eqref{logtdiv0}. We now turn to \eqref{logtdiv}.
By proceeding with Lemma \ref{LEM:Wick2}, we have
\begin{align}
\begin{split}
\E & \Big[ \|(\<1>_{S,N}\<1>_{W,N})(t)\|_{H^{-1}}^2
\Big]\\
& = \sum_{n \in \Z^3}\frac{1}{\jb{n}^2} \int_{\T^3_x \times \T^3_y}
 \E\Big[ (\<1>_{S,N}\<1>_{W,N})(x,t)  \overline{(\<1>_{S,N}\<1>_{W,N})}(y,t)   \Big] e_n(y - x)dx dy\\
& = \sum_{n \in \Z^3}\frac{1}{\jb{n}^2} \int_{\T^3_x \times \T^3_y}
 \E\Big[  \<1>_{S,N}(x,t) \cj {\<1>_{S,N}}(y,t)   \Big]  \E\Big[  \<1>_{W,N}(x)\cj{\<1>_{W,N}}(y,t)   \Big] e_n(y - x)dx dy\\
& =\sum_{n \in \Z^3}\frac{t^2}{\jb{n}^2}
\sum_{\substack{n_1, n_2\in \Z^3\\ |n_1|, |n_2|\le N}}
\frac{1}{\jb{n_1}^2\jb{n_2}^2}
\int_{\T^3_x \times \T^3_y}
e_{n_1 + n_2 - n}(x-y) dx dy\\
& =\sum_{n \in \Z^3}\frac{t^2}{\jb{n}^2}
\sum_{\substack{n = n_1 + n_2\\ |n_1|, |n_2|\le N}}
\frac{1}{\jb{n_1}^2\jb{n_2}^2}\sim t^2 \log N. 
\end{split}
\end{align}

\noi
In the final step to obtain $\sim t^2 \log N$, we repeat the calculation used to derive \eqref{logtdiv0}.

We turn to the proof of \eqref{mean0} and \eqref{logNN0}. By recalling the definition of $\<3>_{S,W,N}(t)=\<2>_{S,N}(t) \<1>_{W,N}(t)$ in \eqref{cher19}, we directly obtain \eqref{mean0} thanks to the independence between $\<2>_{S,N}(t) $  and $\<1>_{W,N}(t)$. Regarding \eqref{logNN0}, note that
\begin{align*}
\E \Bigg[ \bigg|\int_{\T^3} \<3>_{S,W,N}(1) dx \bigg|^2 \Bigg]\sim \<sunset>_N 
&= \sum_{\substack{n_1+n_2+n_3=0 \\ |n_j| \les N }} \jb{n_1}^{-2} \jb{n_2}^{-2} \jb{n_3}^{-2}
\\
&\sim
\sum_{\substack{|n_3|  \les N}} \jb{n_3}^{-2} \jb{n_3}^{-1}
 \\
&\sim \log N, 
\end{align*}

\noi
where we denote by $\<sunset>_N $ the usual ``sunset" diagram appearing in the perturbation theory for $\Phi^4_3$. See \cite[Theorem 1]{Feld74} for the notation.

Regarding \eqref{CCS5}, the estimate  follows from Minkowski’s and Cauchy-Schwarz’ inequalities. See also the proof of Lemma 4.7 in \cite{GOTW} .
\end{proof}

\subsection{killing divergences beyond Wick renormalization}
\label{SUBSEC:Killdiv}

In this subsection, we present an additional renormalization procedure beyond Wick renormalization. For clarity, we will use the following shorthand notations: $\vec {\<1>}_N(t) = \pi_N \vec{\<1>}(t)$ and $\vec \Dr_N(t) = \pi_N \vec \Dr(t)$ with $\vec{\<1>}_N = \vec {\<1>}_N(1)$ and $\vec \Dr_N = \vec \Dr_N(1)$. We also use $\vec{\<1>}=\vec {\<1>}(1)$ and $\vec \Dr =  \vec \Dr (1)$. 

We denote by $\H_N$ the interaction potential with Wick renormalization
\begin{align}
\H_N(u,w):= \frac \ld 2  \int_{\T^3}    :\!  |u_N|^2 w_N \! : \,dx,
\label{CORINT01}
\end{align}

\noi
where $\ld \in \R\setminus \{0\}$ is a coupling constant. In view of 
\begin{align}
 \ind_{\{|\,\cdot \,| \le K\}}(x) \le \exp\big( -  A |x|^\gamma\big) \exp\big(A K^\g\big)
\label{cut000}
\end{align}

\noi
for any $K>0$, $\g>0$, and $A>0$, we set 
\begin{align*}
Z_N &= \int e^{-\H_N(u,w)} \ind_{\{\int_{\T^3} \; :|u_N|^2: \;  dx \, \leq K\}} d\muu (u,w)\\
&\les_{A,K}  \int e^{-\H_N(u,w)- A\big|\int_{\T^3} :\,|u_N|^2: \, dx\big|^3} d\muu (u,w) :=\wt Z_N. 
\end{align*}

\noi
By the Bou\'e-Dupuis formula (Lemma \ref{LEM:BoueDupu}), the partition function $\wt Z_N$ can be written as follows
\begin{align}
\begin{split}
- \log \wt Z_N 
&= \inf_{ \vec {\dr} \in  \vec{\mathbb   H}_a }
\E\bigg[ \H_N (\vec{\<1>}_N+\vec \Dr_N ) +  A\bigg|\int_{\T^3} :\! ( \<1>_{S,N}+\Dr_{S,N})^2\!: \, dx\bigg|^3  \\
&\hphantom{XXXXXXXXXXXX}+\frac 12 \int_0^1 \| \vec \dr(t)      \|_{\vec L^2_x}^2  dt \bigg].
\end{split}
\label{Funct1}
\end{align}

\noi
where $\vec{\<1>}_N=(\<1>_{S,N}, \<1>_{W,N})$, $\vec \Dr_N=(\Dr_{S,N}, \Dr_{W,N} )$, and $\vec \dr(t)=(\dr_S(t), \dr_W(t) ) $. By expanding $\mathcal{H}_N ( \vec{\<1>}_N+\vec \Dr_N)$, we have
\begin{align}
\begin{split}
 \frac \ld2\int_{\T^3} :\! ( \<1>_{S,N}+\Dr_{S,N})^2 (\<1>_{W,N}+\Dr_{W,N}) \!: dx&= \frac \ld2 \int_{\T^3} \<3>_{S,W,N} dx+\frac {\ld}2 \int_{\T^3}  \<2>_{S,N}  \Dr_{W,N }dx\\
& \hphantom{X}+\ld\int_{\T^3}  \<1>_{S,N} \<1>_{W,N}   \Dr_{S,N} dx 
+  \ld \int_{\T^3}  \<1>_{S,N} \Dr_{S,N} \Dr_{W,N} dx\\
& \hphantom{X}+ \frac \ld2\int_{\T^3}  \<1>_{W,N} \Dr_{S,N}^2 dx + \frac\ld2 \int_{\T^3} \Dr_{S,N}^2 \Dr_{W,N} dx.
\end{split}
\label{expa0}
\end{align}

\noi
As we will check below, the second term and third term on the right-hand side of \eqref{expa0}
\begin{align}
\frac \ld2 \int_{\T^3}  \underbrace{\<2>_{S,N}}_{\in C^{-1-}} \underbrace{\Dr_{W,N } }_{\in H^1}dx \quad \text{and} \quad \ld \int_{\T^3}  \underbrace{\<1>_{S,N} \<1>_{W,N} }_{\in \mathcal{C}^{-1-}}  \underbrace{ \Dr_{S,N} }_{\in H^{1}}dx 
\label{notbeha}
\end{align}

\noi
turn out to be divergent. On the other hand, the first term on the right-hand side of \eqref{expa0} vanishes under expectation due to \eqref{mean0}, while we can estimate the fourth, fifth, and sixth terms on the right-hand side of \eqref{expa0}. See Lemma~\ref{LEM:Cor0}. The additional divergences are included in the integrals $\frac \ld2 \int_{\T^3}\<2>_{S,N} \Dr_{W,N} dx  $ and $\ld \int_{\T^3} \int_{\T^3} \<1>_{S,N} \<1>_{W,N} \Dr_{S,N}  dx$ in \eqref{notbeha}. To isolate the divergences in the integrals, we use a change of variables in the drift entropies $\Dr_{S,N}$ and $\Dr_{W,N}$. Once the divergences are isolated, we introduce additional counterterms in \eqref{log0} and \eqref{density11} to eliminate the remaining divergences

Before introducing a change of variables, we first use the Itô product formula as follows
\begin{align}
\frac \ld2 \E\bigg[\int_{\T^3} \<2>_{S,N} \Dr_{W,N} dx \bigg] 
&= \frac \ld2 \E\bigg[ \int_0^1 \int_{\T^3} \<2>_{S,N}(t)  \dot \Dr_{W,N}(t) dx dt \bigg] \label{Divv0}\\
\ld \E\bigg[\int_{\T^3} \<1>_{S,N} \<1>_{W,N} \Dr_{S,N} dx \bigg] 
&= \ld \E\bigg[ \int_0^1 \int_{\T^3}  
\<1>_{S,N}(t) \<1>_{W,N}(t)  \dot \Dr_{S,N}(t) dx dt \bigg],  
\label{Divv1}
\end{align}
 
\noi
where $\dot \Theta_{S,N} (t) = \jb{\nabla}^{-1} \pi_N \theta_S(t)$ and $\dot \Theta_{W,N} (t) = \jb{\nabla}^{-1} \pi_N \theta_W(t)$  in view of \eqref{IdrSW}. Define $\ZZ^N_S$ and $\ZZ^N_W$ with $\ZZ^N_S(0) =0$ and $\ZZ^N_W(0) =0$  by its time derivatives as follows
\begin{align}
\dot \ZZ^N_S (t) &= \frac 12(1-\Delta)^{-1} \<2>_{S,N}(t) \label{Divv2}\\
\dot \ZZ^N_W (t) &= (1-\Delta)^{-1}  \<1>_{S,N}(t) \<1>_{W,N}(t) \label{Divv3}
\end{align}

\noi
and set $\ZZ_{S,N} = \pi_N \ZZ^N_S$ and $\ZZ_{W,N} = \pi_N \ZZ^N_W$. Then, we perform a change of variables 
\begin{align}
\dot \Upsilon^N_{S}(t) &= \dot \Dr_S(t)  + \ld \dot \ZZ_{W,N}(t) \label{Divv4}\\
\dot \Upsilon^N_{W}(t) &= \dot \Dr_W(t)  +  \ld \dot \ZZ_{S,N}(t) \label{Divv5}
\end{align}

\noi
and set $\Upsilon_{S,N} = \pi_N \Upsilon^N_{S}$ and $\Upsilon_{W,N} = \pi_N \Upsilon^N_{W}$.
From \eqref{Divv0}, \eqref{Divv1}, \eqref{Divv2}, \eqref{Divv3}, \eqref{Divv4}, and \eqref{Divv5}, we have 
\begin{align}
\E\bigg[ \frac \ld2 \int_{\T^3} \<2>_{S,N}  \Dr_{W,N} dx +\frac 12\int_0^1 \| \dr_{W}(t) \|^2_{L^2_x} dt \bigg]&=\frac 12 \E\bigg[  \int_0^1 \| \dot \Upsilon^N_{W} (t) \|_{H^1_x}^2 dt
\bigg]-\dl_{S,N, \ld}, \label{divv0}\\
\E\bigg[ \ld\int_{\T^3} \<1>_{S,N} \<1>_{W,N} \Dr_{S,N} dx +\frac 12\int_0^1 \| \dr_{S}(t) \|^2_{L^2_x} dt  \bigg]&=\frac 12 \E\bigg[  \int_0^1 \| \dot \Upsilon^N_{S} (t) \|_{H^1_x}^2 dt
\bigg]-\dl_{W,N, \ld} \label{divv1}
\end{align}

\noi
where
\begin{align*}
\dl_{S,N,\ld}:=\frac {\ld^2}2\E\bigg[ \int_0^1  \| \dot \ZZ_{S,N}(t) \|_{H^1_x}^2 dt     \bigg] \quad \text{and} \quad \dl_{W,N,\ld}:=\frac {\ld^2}2\E\bigg[ \int_0^1  \| \dot \ZZ_{W,N}(t) \|_{H^1_x}^2 dt     \bigg]. 
\end{align*}

\noi
Therefore, by combining \eqref{divv0} and \eqref{divv1}, we have
\begin{equation}
\begin{split}
&\E\bigg[ \frac \ld2 \int_{\T^3} \<2>_{S,N}  \Dr_{W,N} dx + \ld\int_{\T^3} \<1>_{S,N} \<1>_{W,N} \Dr_{S,N} dx 
+ \frac12 \int_0^1 \| \vec \dr(t)    \|_{\vec L^2_x}^2 dt \bigg]\\
&= \frac 12 \E\bigg[  \int_0^1 \| \dot{ \vec{\Upsilon}}^N(t)   \|_{\vec H^1_x}^2 dt
\bigg] - \dl_{N,\ld},
\end{split}
\label{expa1}
\end{equation}

\noi
where $\dot{ \vec{\Upsilon}}^N(t)=\big( \dot \Upsilon^N_{ S} (t), \dot \Upsilon^N_{W} (t) \big)$
and the divergent constant $\dl_{N,\ld}$ is given by 
\begin{align}
\dl_{N,\ld} :=\dl_{S,N,\ld}+\dl_{W,N,\ld}\sim  \ld^2\<sunset>_N \too \infty, 
\label{log0}
\end{align}

\noi
as $N \to \infty$, where $\<sunset>_N$ is the ``sunset" diagram in \eqref{sunset0}. The logarithmic divergence in \eqref{log0} can be easily checked from the spatial regularity $1- \eps$ of $\dot \ZZ_{S,N}(t)  = (1-\Delta)^{-1} \<2>_{S,N}(t)$ and $\dot \ZZ_{W,N}(t)=(1-\Delta)^{-1}  \<1>_{S,N}(t) \<1>_{W,N}(t)$.  See Lemma \ref{LEM:Cor00}.
In summary, we decompose the divergent integrals in \eqref{notbeha} into an integral with the new drift term and purely stochastic terms that correspond to the divergent contributions.

Based on the discussion above with $\mathcal{H}_N$ in \eqref{CORINT01}, we set $\mathcal{H}_N^\dia$ 
\begin{align} 
\begin{split}
\mathcal{H}_N^\dia (u,w):=\mathcal{H}_N(u,w) + \dl_{N,\ld},
\label{density11}
\end{split}
\end{align}
 
\noi
thereby eliminating the additional divergent contribution $\dl_{N,\ld}$ in \eqref{log0}.
Then, as in \eqref{truGibbsN}, we define the truncated and renormalized Gibbs measure $\rhoo_N$ by
\begin{align}
d\rhoo_N (u,w) = Z_N^{-1} e^{-\mathcal{H}_N^\dia(u,w)} \ind_{\{\int_{\T^3} \; :|u_N|^2: \;  dx \, \leq K\}}  d\muu(u,w),
\label{GibbsSN}
\end{align}

\noi
where the partition function $Z_N$ is given as follows
\begin{align}
Z_N &= \int e^{-\mathcal{H}_N^\dia(u,w)} \ind_{\{\int_{\T^3} \; :|u_N|^2: \;  dx \, \leq K\}}  d\muu(u,w) \notag \\
&\les_{A,K}  \int e^{-\H_N^\dia(u,w)- A\big|\int_{\T^3} :\,|u_N|^2: \, dx\big|^3} d\muu (u,w) :=\wt Z_N,
\label{Part81}
\end{align}

\noi
where in the last line \eqref{cut000} was used. From the Bou\'e-Dupuis variational formula (Lemma~\ref{LEM:BoueDupu}), we have
\begin{equation}
\begin{split}
- \log \wt  {Z}_N = \inf_{ \vec \dr \in \vec{\mathbb{H}}_a } \E
\bigg[ &\mathcal{H}_N^\dia (\vec{\<1>}_N+\vec \Dr_N )+ A\bigg|\int_{\T^3} :\! ( \<1>_{S,N}+\Dr_{S,N})^2\!: \, dx\bigg|^3\\
&\hphantom{XXXX} + \frac{1}{2} \int_0^1 \|\vec \dr(t) \|_{\vec L^2_x}^2 dt  \bigg]
\end{split}
\label{drift}
\end{equation}

\noi
for any $N \in \N$, where $\vec{\<1>}_N=(\<1>_{S,N}, \<1>_{W,N})$, $\vec \Dr_N=(\Dr_{S,N}, \Dr_{W,N} )$, and $\vec \dr(t)=(\dr_S(t), \dr_W(t) ) $. By defining
\begin{equation}
\begin{split}
\W_N (\vec \dr) = \E\bigg[ & \mathcal{H}_N^\dia (\vec{\<1>}_N+\vec \Dr_N )+ A\bigg|\int_{\T^3} :\! ( \<1>_{S,N}+\Dr_{S,N})^2\!: \, dx\bigg|^3\\
&\hphantom{XXXX} + \frac{1}{2} \int_0^1 \|\vec \dr(t) \|_{\vec L^2_x}^2 dt \bigg], 
\end{split}
\label{drift1}
\end{equation}

\noi
it follows from \eqref{density11}, \eqref{expa0}, \eqref{expa1}, and Lemma \ref{LEM:Cor00}\,(ii)
that 
\begin{align}
\begin{split}
\W_N (\vec \dr )
&=\E\bigg[ \ld \int_{\T^3}  \<1>_{S,N} \Dr_{S,N} \Dr_{W,N} dx+ \frac \ld2\int_{\T^3}  \<1>_{W,N} \Dr_{S,N}^2 dx +\frac \ld2\int_{\T^3} \Dr_{S,N}^2 \Dr_{W,N} dx\\
&\hphantom{XXXXXX} + A \bigg| \int_{\T^3} \Big( :\! \<1>_{S,N}^2 \!: + 2 \<1>_{S,N} \Dr_{S,N} + \Dr_{S,N}^2 \Big) dx \bigg|^3 \\
&\hphantom{XXXXXX} + \frac 12  \int_0^1 \| \dot{ \vec{\Upsilon}}^N(t)   \|_{\vec H^1_x}^2 dt
\bigg],
\end{split}
\label{drift11}
\end{align}

\noi
where $\dot{ \vec{\Upsilon}}^N(t)=\big( \dot \Upsilon^N_{ S} (t), \dot \Upsilon^N_{W} (t) \big)$. We also set 
\begin{align}
\Upsilon_{S,N} = \Upsilon_{S,N}(1)= \pi_N  \Upsilon^{N}_S(1) \qquad \text{and}& \qquad \ZZ_{S,N} = \ZZ_{S,N}(1) = \pi_N \ZZ^{N}_S(1)\\
\Upsilon_{W,N} = \Upsilon_{W,N}(1)= \pi_N  \Upsilon^N_W(1) \qquad \text{and}& \qquad \ZZ_{W,N} = \ZZ_{W,N}(1) = \pi_N \ZZ^N_W(1).
\label{K9a}
\end{align}

\noi
Thanks to the change of variables \eqref{Divv4} and \eqref{Divv5}, we have
\begin{align}
\Dr_{S,N} = \Upsilon_{S,N} - \ld  \ZZ_{W,N},
\label{chAa0}\\
\Dr_{W,N} = \Upsilon_{W,N} -  \ld  \ZZ_{S,N}.
\label{chAa1}
\end{align}

\noi
Note that from the definition of \eqref{Divv2} and \eqref{Divv3}, 
$\vec {\ZZ}_N=(\ZZ_{S,N}, \ZZ_{W,N})$ is determined by $\vec {\<1>}_N=(\<1>_{S,N}, \<1>_{W,N})$. Hence, we now view $\dot {\vec  \Upsilon}^N=(\dot \Upsilon^N_S, \dot \Upsilon^N_W)$ as a new drift entropy and study the minimization problem~\eqref{drift} with viewing $\W_N=\W_N(\dot{ \vec{\Upsilon}}^N)$ as a function of $\dot {\vec \Upsilon}^N$, and then taking an infimum in $\dot {\vec \Upsilon}^N \in \vec {\Ha^1}$, where $\vec {\Ha^1}$ is as in \eqref{Ha}. Therefore, in the following subsection, the main goal is to prove that 
$\W_N(\dot {\vec \Upsilon}^N)$ in \eqref{drift1} is bounded away from $-\infty$, 
uniformly in $N \in \N$ and  $\dot {\vec \Upsilon}^N \in \vec {\Ha^1}$.

\subsection{Uniform exponential integrability and tightness}
\label{SUBSEC:uniform}

In this subsection, we present the proof of the tightness of the truncated Gibbs measures $\{\rhoo_N\}$. By exploiting Prokhorov’s theorem (Theorem \ref{LEM:Pro}), this implies the weak convergence of a subsequence of $\{\rhoo_N\}$, allowing us to define the Gibbs measure $\rhoo$ as the weak limit.

Before proving the tightness of the truncated Gibbs measures $\{\rhoo_N\}$, we first establish the following uniform exponential integrability, which is an essential ingredient in proving the tightness.

\begin{lemma}\label{LEM:uniLp0}
There exists a small $\lambda_0 > 0$ such that for any $0 < |\lambda| < \lambda_0$ and $K > 0$, we have the uniform exponential integrability of the density
\begin{equation}
\sup_{N\in \N}Z_N = \sup_{N\in \N} \E_{\muu} \bigg[ e^{-\H_N^\dia(u,w)} \ind_{ \{|\int_{\T^3} : | u_N|^2 :   dx| \le K\}}   \bigg]
< \infty,
\label{uniform1}
\end{equation}

\noi
where $\H_N^\dia$ is in \eqref{density11}.
\end{lemma}

\begin{remark}\rm 
In the three-dimensional case, the uniform exponential $L^p(\muu)$-integrability of the density $e^{-\H_N^\dia(u,w)} \ind_{\{|\int_{\T^3} : |u_N|^2 : dx| \le K\}}$ holds only for $p = 1$. In the one- and two-dimensional cases, the uniform exponential $L^p(\muu)$-integrability of the density can be proven for any finite $p \ge 1$, which ensures that the truncated Gibbs measures $\{\rhoo_N\}$ converge to the limiting Gibbs measure $\rhoo$ in total variation as $N \to \infty$. See \cite{BO94b, Seong}. In the current situation, however, due to the additional counterterm $\delta_{N,\lambda}$ in \eqref{density11} beyond Wick renormalization, the convergence of the whole sequence $\{\rhoo_N\}$ in total variation does not hold. Instead, the tightness of the truncated Gibbs measures $\{\rhoo_N\}$ can be obtained from the uniform exponential $L^1(\mu)$-integrability \eqref{uniform1}.

\end{remark}

To prove Lemma \ref{LEM:uniLp0}, we first state two lemmas whose proofs are presented at the end of this subsection. While the first lemma is not difficult to prove, the second lemma (Lemma \ref{LEM:Cor1}) requires a much more thorough analysis due to the focusing nature of the Gibbs measure.

\begin{lemma}\label{LEM:Cor0}
Let $A>0$ and  $0<|\ld|<1$.
Then, there exist small $\eps>0$ and   a constant  $c  >0$ 
 such that, for any $\dl> 0$, there exists $C_\dl > 0$ such that 
\begin{align}
\bigg|\int_{\T^3}  \<1>_{W,N} \Dr_{S,N}^2 dx \bigg|
&\les 1 +  C_\dl \| \<1>_{W,N} \|_{\mathcal{C}^{-\frac 12-\eps}}^c + \dl \| \Upsilon_{S,N}\|_{L^2}^6 + \dl \| \Upsilon_{S,N} \|_{H^1}^2 + \|\ZZ_{W,N}\|_{\mathcal{C}^{1 - \eps}}^c,
\label{REQQ0} \\
\bigg| \int_{\T^3}  \<1>_{S,N} \Dr_{S,N}   \Dr_{W,N} dx \bigg|&\les  1 +  C_\dl \| \<1>_{S,N} \|_{\mathcal{C}^{-\frac 12-\eps}}^c + \dl \| \Upsilon_{S,N}\|_{L^2}^6 + \dl \| \Upsilon_{S,N} \|_{H^1}^2+\dl \| \Upsilon_{W,N} \|_{H^1}^2 \notag \\
&\hphantom{XXXX}+ \|\ZZ_{S,N}\|_{\mathcal{C}^{1 - \eps}}^c +\|\ZZ_{W,N}\|_{\mathcal{C}^{1 - \eps}}^c, \label{REQQ00}   \\
\bigg|  \int_{\T^3} \Dr_{S,N}^2 \Dr_{W,N} dx   \bigg| &\les 
1  + \| \Upsilon_{S,N}\|_{L^2}^6 +  \| \Upsilon_{S,N} \|_{H^1}^2+\| \Upsilon_{W,N} \|_{H^1}^2 +  \|\ZZ_{S,N}\|_{\mathcal{C}^{1 - \eps}}^c +\|\ZZ_{W,N}\|_{\mathcal{C}^{1 - \eps}}^c,
\label{REQQ000}
\end{align}

\noi
and 
\begin{align}
\begin{split}
A \bigg| \int_{\T^3}  \Big(  \<2>_{S,N}   & + \<1>_{S,N} \Dr_{S,N} + \Dr_{S,N}^2 \Big) dx \bigg|^3 \ge \frac A2  \bigg| \int_{\T^3}  \Big( 2 \<1>_{S,N} \Upsilon_{S,N} + \Upsilon_{S,N}^2 \Big) dx \bigg|^3 - \dl \| \Upsilon_{S,N} \|_{L^2}^6 \\ 
&\quad - C_{\dl, \ld} \Bigg\{
\bigg| \int_{\T^3}  \<2>_{S,N} dx \bigg|^3
+ \| \<1>_{S,N} \|_{\mathcal{C}^{-\frac 12-\eps}}^6
+ \| \ZZ_{W,N} \|_{\mathcal{C}^{1-\eps}}^6
\Bigg\}, 
\end{split}
\label{SSS11}
\end{align}

\noi
uniformly in $N \in \N$, where $\Dr_{S,N} = \Upsilon_{S,N} - \ld \ZZ_{W,N}$ and $\Dr_{W,N} = \Upsilon_{W,N} - \ld  \ZZ_{S,N}$ as in \eqref{chAa0} and \eqref{chAa1}.

\end{lemma}

In view of \eqref{REQQ0}, \eqref{REQQ00}, \eqref{REQQ000}, and \eqref{SSS11}, the main part is to control 
$\| \Upsilon_{S,N} \|_{L^2}^6$. In the following lemma, we present the control of the term $\| \Upsilon_{S,N} \|_{L^2}^6$.

\begin{lemma} \label{LEM:Cor1}
There exists a non-negative random variable $B(\o)$ with $\E [ B^p ] \le C_p < \infty$ for any finite $p \ge 1$ such that 
\begin{align}
\| \Upsilon_{S,N} \|_{L^2}^6 \les
\bigg| \int_{\T^3} \Big( 2 \<1>_{S,N} \Upsilon_{S,N} + \Upsilon_{S,N}^2 \Big) dx \bigg|^3
+ \| \Upsilon_{S,N} \|_{H^1}^2 + B(\o),
\label{L6coercieve}
\end{align}

\noi
uniformly in $N \in \N$.

\end{lemma}

By taking Lemmas \ref{LEM:Cor0} and \ref{LEM:Cor1} for granted, we first prove 
the uniform exponential integrability~\eqref{uniform1}.

\begin{proof}[Proof of Lemma \ref{LEM:uniLp0}]

Thanks to \eqref{drift11}, Lemma \ref{LEM:Cor0}, and Lemma \ref{LEM:Cor1}, We define the coercive part, that is, the positive parts $\U_N$ of $\W_N$ by
\begin{align}
\U_N(\dot{ \vec{\Upsilon}}^N)
= \E \bigg[
\frac A2 \bigg| \int_{\T^3} \Big( 2 \<1>_{S,N} \Upsilon_{S,N} + \Upsilon_{S,N}^2 \Big) dx \bigg|^3
+  \frac{1}{2}  \int_0^1 \| \dot{ \vec{\Upsilon}}^N(t)   \|_{\vec H^1_x}^2 dt \bigg],
\label{coer0}
\end{align}

\noi
where $\dot{ \vec{\Upsilon}}^N(t)=(\dot \Upsilon_S^N(t), \dot \Upsilon_W^N(t) )$ and the two terms on the right hand side of \eqref{coer0} play a key role as the coercive terms in the variational problem.

It follows from Lemma~\ref{LEM:Cor00}\,(i) with \eqref{Divv2} and \eqref{Divv3} that 
for any finite $p \ge 1$
\begin{align}
\E\Big[\|\ZZ_{S,N} \|_{\mathcal{C}^{1-\eps}}^p \Big]
&\leq \int_0^1 \E\Big[\| \<2>_{S,N}(t)   \|_{\mathcal{C}^{-1-\eps}}^p\Big] dt 
\les p < \infty,  \notag \\
\E\Big[\|\ZZ_{W,N} \|_{\mathcal{C}^{1-\eps}}^p \Big]
&\leq \int_0^1 \E\Big[\| \<1>_{S,N}(t)\<1>_{W,N}(t) \|_{\mathcal{C}^{-1-\eps}}^p\Big] dt 
\les p < \infty, 
\label{estiZZ}
\end{align}

\noi
uniformly in $N \in \N$. Then, by applying Lemmas \ref{LEM:Cor0} and \ref{LEM:Cor1} to \eqref{drift11} together with 
Lemma~\ref{LEM:Cor00} and \eqref{estiZZ},
we have
\begin{align}
\W_N(\dot{ \vec{\Upsilon}}^N) &\ge  -C_0 +
\E\bigg[
\Big(\frac A{2} -c|\ld|\Big) \bigg| \int_{\T^3}  \Big( 2 \<1>_{S,N} \Upsilon_{S,N} + \Upsilon_{S,N}^2 \Big) dx \bigg|^3 \notag \\
&\hphantom{XXXXXXX} + \Big( \frac 12 - c|\ld| \Big)  \int_0^1 \| \dot{ \vec{\Upsilon}}^N(t)  \|_{\vec H^1_x}^2 dt  \bigg] \notag \\
& \ge -C_0' + \frac 1{10} \U_N(\dot{ \vec{\Upsilon}}^N), 
\label{estiZZ1}
\end{align}

\noi
for any $0 < |\ld| < \ld_0$, provided  $A = A(\ld_0) >0$ is sufficiently large.
We observe that the estimate~\eqref{estiZZ1}
is uniform in $N \in \N$ and $\dot {\vec \Upsilon}^N=(\dot \Upsilon_S^N, \dot \Upsilon_W^N) \in  \vec{\mathbb{H}}_a^1$. Therefore, we get 
\begin{align}
\inf_{N \in \mathbb{N}} \inf_{\dot {\vec \Upsilon}^N \in  \vec{\mathbb{H}}^1_a} \W_N (\dot {\vec \Upsilon}^N)&\geq 
\inf_{N \in \mathbb{N}} \inf_{\dot {\vec \Upsilon}^N \in  \vec{\mathbb{H}}^1_a} \bigg\{ -C_0' + \frac{1}{10}\U_N(\dot {\vec \Upsilon}^N)\bigg\} \notag \\
& \ge - C_0' >-\infty.
\label{drift2}
\end{align}

\noi
Then, the uniform exponential integrability \eqref{uniform1} follows from \eqref{drift}, \eqref{drift1}, and \eqref{drift2}.

\end{proof}

We are now ready to prove the tightness of the truncated Gibbs measures $\{\rhoo_N\}_{N \in \N}$ by exploiting the uniform exponential integrability \eqref{uniform1}.

\begin{proposition}\label{PROP:tight}
There exists a small $\lambda_0 > 0$ such that for any $0 < |\lambda| < \lambda_0$ and $K > 0$, the sequence $\{\rhoo_N\}$ of truncated Gibbs measures in \eqref{GibbsSN} is tight.
\end{proposition}

\begin{proof}
We first prove that $Z_N$ in \eqref{Part81} is uniformly bounded away from 0:
\begin{align}
\inf_{N \in \N} Z_N > 0.
\label{lower0}
\end{align}

\noi
Thanks to \eqref{drift} and \eqref{drift1}, it suffices to establish an upper bound on $\W_N$ in \eqref{drift11}. By Lemma \ref{LEM:KCKON0} and \eqref{chAa0}, we have 
\begin{align*}
\bigg|\int_{\T^3}  \<1>_{S,N} \Dr_{S,N}  dx \bigg|^3
& \les \|\<1>_{S,N}\|_{\mathcal{C}^{-\frac 12-\eps}}^3 \|\Dr_{S,N}\|_{H^{\frac 12 + 2\eps}}^3\\
& \les 1 + \|\<1>_{S,N}\|_{\mathcal{C}^{-\frac 12-\eps}}^c + \|\ZZ_{W,N} \|_{\mathcal{C}^{1-\eps}}^c  + \|\Upsilon_{S,N}\|_{H^1}^c.
\end{align*}

\noi
Therefore, we have
\begin{align}
A \bigg| &\int_{\T^3} \Big( \<2>_{S,N} + 2 \<1>_{S,N} \Dr_{S,N} + \Dr_{S,N}^2 \Big) dx \bigg|^3 \notag \\
& \les 1 + \|\<2>_{S,N} \|_{\mathcal{C}^{-1-\eps}}^3
+ \|\<1>_{S,N}\|_{\mathcal{C}^{-\frac 12-\eps}}^c + \|\ZZ_{W,N} \|_{\mathcal{C}^{1-\eps}}^c + \|\Upsilon_{S,N}\|_{H^1}^c.
\label{tight4}
\end{align}

\noi
It follows from \eqref{drift11}, Lemma \ref{LEM:Cor0}, \eqref{tight4}, Lemma \ref{LEM:Cor00}, and \eqref{estiZZ} that we have
\begin{align*}
\inf_{(\dot \Upsilon_S^N, \dot \Upsilon_W^N) \in  \vec{\mathbb{H}}_a}  \W_N 
& \les 1 + 
\inf_{(\dot \Upsilon_S^N, \dot \Upsilon_W^N) \in  \vec{\mathbb{H}}_a}  \E\Bigg[\bigg( \int_0^1 \| \dot \Upsilon_S^N(t) \|_{H^1_x}^2 +\| \dot \Upsilon_W^N(t) \|_{H^1_x}^2 dt\bigg)^c\Bigg]
\les 1
\end{align*}

\noi
by taking $\dot {\vec \Upsilon}^N=(\dot \Upsilon_S^N, \dot \Upsilon_W^N)=(0,0)$, for example. 
Hence, we obtain the result \eqref{lower0}.

We now turn to the proof of tightness of the truncated Gibbs measures $\{\rhoo_N\}_{N \in \N}$. Fix small $\eps > 0$ and  let $ B_R \subset \vec H^{-\frac 12 - \eps}(\T^3)$ be the closed ball of radius $R> 0$ centered at the origin. Then, thanks to Rellich's compactness lemma, $B_R$ is compact in $\vec H^{-\frac 12 - 2\eps}(\T^3)$. In the following, we show that given any small $\dl > 0$, there exists $R = R(\dl ) \gg1 $ such that 
\begin{align}
\sup_{N \in \N} \rhoo_N(B_R^c)< \dl.
\label{tight0}
\end{align}

\noi
Given $M \gg1 $, let $F$ be a bounded smooth non-negative function such that 
\begin{align}
F(u,w)=
\begin{cases}
M, & \text{if }\|(u,w)\|_{\vec H^{-\frac 12 - \eps}(\T^3)} \leq \frac R2,\\
0, & \text{if }\|(u,w)\|_{\vec H^{-\frac 12 - \eps}(\T^3) }> R.
\end{cases}
\label{Cu0}
\end{align}

\noi
Then, in view of \eqref{lower0} and \eqref{cut000}, we have 
\begin{align}
\rhoo_N(B_R^c) &\leq Z_N^{-1} \int e^{-F(u,w)-\mathcal{H}_N^{\dia}(u,w)} \ind_{\{\int_{\T^3} \; :|u_N|^2: \;  dx \, \leq K\}} d\muu(u,w) \notag \\
&\les_{A,K} \int e^{-F(u,w)-\mathcal{H}_N^{\dia}(u,w)- A\big|\int_{\T^3} :\,|u_N|^2: \, dx\big|^3} d\muu(u,w) \notag \\
&=: \ft Z_N, 
\label{tight1}
\end{align}

\noi
uniformly in  $N \gg 1$. By recalling  the change of variables \eqref{Divv4} and \eqref{Divv5}, we define $\ft {\mathcal{H}}_N^\dia (\<1>_S+\Dr_S, \<1>_W+\Dr_W)$  by 
\begin{align}
\begin{split}
\ft {\mathcal{H}}_N^\dia (\<1>_S+\Dr_S, \<1>_W+\Dr_W)&= \frac \ld2 \int_{\T^3} \<3>_{S,W,N} dx +\ld \int_{\T^3}  \<1>_{S,N} \Dr_{S,N} \Dr_{W,N} dx\\
&\hphantom{X} +\frac \ld2\int_{\T^3}  \<1>_{W,N} \Dr_{S,N}^2 dx + \frac \ld2\int_{\T^3} \Dr_{S,N}^2 \Dr_{W,N} dx
\\
&\hphantom{X} + A \bigg| \int_{\T^3} \Big( \<2>_{S,N} + 2 \<1>_{S,N} \Dr_{S,N} + \Dr_{S,N}^2 \Big) dx \bigg|^3, 
\end{split}
\label{PUBAO1}
\end{align}

\noi
where 
$\Dr_{S,N} = \Upsilon_{S,N} - \ld \ZZ_{W,N}$ and $\Dr_{W,N} = \Upsilon_{W,N} -  \ld  \ZZ_{S,N}$  as in \eqref{chAa0} and \eqref{chAa1}. From \eqref{tight1} and
the Bou\'e-Dupuis  variational formula (Lemma~\ref{LEM:BoueDupu}), we have
\begin{align}
\begin{split}
-\log \ft Z_{N}= \inf_{(\dot \Upsilon^{N}_S, \dot \Upsilon^{N}_W)\in   \vec{\mathbb H}_a^1}
\E \bigg[& F(\<1>_S+ \Upsilon^{N}_S - \ld  \ZZ_{W,N}, \<1>_W+ \Upsilon^{N}_W  -  \ld  \ZZ_{S,N}) \\
& +  \ft {\mathcal{H}}^{\dia}_{N}(\<1>_S+ \Upsilon^{N}_S - \ld  \ZZ_{W,N}, \<1>_W+ \Upsilon^{N}_W  -  \ld  \ZZ_{S,N})\\
& + \frac{1}{2}  \int_0^1 \|\big( \dot \Upsilon^N_{S} (t), \dot \Upsilon^N_{W} (t) \big) \|_{H^1_x\times H^1_x}^2 dt \bigg].
\end{split}
\label{S0}
\end{align}

\noi
Since $\<1>_S- \ld \ZZ_{W,N}, \<1>_W-\ld \ZZ_{S,N} \in \H_{\leq 2}$, it follows from  Lemma \ref{LEM:Cor00}, \eqref{estiZZ}, Chebyshev's inequality, and choosing $R \gg 1$ that 
\begin{align}
\PP \bigg\{    \big\| \vec {\<1>}+\vec{\Upsilon}^N  -\ld \vec{\ZZ}^r_N
\big\|_{\vec H^{-\frac 12 - \eps}} >  \tfrac R2
\bigg\} &  \leq  \PP \bigg\{ \big\|\vec {\<1>}-\ld \vec{\ZZ}^r_N  \big\|_{\vec H^{-\frac 12 - \eps}} >  \tfrac R4 \bigg\} + \PP \bigg\{ \big\| \vec{\Upsilon}^N \big\|_{\vec H^1} >  \tfrac R4 \bigg\}  \notag \\
& \leq \frac 12 + \frac {16}{R^2} \E\Big[ \big\|    \vec{\Upsilon}^N   \big\|_{\vec H^1}^2    \Big], 
\label{S1}
\end{align}

\noi
uniformly in $N \in \N$ and $R\gg1$, where $\vec{\<1>}=(\<1>_S,\<1>_W)$, $\vec{\Upsilon}^N =({\Upsilon}^N_S, {\Upsilon}^N_W)$, and $\vec{\ZZ}^r_N=(\ZZ_{W,N}, \ZZ_{S,N})$\footnote{The superscript `r' means the reverse i.e. $\vec{\ZZ}_N=(\ZZ_{S,N}, \ZZ_{W,N})$ }.
Then, thanks to \eqref{Cu0}, \eqref{S1},  and Lemma \ref{LEM:Cor00}, we have 
\begin{align}
\E  \Big[ F\big(\vec{\<1>}+ \vec{\Upsilon}^N  -\ld \vec{\ZZ}^r_N   \big) \Big] &\ge M \E \bigg[  \ind_{ \Big\{\big\|\vec{\<1>}+ \vec{\Upsilon}^N  -\ld \vec{\ZZ}^r_N     \big\|_{\vec H^{-\frac 12 - \eps}} \leq  \tfrac R2\Big\}}\bigg]  \notag \\
&  \geq \frac{M}{2} -  \frac {16M }{R^2} \E\Big[  \big\| \vec{\Upsilon}^N  \big\|_{\vec H^1}^2   \Big] \notag \\
&  \geq \frac{M}{2}-  \frac {1}{4}\E\bigg[ \int_0^1 \big\|   \dot{ \vec{\Upsilon}}^N(t)  \big\|_{\vec H^1_x}^2 dt \bigg], 
\label{whrkx}
\end{align}

\noi
where in the last step we set  $M := \frac{1}{64} R^2$.
Therefore, it follows from \eqref{S0}, \eqref{whrkx}, and repeating the argument leading to \eqref{drift2} with possibly choosing $\ld$ sufficiently small that
\begin{align}
-\log \ft Z_{N} & \geq \frac M{2}+ \inf_{\dot {\vec \Upsilon}^N  \in \vec{\mathbb H}_a^1}
\E \bigg[ \ft {\mathcal{H}}^{\dia}_{N}(\vec {\<1>}+\vec{\Upsilon}^N  -\ld \vec{\ZZ}^r_N  ) + \frac14  \int_0^1 \big\|   \dot{ \vec{\Upsilon}}^N(t)  \big\|_{\vec H^1_x}^2 dt    \bigg] \notag \\
& \geq \frac M4, 
\label{tight3}
\end{align}

\noi
uniformly $N \in \N$ and $M = \frac{1}{64}  R^2 \gg 1$, where $\dot {\vec \Upsilon}^N= (\dot \Upsilon^{N}_S, \dot \Upsilon^{N}_W)$.
Then, for any given $\dl > 0$, we can choose $R = R(\dl ) \gg1 $
and set $M = \frac 1{64} R^2\gg1$ such that the desired bound \eqref{tight0} follows from \eqref{tight1} and~\eqref{tight3}. 
This proves the tightness of truncated Gibbs measures $\{\rhoo_N\}_{N\in \N}$.

\end{proof}

Thanks to Proposition \ref{PROP:tight}, for any $K>0$ and sufficiently small $0<|\ld|\ll 1$,  there exists a weakly convergent subsequence $\{\rhoo_{N_k}\}$ such that $\rhoo_{N_k}$ in \eqref{GibbsSN} converges weakly to a limit $\rhoo$, which is called the Gibbs measure, formally given by
\begin{align*}
d\rhoo(u,w) = Z^{-1}  
e^{ -\frac \ld2 \int_{\T^3}   \; :  |u|^2w : \;   \,dx -\infty } \ind_{ \{|\int_{\T^3} : | u|^2 :   dx| \le K\}}  d \muu(u,w).
\end{align*}

\noi
We conclude this subsection by presenting the proofs of  Lemmas \ref{LEM:Cor0} and  \ref{LEM:Cor1}.

\begin{proof}[Proof of Lemma \ref{LEM:Cor0}]

It follows from \eqref{dual}, \eqref{prod}, \eqref{embed}, \eqref{INT0P}
in Lemma \ref{LEM:KCKON0} and Young's inequality that we have
\begin{align}
\begin{split}
&\bigg| \int_{\T^3}   \<1>_{S,N} \Dr_{S,N} \Dr_{W,N} dx   \bigg|  \les \| \<1>_{S,N} \|_{\mathcal{C}^{-\frac 12-\eps}} \Big( \| \Dr_{S,N} \|_{H^{\frac 12+2\eps}} \| \Dr_{W,N} \|_{L^2} + \| \Dr_{W,N} \|_{H^{\frac 12+2\eps}} \| \Dr_{S,N} \|_{L^2}\Big)\\
&\les \| \<1>_{S,N} \|_{\mathcal{C}^{-\frac 12-\eps}}
\Big( \| \Upsilon_{S,N} \|_{H^{\frac 12+2\eps}} \big(\| \Upsilon_{W,N} \|_{L^2} + \| \ZZ_{S,N} \|_{\mathcal{C}^{1-\eps}} \big) + \| \ZZ_{S,N} \|_{\mathcal{C}^{1-\eps}}\| \ZZ_{W,N} \|_{\mathcal{C}^{1-\eps}}\\
&\hphantom{XX}+  \| \Upsilon_{W,N} \|_{H^{\frac 12+2\eps}} \big(\| \Upsilon_{S,N} \|_{L^2} + \| \ZZ_{W,N} \|_{\mathcal{C}^{1-\eps}} \big) \Big)\\
&\les \| \<1>_{S,N} \|_{\mathcal{C}^{-\frac 12-\eps}}
\Big( \| \Upsilon_{S,N} \|_{L^2}^{\frac 12-2\eps} \| \Upsilon_{S,N} \|_{H^1}^{\frac 12+2\eps}  \big(\| \Upsilon_{W,N} \|_{L^2} + \| \ZZ_{S,N} \|_{\mathcal{C}^{1-\eps}} \big) + \| \ZZ_{S,N} \|_{\mathcal{C}^{1-\eps}}\| \ZZ_{W,N} \|_{\mathcal{C}^{1-\eps}}\\
&\hphantom{XX}+  \| \Upsilon_{W,N} \|_{H^{\frac 12+2\eps}} \big(\| \Upsilon_{S,N} \|_{L^2} + \| \ZZ_{W,N} \|_{\mathcal{C}^{1-\eps}} \big) \Big)\\
&\les 1 + C_\dl \| Y_{W,N} \|_{\mathcal{C}^{-\frac 12-\eps}}^c
+ \dl\big( \| \Upsilon_{S,N} \|_{L^2}^6 + \| \Upsilon_{S,N} \|_{H^1}^2 + \| \Upsilon_{W,N} \|_{H^1}^2\big)
 +\| \ZZ_{S,N} \|_{\mathcal{C}^{1-\eps}}^c+ \| \ZZ_{W,N} \|_{\mathcal{C}^{1-\eps}}^c
\end{split}
\label{MCB01}
\end{align}

\noi
This yields \eqref{REQQ00}. Regarding \eqref{REQQ0},  it follows from \eqref{dual}, \eqref{prod}, \eqref{embed}, \eqref{INT0P} in Lemma \ref{LEM:KCKON0}, and Young's inequality that we have
\begin{align}
\begin{split}
\bigg|  \int_{\T^3} &   \<1>_{W,N} \Dr_{S,N}^2 dx \bigg| \les
\| \<1>_{W,N} \|_{\mathcal{C}^{-\frac 12-\eps}} \| \Dr_{S,N} \|_{H^{\frac 12+2\eps}} \| \Dr_{S,N} \|_{L^2} \\
&\les \| \<1>_{W,N} \|_{\mathcal{C}^{-\frac 12-\eps}}
\Big( \| \Upsilon_{S,N} \|_{H^{\frac 12+2\eps}} \big(\| \Upsilon_{S,N} \|_{L^2} + \| \ZZ_{W,N} \|_{\mathcal{C}^{1-\eps}} \big) + \| \ZZ_{W,N} \|_{\mathcal{C}^{1-\eps}}^2 \Big) \\
&\les
1+ C_\dl \| \<1>_{W,N} \|_{\mathcal{C}^{-\frac 12-\eps}}^c
+ \dl \| \Upsilon_{S,N} \|_{L^2}^6 + \dl \| \Upsilon_{S,N} \|_{H^1}^2
 + \| \ZZ_{W,N} \|_{\mathcal{C}^{1-\eps}}^c,
\end{split}
\label{MCB012}
\end{align}

\noi
which yields \eqref{REQQ0}. As for  the estimate \eqref{REQQ000}, it follows from Sobolev's inequality, the interpolation \eqref{INT0P}, Young's inequality, and H\"older's inequality with \eqref{embed} that  
\begin{align}
\bigg| \int_{\T^3} \Upsilon_{S,N}^2 \Upsilon_{W,N} dx \bigg|
&\les \| \Upsilon_{S,N} \|_{L^3} \| \Upsilon_{S,N} \|_{L^2} \| \Upsilon_{W,N} \|_{L^6}
\les \| \Upsilon_{S,N} \|_{H^\frac 12} \| \Upsilon_{S,N} \|_{L^2} \| \Upsilon_{W,N} \|_{H^1} \notag \\
& \les  \| \Upsilon_{S,N} \|_{L^2}^{\frac 12} \| \Upsilon_{S,N} \|_{H^1}^{\frac 12} \| \Upsilon_{S,N} \|_{L^2} \| \Upsilon_{W,N} \|_{H^1} \notag \\
&\les \| \Upsilon_{S,N} \|_{L^2}^6 + \| \Upsilon_{S,N} \|_{H^1}^2+\| \Upsilon_{W,N} \|_{H^1}^2
\label{MCB0123}
\end{align}

\noi
and 
\begin{align*}
&\bigg| \int_{\T^3} \Upsilon_{S,N}^2  \ZZ_{S,N} dx \bigg|
+\bigg| \int_{\T^3} \Upsilon_{W,N}  \ZZ_{W,N}^2 dx \bigg|
+
\bigg| \int_{\T^3}  \ZZ_{W,N}^2 \ZZ_{S,N} dx \bigg|\\
&\les 1+ \| \Upsilon_{S,N} \|_{L^2}^6+\| \Upsilon_{S,N} \|_{H^1}^2 +\| \Upsilon_{W,N} \|_{H^1}^2   + \| \ZZ_{S,N} \|_{\mathcal{C}^{1-\eps}}^c+\| \ZZ_{W,N} \|_{\mathcal{C}^{1-\eps}}^c.
\end{align*}

It remains to prove \eqref{SSS11}. Given any $\g > 0$,  there exists a constant $C= C(J)>0$ such that 
\begin{align}
\bigg|\sum_{j = 1}^J a_j\bigg|^\g
\ge \frac 12 |a_1|^\g -C \bigg( \sum_{j = 2}^J |a_j|^\g\bigg)
\label{YoungaJ}
\end{align}

\noi
for any $a_j\in \R$. Then, from \eqref{YoungaJ} and Cauchy's inequality, 
we have
\begin{align*}
&A \bigg| \int_{\T^3}  \Big( \<2>_{S,N}  + 2 \<1>_{S,N} \Dr_{S,N} + \Dr_{S,N}^2 \Big) dx \bigg|^3 \\
&\ge
\frac A2 
\bigg| \int_{\T^3}  \Big( 2 \<1>_{S,N} \Upsilon_{S,N} + \Upsilon_{S,N}^2 \Big) dx \bigg|^3
- C A \Bigg\{
\bigg| \int_{\T^3} \<2>_{S,N} dx \bigg|^3
+ |\ld|^3 \bigg| \int_{\T^3} \<1>_{S,N}  \ZZ_{W,N} dx \bigg|^3 \\
&\hphantom{XXXX}
+ |\ld|^3 \bigg| \int_{\T^3} \Upsilon_{S,N}  \ZZ_{W,N} dx \bigg|^3
+ \ld^6 \bigg| \int_{\T^3}  \ZZ_{W,N}^2 dx \bigg|^3
\Bigg\} \\
&\ge
\frac A2 
\bigg| \int_{\T^3}  \Big( 2 \<1>_{S,N} \Upsilon_{S,N} + \Upsilon_{S,N}^2 \Big) dx \bigg|^3
- \dl \| \Upsilon_{S,N} \|_{L^2}^6 \\
&\hphantom{XXXX} - C_{\dl, \ld} \Bigg\{
\bigg| \int_{\T^3} \<2>_{S,N} dx \bigg|^3
+ \| \<1>_{S,N} \|_{\mathcal{C}^{-\frac 12-\eps}}^6
+ \| \ZZ_{W,N} \|_{\mathcal{C}^{1-\eps}}^6
\Bigg\}.
\end{align*}

\noi
This completes the proof of Lemma \ref{LEM:Cor0}.
\end{proof}

Finally, we present the proof of Lemma \ref{LEM:Cor1}.

\begin{proof}[Proof of Lemma \ref{LEM:Cor1}]
The argument for the proof is based on \cite[Lemma 3.6]{OOT2}. Let us first consider the following case
\begin{align}
\| \Upsilon_{S,N} \|_{L^2}^2 \gg \bigg| \int_{\T^3} \<1>_{S,N} \Upsilon_{S,N} dx \bigg|. 
\label{OBW0}
\end{align}

\noi
In this case, we have
\begin{align*}
\begin{split}
\| \Upsilon_{S,N} \|_{L^2}^6 =
\bigg( \int_{\T^3} \Upsilon_{S,N}^2 dx\bigg)^3
\sim
\bigg| \int_{\T^3} \Big( 2 \<1>_{S,N} \Upsilon_{S,N} + \Upsilon_{S,N}^2 \Big) dx \bigg|^3,
\end{split}
\end{align*}

\noi
which shows \eqref{L6coercieve}.
Hence, we now consider the case 
\begin{align}
\| \Upsilon_{S,N} \|_{L^2}^2 \les \bigg| \int_{\T^3} \<1>_{S,N} \Upsilon_{S,N} dx \bigg|
\label{OBW1}
\end{align}

\noi
in the following.

Given $j \in \N$, 
define the  sharp frequency projectors $\Pi_j$
with a Fourier multiplier 
$\ind_{\{|n|\le 2\}}$ when $j = 1$
and $\ind_{\{2^{j-1}<  |n|\le 2^j\}}$ when $j \ge 2$.
We also set $\Pi_{\le j} = \sum_{k = 1}^j \Pi_k$
and $\Pi_{> j} = \Id - \Pi_{\le j}$.
Then, $\Upsilon_{S,N}$ can be written as 
\begin{align}
\Upsilon_{S,N}  = \sum_{j=1}^\infty \Pi_j \Upsilon_{S,N} = \sum_{j=1}^\infty (\ld_j \proj_j \<1>_{S,N} + w_j),
\label{OBW2}
\end{align}

\noi
where $\ld_j$ and $w_j$ are given by 
\begin{align}
\ld_j &:=
\begin{cases}
\frac{\jb{\Upsilon_{S,N}, \proj_j \<1>_{S,N}}}{\|\proj_j \<1>_{S,N}\|_{L^2}^2},  & \text{if } \| \proj_j \<1>_{S,N} \|_{L^2} \neq 0, \\
0, & \text{otherwise},
\end{cases}
\qquad
\text{and}
\qquad  
w_j :=
\proj_j \Upsilon_{S,N} - \ld_j \proj_j \<1>_{S,N}.
\label{OBW2a}
\end{align}

\noi
From the definition in \eqref{OBW2a},  $w_j = \Pi_j w_j$ is orthogonal to $\proj_j \<1>_{S,N}$ (and also to $\<1>_{S,N}$) in $L^2(\T^3)$.
Thus, we have 
\begin{align}
\| \Upsilon_N \|_{L^2}^2
&= \sum_{j=1}^\infty \Big( \ld_j^2 \| \proj_j \<1>_{S,N} \|_{L^2}^2 + \| w_j \|_{L^2}^2 \Big), \label{OBW3} \\
\int_{\T^3} \<1>_{S,N} \Upsilon_{S,N} dx
&= \sum_{j=1}^\infty \ld_j \| \proj_j \<1>_{S,N} \|_{L^2}^2.
\label{MSI19}
\end{align}

\noi
Therefore, thanks to \eqref{OBW1}, \eqref{OBW3}, and \eqref{MSI19}, we have
\begin{align}
\sum_{j=1}^\infty \ld_j^2 \| \proj_j \<1>_{S,N} \|_{L^2}^2
\les \bigg| \sum_{j=1}^\infty \ld_j \| \proj_j \<1>_{S,N} \|_{L^2}^2 \bigg|.
\label{OBW4}
\end{align}

\noi 
Fix  $j_0 \in \N$, which will be chosen later. From  Cauchy-Schwarz inequality and \eqref{OBW2a}, we have
\begin{align}
\bigg| \sum_{j=j_0+1}^\infty \ld_j \| \proj_j \<1>_{S,N} \|_{L^2}^2 \bigg|
&\le \bigg( \sum_{j=1}^\infty \ld_j^2 2^{2j} \| \proj_j \<1>_{S,N} \|_{L^2}^2 \bigg)^{\frac 12}
\bigg( \sum_{j=j_0+1}^\infty 2^{-2j} \| \proj_j \<1>_{S,N} \|_{L^2}^2 \bigg)^{\frac 12} \notag \\
&\le \bigg( \sum_{j=1}^\infty 2^{2j} \| \proj_j \Upsilon_{S,N} \|_{L^2}^2 \bigg)^{\frac 12}
\bigg( \sum_{j=j_0+1}^\infty 2^{-2j} \| \proj_j \<1>_{S,N} \|_{L^2}^2 \bigg)^{\frac 12} \notag \\
&\sim\| \Upsilon_{S,N} \|_{H^1}
 \|\Pi_{> j_0} \<1>_{S,N} \|_{H^{-1}}.
\label{MSI03}
\end{align}

\noi
On the other hand, it follows
from  Cauchy-Schwarz inequality, \eqref{OBW4}, and 
Cauchy's inequality that 
\begin{align}
\bigg| \sum_{j=1}^{j_0} \ld_j \| \proj_j \<1>_{S,N} \|_{L^2}^2 \bigg|
&\le \bigg( \sum_{j=1}^\infty \ld_j^2 \| \proj_j \<1>_{S,N} \|_{L^2}^2 \bigg)^{\frac 12}
\bigg( \sum_{j=1}^{j_0} \| \proj_j \<1>_{S,N} \|_{L^2}^2 \bigg)^{\frac 12}  \notag \\
&\le C\bigg| \sum_{j=1}^\infty \ld_j \| \proj_j \<1>_{S,N} \|_{L^2}^2 \bigg|^{\frac 12}
\bigg( \sum_{j=1}^{j_0} \| \proj_j \<1>_{S,N} \|_{L^2}^2 \bigg)^{\frac 12} \notag \\
&\le
\frac 12 \bigg| \sum_{j=1}^\infty \ld_j \| \proj_j \<1>_{S,N} \|_{L^2}^2 \bigg|
+ C' \|\Pi_{\le j_0}\<1>_{S,N} \|_{L^2}^2.
\label{MSI003}
\end{align}

\noi
Hence,  from \eqref{MSI03} and \eqref{MSI003}, 
we obtain 
\begin{align}
\begin{split}
\bigg| \sum_{j=1}^\infty \ld_j \| \proj_j \<1>_{S,N} \|_{L^2}^2 \bigg|
\les
\| \Upsilon_{S,N}\|_{H^1} \|\Pi_{> j_0} \<1>_{S,N} \|_{H^{-1}}
+\|\Pi_{\le j_0}\<1>_{S,N} \|_{L^2}^2.
\end{split}
\label{OBW5}
\end{align}

\noi
By noticing that $\<1>_{S,N}$ is spatially homogeneous, we have
\begin{align}
 \|\Pi_{> j_0} \<1>_{S,N} \|_{H^{-1}}^2
&=  \int_{\T^3} :\! ( \jb{\nb}^{-1}\Pi_{> j_0} \<1>_{S,N} )^2 \!: dx
+  \E \big[( \jb{\nb}^{-1}\Pi_{> j_0} \<1>_{S,N} )^2 \big].
\label{OBW5x}
\end{align}

\noi
With recalling \eqref{CenGauss}, we can bound  the second term  by 
\begin{align}
\wt \s_{j_0} : =  \E \big[( \jb{\nb}^{-1}\Pi_{> j_0} \<1>_{S,N} )^2 \big]
 = \sum_{\substack{n\in \Z^3\\|n| > 2^{j_0}}}\frac{\chi_N^2(n)}{\jb{n}^4}
 \les 2^{-j_0}.
\label{OBW5xx}
\end{align}

\noi
Let $Z_{N, j_0} =  \jb{\nb}^{-1}\Pi_{> j_0} \<1>_{S,N} $.
By proceeding with Lemma \ref{LEM:Wick2}, we have
\begin{align}
\E \Bigg[ \bigg(
 \int_{\T^3} :\! Z_{N, j_0}^2\!: dx\bigg)^2 \Bigg]
& = \int_{\T^3_x \times \T^3_y}
 \E\Big[  H_2( Z_{N, j_0}(x); \wt \s_{j_0})
 H_2( Z_{N, j_0} (y); \wt \s_{j_0})\Big] dx dy \notag \\
& =
\sum_{\substack{n_1, n_2\in \Z^3\\|n_j| > 2^{j_0}}}\frac{\chi_N^2(n_1)\chi_N^2(n_2)}{\jb{n_1}^4\jb{n_2}^4}
\int_{\T^3_x \times \T^3_y}
e_{n_1 + n_2}(x-y) dx dy \notag \\
& 
= \sum_{\substack{n\in \Z^3\\|n| > 2^{j_0}}}\frac{\chi_N^4(n)}{\jb{n}^8}
\sim 2^{-5j_0}.
\label{OBW5y}
\end{align}

\noi
We now define a non-negative random variable $B_1(\o)$ by 
\begin{align}
B_1(\o) = 
\bigg(
\sum_{j=1}^\infty
2^{4 j}
\Big(
 \int_{\T^3} :\! Z_{N, j}^2\!: dx \Big)^2
\bigg)^{\frac 12}.
\label{MSI06}
\end{align}

\noi
Thanks to Minkowski's integral inequality, the Wiener chaos estimate (Lemma \ref{LEM:hyp}), and \eqref{OBW5y}, 
we have 
\begin{align*}
\E \big[B_1^p\big] \leq p^p
\Bigg(
\sum_{j=1}^\infty
2^{4 j}
\bigg\|
 \int_{\T^3} :\! Z_{N, j}^2\!: dx \bigg\|_{L^2(\O)}^2
\Bigg)^{\frac p2}
\les p^p < \infty
\end{align*}

\noi
for any finite $p \geq 2$ (and hence for any finite $p \ge 1$).
Hence, from 
\eqref{OBW5x}, \eqref{OBW5xx},  and \eqref{MSI06}, we obtain
\begin{align}
 \|\Pi_{> j_0} \<1>_{S,N} \|_{H^{-1}}^2
\les  2^{-2 j_0} B_1(\o)+ 2^{-j_0}.
\label{MSI0003}
\end{align}

\noi
We next set a non-negative random variable $B_2(\o)$ as follows:
\begin{align*} 
B_2(\o) := \sum_{j = 1}^\infty\bigg|\int_{\T^3} :\! (\proj_{j} \<1>_{S,N})^2 \!: dx\bigg|.
\end{align*}

\noi
Then, a similar computation as above shows 
\begin{align}
\|\Pi_{\le j_0}\<1>_{S,N} \|_{L^2}^2
&=  \int_{\T^3} :\! (\proj_{\le j_0} \<1>_{S,N})^2 \!: dx
+  \E \big[ (\proj_{\le j_0} \<1>_{S,N} )^2 \big]  \notag \\
&\les B_2(\o) + 2^{j_0}
\label{MSI01}
\end{align}

\noi
and  $\E \big[B_2^p\big] \leq C_p < \infty$
for any finite $p \geq 1$.

Therefore, by combining \eqref{OBW1}, \eqref{MSI19} \eqref{OBW5},  \eqref{MSI0003}, and \eqref{MSI01} 
together, choosing $2^{j_0} \sim 1 + \| \Upsilon_{S,N} \|_{H^1}^{\frac 23}$, 
and applying Cauchy's inequality, we have
\begin{align}
\| \Upsilon_{S,N} \|_{L^2}^6
& \les \bigg| \int_{\T^3} \<1>_{S,N} \Upsilon_{S,N} dx \bigg|^3
= \bigg| \sum_{j=0}^\infty \ld_j \| \proj_j \<1>_{S,N} \|_{L^2}^2 \bigg|^3 \notag \\
&\les \Big(   2^{-3j_0} B_1(\o)^{\frac 32} + 2^{-\frac 32 j_0}  \Big) \| \Upsilon_{S,N} \|_{H^1}^3 +   B_2^3(\o)
+ 2^{3j_0}   \notag \\
& \les
\| \Upsilon_{S,N} \|_{H^1}^2 + 
B_1^3(\o) + B_2^3(\o) + 1.
\label{MSI009}
\end{align}

\noi
This proves \eqref{L6coercieve} in the case where \eqref{OBW1} holds.
Therefore, we conclude the proof of Lemma \ref{LEM:Cor1}.
\end{proof}

\begin{remark}\label{REM:bd}
\rm

From the proof of Lemma \ref{LEM:Cor1}
(see \eqref{OBW0} and \eqref{MSI009})
with Lemma \ref{LEM:Cor1},
we also have
\begin{align}
\begin{split}
\E\Bigg[\bigg| \int_{\T^3} \<1>_{S,N} \Upsilon_{S,N} dx \bigg|^3\Bigg]
& \les
\E\Big[\| \Upsilon_{S,N} \|_{L^2}^6
+ \| \Upsilon_{S,N} \|_{H^1}^{2}\Big]
+ 1\\
& \les
\U_N + 1, 
\end{split}
\label{Cor01}
\end{align}

\noi
where  $\U_N$ is as in \eqref{coer0}.

\end{remark}

\subsection{Mutual singularity of Gibbs measures and Gaussian free field}
\label{SUBSEC:singular}
In this subsection, we present the proof of the mutual singularity of the Gibbs measure $\rhoo$, constructed in the previous subsections, and the base Gaussian free field $\muu$.

\begin{proposition}\label{PROP:singular}
Let $\eps>0$ and $\ld  \in \R \setminus\{0\}$. Then, there exists a strictly increasing sequence $\{N_k \}_{k\in \N} \subset \N$ such that the set
\begin{align*}
S:=\{ (u,w) \in \vec H^{-\frac 12-\eps}(\T^3): \lim_{k\to \infty } (\log N_k)^{-\frac 34} \H_{N_k}(u,w)=0    \}  
\end{align*}

\noi
satisfies
\begin{align}
\muu (S)=1 \qquad \text{but} \qquad \rhoo(S)=0,
\label{Sing00}
\end{align}

\noi
where $\H_N$ is the interaction potential with Wick renormalization as in \eqref{CORINT01}
\begin{align*}
\H_N (u, w)=  \frac \ld 2  \int_{\T^3}    :\!  |u_N|^2  w_N \! :   \,dx. 
\end{align*}

\noi
This shows that the Gibbs measure $\rhoo$ in \eqref{GibPARIS008} and the massive Gaussian free field $\muu$ in \eqref{gauss1} are mutually singular.
\end{proposition}

\begin{proof}
Thanks to Wick's theorem  (Lemma \ref{LEM:Wick}) and Lemma \ref{LEM:COKSS0}, we have
\begin{align*}
\| \H_N(u,w)  \|_{L^2(\muu)}^2 &=\E_{\muu} \Bigg[  \bigg|\frac \ld2 \int_{\T^3} :\! |u_N|^2w_N \!:  \, dx\bigg|^2 \Bigg] \\
&=\E_{\PP} \Bigg[ \bigg|\frac \ld2 \int_{\T^3} \<3>_{S,W,N}(1) dx \bigg|^2 \Bigg]
\\
&\sim \ld^2 \sum_{\substack{n_1+n_2+n_3=0 \\ |n_j|\les N }} \jb{n_1}^{-2} \jb{n_2}^{-2} \jb{n_3}^{-2}
\\
&\sim \ld^2
\sum_{\substack{|n_3|  \les N}} \jb{n_3}^{-2} \jb{n_3}^{-1} \sim \log N, 
\end{align*}

\noi
as $N\to \infty$, where $\<3>_{S,W,N}(1)$ is  in \eqref{cher19}.
Hence, we obtain 
\begin{align*}
\lim_{N \to \infty} (\log N)^{-\frac 34}
\| \H_N(u,w) \|_{L^2(\muu)}
\les \lim_{N \to \infty}
(\log N)^{-\frac 14}
= 0.
\end{align*}

\noi
Therefore, there exists a subsequence of $\{\H_{N}\}_{N\in \N}$ such that 
\begin{align*}
\lim_{k \to \infty} (\log N_k)^{-\frac 34} \H_{N_k}(u,w) =0, 
\end{align*}

\noi
almost surely with respect to $\muu=\mu_1 \otimes \mu_1$. This proves $\muu(S) = 1$ in \eqref{Sing00}.

It remains to show $\rhoo(S)=0$. For any given $k \in \N$, define $G_k(u,w)$ by 
\begin{align}
G_k(u,w) = -(\log N_k)^{-\frac 34} \H_{N_k} (u,w).
\label{Gk}
\end{align}

\noi
Our goal is to prove that $e^{G_k(u,w)}$ tends to $0$ in $L^1(\rhoo)$, which implies that there exists a subsequence of $G_k(u,w)$ tending to $- \infty$, almost surely with respect to the Gibbs measure $\rhoo$, which in turn yields the second part in~\eqref{Sing00}: $\rhoo(S) = 0$.

Recall that $\phi$ is a smooth bump function as in Subsection \ref{SUBSEC:21} (i.e.~as in \eqref{phi1}). Thanks to Fatou's lemma, the weak convergence of $\rhoo_M$ to $\rhoo$, the boundedness of $\phi$, \eqref{GibbsSN}, and \eqref{cut000}, we have
\begin{align}
\int  e^{G_k(u,w)}d\rhoo(u,w)
& \leq \liminf_{K \to \infty}
\int \phi \bigg(\frac{G_k(u,w) }{K}\bigg)e^{G_k(u,w)}d\rhoo(u,w) \notag \\
& = \liminf_{K \to \infty}
\lim_{M \to \infty}
\int \phi \bigg(\frac{G_k(u,w)}{K}\bigg)e^{G_k(u,w)}d\rhoo_M(u,w) \notag \\
& \leq  \lim_{M \to \infty} \int e^{G_k(u,w)}d\rhoo_M(u,w) \notag \\
&= Z^{-1} \lim_{M \to \infty} \int e^{G_k(u,w) - \H_M^\dia(u,w)}  \ind_{\{\int_{\T^3} \; :|u_M|^2: \;  dx \, \leq K\}}d\muu(u,w) \notag \\
&\les_{A,K} Z^{-1} \lim_{M \to \infty} \int e^{G_k(u,w)-\H_M^\dia(u,w)- A\big|\int_{\T^3} :\,|u_M|^2: \, dx\big|^3}d\muu(u,w) \notag \\
& =: Z^{-1} \lim_{M\to \infty} C_{M, k},
\label{exp0} 
\end{align}

\noi
provided that  $\lim_{M\to \infty} C_{M, k} $ exists.
Here,  $Z = \lim_{M \to \infty} Z_M$ denotes the partition function for~$\rhoo$.

We now show that the right-hand side of \eqref{exp0} converges to $0$ as $k \to \infty$.
As in the proof of \eqref{drift2}, we proceed with the change of variables \eqref{Divv4} and \eqref{Divv5}: 
\begin{align}
\dot \Upsilon^N_{S}(t) &= \dot \Dr_S(t)  + \ld \dot \ZZ_{W,N}(t), \\
\dot \Upsilon^N_{W}(t) &= \dot \Dr_W(t)  +  \ld \dot \ZZ_{S,N}(t). 
\end{align}

\noi
Then, by the Bou\'e-Dupuis variational formula (Lemma \ref{LEM:BoueDupu}) and \eqref{Gk}, we have 
\begin{align}
- \log C_{M, k} & = \inf_{(\dot \Upsilon^{M}_S, \dot \Upsilon^{M}_W )\in  \vec { \mathbb{H}}_a^1}
\E \bigg[  (\log N_k)^{-\frac 34} \H_{N_k}(\<1>_S+ \Upsilon^{M}_S -\ld \ZZ_{W,M}, \<1>_W+ \Upsilon^{M}_W - \ld \ZZ_{S,M})  \notag \\
& \hphantom{XXXXXXXXXXX}+  \H^{\dia}_{M}(\<1>_S+ \Upsilon^{M}_S -\ld \ZZ_{W,M}, \<1>_W+ \Upsilon^{M}_W -\ld \ZZ_{S,M} ) \notag \\
&\hphantom{XXXXXXXXXXX}+ A \bigg| \int_{\T^3} \Big( \<2>_{S,M} + 2 \<1>_{S,M} \Dr_{S,M} + \Dr_{S,M}^2 \Big) dx \bigg|^3 \notag \\
&\hphantom{XXXXXXXXXXX} + \frac{1}{2}   \int_0^1 \|\big( \dot \Upsilon^M_{S} (t), \dot \Upsilon^M_{W} (t) \big) \|_{H^1_x\times H^1_x}^2 dt  \bigg] \notag \\
& =:  \inf_{\dot{ \vec{\Upsilon}}^M \in  \vec { \mathbb{H}}_a^1}   \ft \W_{M, k}\big( \dot{ \vec{\Upsilon}}^M \big), 
\label{OBMK0}
\end{align}

\noi
where $\dot{ \vec{\Upsilon}}^M=(\dot \Upsilon^{M}_S, \dot \Upsilon^{M}_W)$ and $\H^{\dia}_{M}(u,w)$ is as in \eqref{density11}. We now prove that the right-hand side (and hence the left-hand side) of \eqref{OBMK0} diverges to $\infty$ as $k \to \infty$.

As in the proof of \eqref{estiZZ1} in Subsection \ref{SUBSEC:uniform}, we estimate the last three terms on the right-hand side of \eqref{OBMK0} as follows
\begin{align}
&\E \Bigg[  \H^{\dia}_{M}(\<1>_S+ \Upsilon^{M}_S -\ld \ZZ_{W,M}, \<1>_W+ \Upsilon^{M}_W -\ld \ZZ_{S,M} )  + \frac{1}{2}  \int_0^1 \|\big( \dot \Upsilon^M_{S} (t), \dot \Upsilon^M_{W} (t) \big) \|_{H^1_x\times H^1_x}^2 dt \notag \\
&\hphantom{XXXXXXXXXXX} +A \bigg| \int_{\T^3} \Big( \<2>_{S,M} + 2 \<1>_{S,M} \Dr_{S,M} + \Dr_{S,M}^2 \Big) dx \bigg|^3  \Bigg] \notag \\
&\geq   -C_0 + \frac{1}{10}\U_M, 
\label{UM00}
\end{align}

\noi
where $\U_M = \U_M(\dot{ \vec{\Upsilon}}^M)$ is given by  \eqref{coer0} 
\begin{align}
\U_M= \E \bigg[ \frac A2 \bigg| \int_{\T^3} \Big( 2 \<1>_{S,M} \pi_M \Upsilon^M_S + (\pi_M \Upsilon^M_S)^2 \Big) dx \bigg|^3 + \frac{1}{2} \int_0^1 \| \big( \dot \Upsilon^M_{S} (t), \dot \Upsilon^M_{W} (t) \big) \|_{H^1_x\times H^1_x}^2 dt \bigg].
\label{UM01}
\end{align}

We now turn to the first term of the right-hand side of \eqref{OBMK0}, which gives the main divergent contribution. It follows from \eqref{expa0} that we have
\begin{align}
&\H_{N_k}(\<1>_S+ \Upsilon^{M}_S -\ld \ZZ_{W,M}, \<1>_W+ \Upsilon^{M}_W - \ld \ZZ_{S,M}) \notag \\
&= \frac \ld2 \int_{\T^3} \<3>_{S,W,N_k} dx+\frac {\ld}2 \int_{\T^3}  \underbrace{\<2>_{S,N_k} }_{\in \mathcal{C}^{-1-\eps}}  \underbrace{\Dr_{W,N_k }}_{\in H^1}dx \notag \\
& \hphantom{X}+\ld\int_{\T^3} \underbrace{ \<1>_{S,N_k} \<1>_{W,N_k} }_{\in \mathcal{C}^{-1-\eps}}  \underbrace{  \Dr_{S,N_k} }_{\in H^1} dx 
+  \ld \int_{\T^3}  \<1>_{S,N_k} \Dr_{S,N_k} \Dr_{W,N_k} dx \notag \\
& \hphantom{X}+ \frac \ld2\int_{\T^3}  \<1>_{W,N_k} \Dr_{S,N_k}^2 dx + \frac\ld2 \int_{\T^3} \Dr_{S,N_k}^2 \Dr_{W,N_k} dx \notag \\
&=: \text{I} + \II + \III + \IV + \5+\6
\label{expan}
\end{align}

\noi
for $N_k \leq M$,
where
$\Dr_{S,N_k}$ and $\Dr_{W,N_k}$ are given by 
\begin{align}
\Dr_{S,N_k}:&= \pi_{N_k} \Dr_S = \pi_{N_k} \Upsilon^M_S - \s \pi_{N_k} \ZZ_{W,M} \label{CHA00},\\
\Dr_{W,N_k}:&= \pi_{N_k} \Dr_W = \pi_{N_k} \Upsilon^M_W - \s \pi_{N_k} \ZZ_{S,M}.
\label{CHA0}
\end{align}

\noi
After taking the expectation in \eqref{expan}, we will see that the second term $\II$ and third therm $\III$ on the right-hand side of \eqref{expan} 
(which are precisely the terms killed by additional counterterms $\dl_{N, \ld}$) give divergent contributions; see \eqref{logdiv00} and \eqref{logdiv01} below. Thanks to Lemma \ref{LEM:Cor00} (ii), the first term $\text{I}$ on the right-hand side of \eqref{expan} becomes $0$ under an expectation. Regarding the last three terms, we proceed as in Subsection \ref{SUBSEC:uniform} and obtain
\begin{align}
\begin{split}
\big|\E \big[ \IV + \5+\6 \big] \big| &\les 
C(\<1>_{S, N_k}, \<1>_{W, N_k}, \pi_{N_k} \ZZ_{S,M}, \pi_{N_k} \ZZ_{W,M})+ \U_{N_k}\\
& \les  1+  \U_{N_k},
\end{split}
\label{UM02}
\end{align}

\noi
where $C(\<1>_{S, N_k}, \<1>_{W, N_k}, \pi_{N_k} \ZZ_{S,M}, \pi_{N_k} \ZZ_{W,M})$ denotes high moments of several stochastic objects including $\<1>_{S, N_k}, \<1>_{W, N_k}, \pi_{N_k} \ZZ_{S,M}$, and $\pi_{N_k} \ZZ_{W,M}$ and  $\U_{N_k} = \U_{N_k}( \pi_{N_k} \dot \Upsilon^M_S, \pi_{N_k} \dot \Upsilon^M_W)$ is given by  \eqref{coer0} with $\Upsilon_{S,N_k}=\pi_{N_k} \Upsilon^M_S$ and $\Upsilon_{W,N_k}=\pi_{N_k} \Upsilon^M_W$:
\begin{align}
\U_{N_k}= \E \bigg[ \frac A2 \bigg| \int_{\T^3} \Big( 2 \<1>_{S,N_k} \pi_{N_k} \Upsilon^M_S + (\pi_{N_k} \Upsilon^M_S)^2 \Big) dx \bigg|^3 + \frac{1}{2} \int_0^1 \| \big( \pi_{N_k} \dot \Upsilon^M_{S} (t), \pi_{N_k} \dot \Upsilon^M_{W} (t) \big) \|_{H^1_x\times H^1_x}^2 dt \bigg].
\label{uber01}
\end{align}

\noi
Thanks to the smallness of $(\log N_k)^{-\frac 34} $ in \eqref{OBMK0}, the second term in \eqref{uber01} can be controlled by the positive terms $\U_M$ in \eqref{UM00} (in particular by the second term in \eqref{UM01}). Regarding  the first term in \eqref{uber01},
it follows from \eqref{Cor01}, $\pi_{N_k}  \Upsilon^M_S = \pi_{N_k} \pi_M \Upsilon^M_S$ for $N_k \le M$, and Lemma~\ref{LEM:Cor1} with \eqref{UM01} that 
\begin{align}
\begin{split}
 \E \Bigg[ \bigg| \int_{\T^3} \Big( 2 \<1>_{S,N_k} \Upsilon^M_S + (\pi_{N_k} \Upsilon_S^M)^2 \Big) dx \bigg|^3 \Bigg]
&
\les 
\bigg\| \int_{\T^3} \<1>_{S, N_k} \pi_{N_k} \Upsilon_S^M dx \bigg\|_{L^3_\o}^3
+ \| \pi_{N_k} \Upsilon_S^M \|_{L^6_\o L^2_x}^6 \\
&
\les 
1+ \| \pi_M \Upsilon^M_S \|_{L^6_\o L^2_x}^6
+ \| \Upsilon^M_S \|_{L^2_\o H^1_x}^2 \\
&\les
1+\U_M
\end{split}
\notag
\end{align}
for $N_k \le M$. Hence,  $\U_{N_k}$ in \eqref{uber01} can be controlled by  $\U_M$ in \eqref{UM01}:
\begin{align}
\U_{N_k}
\les 
1+\U_M.
\label{UM03}
\end{align}

\noi
Therefore, it follows from \eqref{OBMK0}, \eqref{UM00}, \eqref{expan}, \eqref{UM02}, 
and \eqref{UM03} that we have
\begin{align}
&\ft \W_{M, k}(\dot \Upsilon^M_S, \dot \Upsilon^M_W) \notag \\
&\ge \ld (\log N_k)^{-\frac 34}  \Bigg(\E \bigg[ 
\int_{\T^3}  \<2>_{S,N_k}  \Dr_{W,N_k}dx \bigg] +\E\bigg[\int_{\T^3}  \<1>_{S,N_k} \<1>_{W,N_k}   \Dr_{S,N_k} dx  \bigg] \Bigg)- C_1 + \frac{1}{20}\U_M
\label{W0}
\end{align}

\noi
for any $M\ge  N_k \gg 1$. Therefore, it suffices to estimate the contribution from the second and third term on the right-hand side of~\eqref{expan}. Let us first state a lemma whose proof is presented at the end of this subsection.

\begin{lemma}\label{LEM:logdiv}
We have
\begin{align}
\E \bigg[  \int_0^1 \jb{ \dot \ZZ_{S,N} (t), \dot \ZZ_{S,M} (t)}_{H^1_x}dt \bigg] 
&\sim  \log N \label{SNM}\\
\E \bigg[  \int_0^1 \jb{ \dot \ZZ_{W,N} (t), \dot \ZZ_{W,M} (t)}_{H^1_x}dt \bigg] 
&\sim  \log N
\label{WNM}
\end{align}

\noi
for  any $1\leq N \le M$, where $\dot \ZZ_{S,N} = \pi_N \dot \ZZ^{N}_S$ and  $\dot \ZZ_{W,N} = \pi_N \dot \ZZ^{N}_W$.
\end{lemma}

By taking Lemma \ref{LEM:logdiv} for granted, we first complete the proof of Proposition \ref{PROP:singular}. It follows from \eqref{Divv0}, \eqref{Divv1}, \eqref{Divv2}, \eqref{Divv3},  $\ZZ_{S, N_k} = \pi_{N_k}\ZZ^{N_k}_S, \ZZ_{W, N_k} = \pi_{N_k}\ZZ^{N_k}_W$, \eqref{CHA00}, \eqref{CHA0}, Lemma \ref{LEM:logdiv}, Cauchy's inequality (with small $\eps_0 > 0$), and Lemma \ref{LEM:Cor00} (see \eqref{logtdiv}) that we have 
\begin{align}
\begin{aligned}
\ld \E &\bigg[ \int_{\T^3}  \<2>_{S,N_k}  \Dr_{W,N_k}dx   \bigg] = \ld \E\bigg[ \int_0^1 \int_{\T^3}  \<2>_{S,N_k}(t) \dot \Dr_{W,N_k}(t) dx dt \bigg] \\
& = \ld^2 \E \bigg[ \int_0^1  \jb{ \dot \ZZ_{S,N_k}(t),  \dot \ZZ_{S,M} (t)}_{H^1_x}  dt \bigg]
+ \ld \E\bigg[   \int_0^1   \jb{\dot \ZZ_{S,N_k} (t), \dot \Upsilon^M_W (t)}_{H^1_x} dt \bigg]\\
& \geq  c\log N_k
- \eps_0 \E\bigg[ \int_0^1 \|   \<2>_{N_k}(t) \|_{H^{-1}_x}^2 dt   \bigg]
- C_{\eps_0}\E\bigg[ \int_0^1 \| \dot \Upsilon^M_W (t) \|_{H^1_x}^2 dt   \bigg]\\
& \ge  
\frac{c}{2} \log N_k
- C_{\eps_0} \E\bigg[ \int_0^1 \| \dot \Upsilon^M_W (t) \|_{H^1_x}^2 dt \bigg]
\end{aligned}
\label{logdiv00}
\end{align}

\noi
and
\begin{equation}
\begin{split}
\ld \E &\bigg[ \int_{\T^3}  \<1>_{S,N_k} \<1>_{W,N_k}   \Dr_{S,N_k} dx   \bigg] = \ld \E\bigg[ \int_0^1 \int_{\T^3}   (\<1>_{S,N_k}\<1>_{W,N_k})(t)  \dot \Dr_{S,N_k}(t) dx  dt \bigg] \\
& = \ld^2 \E \bigg[ \int_0^1  \jb{ \dot \ZZ_{W,N_k}(t),  \dot \ZZ_{W,M} (t)}_{H^1_x}  dt \bigg]
+ \ld \E\bigg[   \int_0^1   \jb{\dot \ZZ_{W,N_k} (t), \dot \Upsilon^M_S (t)}_{H^1_x} dt \bigg]\\
& \geq  c\log N_k
- \eps_0 \E\bigg[ \int_0^1 \|   (\<1>_{S,N_k}\<1>_{W,N_k})(t) \|_{H^{-1}_x}^2 dt   \bigg]
- C_{\eps_0}\E\bigg[ \int_0^1 \| \dot \Upsilon^M_S (t) \|_{H^1_x}^2 dt   \bigg]\\
& \ge  
\frac{c}{2} \log N_k
- C_{\eps_0} \E\bigg[ \int_0^1 \| \dot \Upsilon^M_S (t) \|_{H^1_x}^2 dt \bigg]
\end{split}
\label{logdiv01}
\end{equation}

\noi
for $M\ge  N_k \gg 1$. Then, thanks to \eqref{OBMK0}, \eqref{W0}, \eqref{logdiv00}, and \eqref{logdiv01}, we have 
\begin{align}
- \log C_{M, k} &  \ge \inf_{(\dot \Upsilon^{M}_S, \dot \Upsilon^{M}_W)\in  \vec { \mathbb{H}}_a^1  }\Big\{ 
c (\log N_k)^{\frac 14} - C_2 + \frac{1}{40}\U_M\Big\}
\geq 
c (\log N_k)^{\frac 14} - C_2
\label{OBMK02}
\end{align}

\noi
for any sufficiently large $k \gg1 $ (such that $N_k \gg 1$).
Therefore, from \eqref{OBMK02}, we obtain 
\begin{align}
C_{M, k} \les \exp \Big( -c (\log N_k)^{\frac 14} \Big)
\label{OBMK03}
\end{align}

\noi
for $M\ge  N_k \gg 1$,  uniformly in $M \in \N$. Therefore, 
by taking limits in $M \to \infty$ and then $k \to \infty$, 
we obtain from \eqref{exp0} and \eqref{OBMK03} that
\begin{align*}
\lim_{k \to \infty} \int  e^{G_k(u,w)}d\rhoo(u,w) = 0,
\end{align*}

\noi
which implies that there exists a subsequence of $G_k(u,w)$ tending to $- \infty$, almost surely with respect to the Gibbs measure $\rhoo$. This yields the second part in~\eqref{Sing00} i.e.~$\rhoo(S) = 0$.

\end{proof}

We conclude this section  by presenting the proof of Lemma \ref{LEM:logdiv}.

\begin{proof}[Proof of Lemma \ref{LEM:logdiv}]
We first consider \eqref{SNM}. For the convenience of the presentation, we suppress the time dependence in the following. From~\eqref{Divv2}, we have 
\begin{align}
\ft {\dot \ZZ}_{S, N} (n)
&= \jb{n}^{-2} 
\sum_{\substack{n_1, n_2 \in \Z^3 \\ n=n_1+n_2 \neq 0}}
\ft {\<1>}_{S,N}(n_1) \ft {\<1>}_{S,N}(n_2)
\label{SS0}
\end{align}

\noi
for $n \ne 0$.
On the other hand, when $n = 0$, 
it follows from Lemma \ref{LEM:Wick2} that 
\begin{align}
\E \Big[ |\ft {\dot \ZZ}_{S,N} (0)|^2 \Big]
= \E\bigg[\Big(\sum_{\substack{n_1 \in \Z^3\\ n_1\in NQ}} \big(|\ft {\<1>}_{S,N}(n_1)|^2 - \jb{n_1}^{-2}\big)\Big)^2\bigg]
\les \sum_{n_1 \in \Z^3} \jb{n_1}^{-4}
\les 1, 
\label{SS1}
\end{align}

\noi
Hence, from \eqref{SS0} and \eqref{SS1},
we have
\begin{align*}
\E \bigg[ \int_0^1 & \jb{ \dot \ZZ_{S,N}(t), \dot\ZZ_{S,M}(t)}_{H^1_x}dt \bigg] 
= \int_0^1 \E \Bigg[ \sum_{n \in \Z^3} \jb{n}^2 
\ft {\dot \ZZ}_{S,N} (n, t)\cj{\ft {\dot\ZZ}_{S,M} (n, t)}  \bigg] dt \\
&=\int_0^1 \E \Bigg[ \sum_{n \in \Z^3\setminus \{0\}} \jb{n}^2 
\ft {\dot \ZZ}_{S,N} (n, t)\cj{\ft {\dot\ZZ}_{S,M} (n, t)}  \bigg] dt 
+ O (1).
\end{align*}

\noi
We now proceed as in the proof of \eqref{logtdiv} in Lemma \ref{LEM:Cor00}\,(i).
By  applying \eqref{Divv2} and   Lemma~\ref{LEM:Wick2}, 
and summing over  $\big\{|n| \le \frac{2}{3}N, 
\ \frac 14 |n| \leq |n_1| \leq \frac 12 |n|\big\}$
(which implies $|n_2| \sim |n|$ and $|n_j| \le  N$, $j = 1, 2$), 
we have 
\begin{align*}
\E & \Bigg[ \sum_{n \in \Z^3\setminus \{0\}} \jb{n}^2 
\ft {\dot \ZZ}_{S,N} (n, t)\cj{\ft {\dot\ZZ}_{S,M} (n, t)}  \bigg] \\
& = \sum_{\substack{n \in \Z^3\\ |n|\le \min\{N,M\} }}\frac{1}{\jb{n}^2} \int_{\T^3_x \times \T^3_y}
 \E\Big[  H_2(\<1>_{S,N}(x, t); \<tadpole>_{S,N}(t)) H_2(\<1>_{S,N}(y, t); \<tadpole>_{S,N}(t))\Big] e_n(y - x)dx dy\\
& =\sum_{\substack{n \in \Z^3\\ |n|\le \min\{N,M\} }}\frac{t^2 }{\jb{n}^2}
\sum_{\substack{n_1, n_2\in \Z^3\\|n_1| \le N, |n_2| \le M } }
\frac{1}{\jb{n_1}^2\jb{n_2}^2}
\int_{\T^3_x \times \T^3_y}
e_{n_1 + n_2 - n}(x-y) dx dy\\
& =\sum_{\substack{n \in \Z^3\\ |n|\le \min\{N,M\} }}\frac{t^2}{\jb{n}^2} \sum_{\substack{n = n_1 + n_2\\ |n_1|\le N, |n_2|\le M}}
\frac{1}{\jb{n_1}^2\jb{n_2}^2}
\sim t^2 \log N, 
\end{align*}

\noi
where $e_n(z):=e^{in\cdot z}$.
By integrating on $[0, 1]$, we obtain the desired bound \eqref{SNM}.

We now consider \eqref{WNM}. Note that 
\begin{align}
\ft {\dot \ZZ}_{W, N} (n)
&= \jb{n}^{-2} 
\sum_{\substack{n_1, n_2 \in \Z^3 \\ n=n_1+n_2}}
\ft {\<1>}_{S,N}(n_1) \ft {\<1>}_{W,N}(n_2).
\label{ZZ02}
\end{align}

\noi
Thanks to the independence between $\<1>_S$ and $\<1>_W$, we do not have to consider the pairing case (i.e. $n_1+n_2=0$). By exploiting \eqref{ZZ02} and Wick's theorem \ref{LEM:Wick}, we have
\begin{equation}
\begin{split}
\E & \Bigg[ \sum_{n \in \Z^3} \jb{n}^2 
\ft {\dot \ZZ}_{W,N} (n, t)\cj{\ft {\dot\ZZ}_{W,M} (n, t)}  \bigg]\\
& = \sum_{\substack{n \in \Z^3\\ |n|\le \min\{N,M\} }}\frac{t^2}{\jb{n}^2} \sum_{\substack{n_1, n_2 \in \Z^3 \\ n=n_1+n_2}} \sum_{\substack{m_1, m_2 \in \Z^3 \\ n=m_1+m_2}} \E\Big[\ft {\<1>}_{S,N}(n_1) \ft {\<1>}_{W,N}(n_2)\cj{\ft {\<1>}_{S,N}(m_1) \ft {\<1>}_{W,N}(m_2)   } \Big]   \\
& =\sum_{\substack{n \in \Z^3\\ |n|\le \min\{N,M\} }}\frac{t^2}{\jb{n}^2}
\sum_{\substack{n = n_1 + n_2\\ |n_1|\le N, |n_2|\le M}}
\frac{1}{\jb{n_1}^2\jb{n_2}^2}
\sim t^2 \log N.
\label{WWMN}
\end{split} 
\end{equation}

\noi
By integrating \eqref{WWMN} on $[0, 1]$, we have
\begin{align*}
\E \bigg[  \int_0^1 \jb{ \dot \ZZ_{W,N} (t), \dot \ZZ_{W,M} (t)}_{H^1_x}dt \bigg]  \sim \log N.
\end{align*}

\noi
This completes the proof of Proposition \ref{LEM:logdiv}.

\end{proof}

\subsection{Absolute continuity of Gibbs measures with respect to the shifted measure}
\label{SUBSEC:AC}
In this subsection, we constructe a shifted measure with respect to which the Gibbs measure $\rhoo$ established in Theorem \ref{THM:Gibbs} (i) is absolutely continuous.

\begin{proposition}\label{PROP:shift}
The Gibbs measure $\rhoo$ constructed in Theorem \ref{THM:Gibbs}\, \textup{(i)} is absolutely continuous with respect to the shifted measure
\begin{align}
\nuu=\Law (\vec {\<1>}(1) -\ld \vec \ZZ^r(1) + \vec \W(1)),
\label{shift00}
\end{align}

\noi
where $\vec {\<1>}=(\<1>_S, \<1>_W)$ is as in \eqref{CenGauss} with $\Law_\PP(\vec {\<1>}(1))=\muu$, $\vec \ZZ^r=(\ZZ_W, \ZZ_S)$\footnote{The superscript `r' means the reverse i.e. $\vec{\ZZ}=(\ZZ_{S}, \ZZ_{W})$.}
is defined as the limit of the antiderivative of $\dot {\vec \ZZ}^{r,N}=(\dot \ZZ^{N}_W, \dot \ZZ^N_S )$ in \eqref{Divv3} and \eqref{Divv2} as $N \to \infty$, and the auxiliary process $\vec \W=(\W_S, \W_W)$ is defined by 
\begin{align}
\W_S (t) = (1-\Delta)^{-1} \int_0^t \jb{\nabla}^{-\frac 12 - \eps} 
\big(\jb{\nabla}^{-\frac 12 - \eps} \<1>_S(t')\big)^{5} dt' \label{W0a}\\
\W_W(t)=  (1-\Delta)^{-1} \int_0^t \jb{\nabla}^{-\frac 12 - \eps} 
\big(\jb{\nabla}^{-\frac 12 - \eps} \<1>_W(t')\big)^{5} dt'
\label{W0b}
\end{align}

\noi
for some small $\eps > 0$.
\end{proposition}

\begin{remark}\rm
The coercive term $\vec \W$ is introduced  to  guarantee  global existence of a drift on the time interval $[0, 1]$. See Lemma \ref{LEM:coupeqn} below.
\end{remark}

We first state the following lemma, giving a criterion for absolute continuity. For the proof, see Lemma C.1  in \cite{OOT1} 

\begin{lemma} \label{LEM:AC1}
Let $\mu_n$ and $\rho_n$ be probability measures on a Polish space $X$.
Suppose that $\mu_n$ and $\rho_n$  converge weakly to $\mu$ and $\rho$, respectively.
Furthermore, suppose  that for every $\eps > 0$,
there exist $\delta(\eps) >0 $ and $\eta(\eps)>0 $ 
with $\delta(\eps)$, $\eta(\eps) \to 0$ as $\eps \to 0$ such that for every continuous function $F: X \to \R$ with $0 < \inf F \le F \le 1$ satisfying  
\[\mu_n(\{F \le \eps\}) \ge 1- \delta(\eps)\] 

\noi
for any $n \geq n_0 (F)$, 
we have
\begin{align}
\limsup_{n \to \infty} \int F(u) d \rho_n(u) \le \eta(\eps).
\notag
\end{align}
Then, $\rho$ is absolutely continuous with respect to $\mu$.
\end{lemma}

We take $\dot {\vec \ZZ}^{r,N}=(\dot \ZZ^N_W, \dot \ZZ^N_S)$ in \eqref{Divv3} and \eqref{Divv2} and $\vec \W=(\W_S, \W_W)$ in \eqref{W0a} and \eqref{W0b} as functions of $\vec {\<1>}=(\<1>_S,\<1>_W)$. Then, they can be written as  
\begin{align}
\dot \ZZ^N_S (\<1>_S) (t) &= (1-\Delta)^{-1} \<2>_{S,N}(t), \label{Z0} \\
\dot \ZZ^N_W (\<1>_S,\<1>_W) (t) &= (1-\Delta)^{-1}\big( \<1>_{S,N}(t)\<1>_{W,N}(t) \big) , \label{Z1} \\
\W_S (\<1>_S) (t) &= (1-\Delta)^{-1} \int_0^t \jb{\nabla}^{-\frac 12 - \eps} 
\big(\jb{\nabla}^{-\frac 12 - \eps} \<1>_S(t')\big)^{5} dt' \notag\\
\W_W (\<1>_W) (t) &= (1-\Delta)^{-1} \int_0^t \jb{\nabla}^{-\frac 12 - \eps} 
\big(\jb{\nabla}^{-\frac 12 - \eps} \<1>_W(t')\big)^{5} dt' \notag
\end{align}

\noi
and we set $\dot {\vec \ZZ}_{N}^r  (\vec {\<1>})  = \pi_N \dot {\vec \ZZ}^{r,N} (\vec {\<1>})$.
Then, from \eqref{Z0} and \eqref{Z1}, we have 
\begin{align}
&\dot \ZZ_{W,N} (\<1>_S + \Dr_S, \<1>_W + \Dr_W) - \dot \ZZ_{W,N}(\<1>_S,\<1>_W) \label{Jb00}\\
&\hphantom{XXXXXXXXXXXXXX}= (1-\Delta)^{-1} \pi_N (
\vec {\<1>}_N \cdot \vec \Dr_N^r+\Dr_{S,N}\Dr_{W,N}), \notag \\
&\dot \ZZ_{S,N} (\<1>_S + \Dr_S) - \dot \ZZ_{S,N} (\<1>_S)= (1-\Delta)^{-1} \pi_N ( 2\Theta_{S,N} \<1>_N + \Theta_{S,N}^2),
\label{Jb0}
\end{align}

\noi
where $\vec {\<1>}_N \cdot \vec \Dr_N^r=(\<1>_{S,N},\<1>_{W,N})\cdot (\Dr_{W,N}, \Dr_{S,N})=\<1>_{S,N}\Dr_{W,N} +\<1>_{W,N}\Dr_{S,N} $ and  $\vec \Dr_N = \pi_N \vec \Dr$. We also define $\vec \W_N(\vec {\<1>})(t)$ by 
\begin{align}
\W_{S,N} (\<1>_S) (t) &= (1-\Delta)^{-1} \pi_N \int_0^t \jb{\nabla}^{-\frac 12 - \eps} 
\big(\jb{\nabla}^{-\frac 12 - \eps} \<1>_{S,N} (t')\big)^{5} dt' \label{WSN00},\\
\W_{W,N} (\<1>_W) (t) &= (1-\Delta)^{-1} \pi_N \int_0^t \jb{\nabla}^{-\frac 12 - \eps} 
\big(\jb{\nabla}^{-\frac 12 - \eps} \<1>_{W,N} (t')\big)^{5} dt' .
\label{WSN01}
\end{align}

Next, we state a lemma on the construction of  a drift $\vec \Dr=(\Dr_S,\Dr_W)$.

\begin{lemma} \label{LEM:coupeqn}
Let $\ld \in \R$ and  $\dot {\vec \Upsilon}=(\dot \Upsilon_S, \dot \Upsilon_W) \in L^2 ([0,1]; \vec H^1(\T^3))$.
Then, given any $N \in \N$, the Cauchy problem for $\vec \Dr=(\Dr_S,\Dr_W)$\textup{:}
\begin{equation} 
\begin{cases}
\dot \Dr_S - \ld (1-\Delta)^{-1} \pi_N (\vec {\<1>}_N\cdot \vec \Dr_N^r+\Dr_{S,N}\Dr_{W,N})+  \dot \W_{S,N}(\<1>_{S}+\Dr_S) - \dot \Upsilon_S = 0,\\
\dot \Dr_W - \ld (1-\Delta)^{-1} \pi_N (2 \Dr_{S,N} \<1>_{S,N} + \Dr_{S,N}^2)+  \dot \W_{W,N}(\<1>_{W}+\Dr_W) - \dot \Upsilon_W = 0, \\
(\Dr_S(0), \Dr_W(0)) = \vec 0,
\end{cases}
\label{MPI0}
\end{equation}

\noi
where $\vec {\<1>}_N \cdot \vec \Dr_N^r=\<1>_{S,N}\Dr_{W,N} +\<1>_{W,N}\Dr_{S,N}$, is almost surely globally well-posed on the time interval $[0, 1]$ such that a solution $\vec \Dr$ belongs to  $ C([0, 1]; \vec H^1(\T^3))$. Moreover, if  $\|\dot{ \vec \Upsilon }\|^2_{L^2([0,\tau]; \vec H^1_x)} \le M$ for some $M>0$
and for some stopping time $\tau \in [0, 1]$, 
then, for any $ 1 \le p <\infty$, there exists $C= C(M,p)>0$ such that 
\begin{equation}
\E \Big[ \|\dot {\vec \Dr} \|_{L^2([0,\tau]; \vec H^1_x)}^p \Big]
 \le C(M,p),
\label{MP2}
\end{equation}

\noi
where $C(M,p)$ is independent of $N \in \N$.
\end{lemma}

By taking Lemma~\ref{LEM:coupeqn} for granted, we first prove Proposition \ref{PROP:shift}. We present the proof of Lemma~\ref{LEM:coupeqn} at the end of this subsection.
For simplicity, we use the same short-hand notations as in 
Sections \ref{SEC:Construc}, for instance,
$\vec {\<1>}=\vec {\<1>}(1)$, $\vec \ZZ = \vec \ZZ^r(1)$, $\vec \W = \vec \W (1)$, and $\vec \W_N =  \vec \W_N(1)$.

Given $L \gg1$, let $\dl(L)$ and $R(L)$ satisfy $\dl (L) \to 0$ and $R(L) \to \infty$ as $L \to \infty$, which will be specified later.
In view of Lemma \ref{LEM:AC1}, it suffices to show that if $G: \vec {\mathcal{C}}^{-100}(\T^3) \to \R$ is a bounded continuous function
with $G > 0$ and
\begin{align}
\PP \big(\{G(\vec {\<1>} -\ld \vec \ZZ^r_N+ \vec \W_N) \ge L \}\big) \ge 1 - \delta(L),
\label{assup00}
\end{align}

\noi
then we have
\begin{align}
\limsup_{N \to \infty} \int \exp(- G(u,w)) d \rhoo_N (u,w) \le \exp(-R (L)), 
\label{result0}
\end{align}

\noi
where $\rhoo_N$ denotes the truncated Gibbs measure defined in \eqref{truGibbsN}.
Here, we regard $\Law(\vec {\<1>} -\ld \vec \ZZ_N^r+ \vec \W_N)$ as the measure  $\muu_N$, weakly 
converging to $\muu = \Law(\vec {\<1>} -\ld \vec \ZZ^r+ \vec \W)$.

Thanks to \eqref{cut000}, Bou\'e-Dupuis variational formula (Lemma \ref{LEM:BoueDupu}), and 
the change of variables \eqref{Divv4}, \eqref{Divv5}, we have
\begin{align*}
- \log &  \bigg( \int \exp\big(- G(u,w) - \mathcal{H}_N^\dia(u,w) \big) d \muu (u,w) \bigg) \\
&=\inf_{\dot {\vec \Upsilon}^N \in  \vec {\mathbb {H}}_a^1}   \E \bigg[ 
G(\vec {\<1>}+ \vec \Upsilon^N-\ld \vec \ZZ_N^r)
+ \ft {\mathcal{H}}^{\dia}_N (\vec {\<1>}+\vec \Upsilon^N -\ld \vec \ZZ_N^r) + \frac1{2} \int_0^1 \| \dot {\vec \Upsilon}^N (t) \|_{\vec H^1_x}^2dt \bigg],
\end{align*}

\noi
where $\ft {\H}^{\dia}_N$ is as in \eqref{PUBAO1}.
By using Lemmas \ref{LEM:Cor0} and \ref{LEM:Cor1} with Lemma \ref{LEM:Cor00}, \eqref{estiZZ}, and the smallness of the coupling constant $|\ld|$ as in Subsection \ref{SUBSEC:uniform} (see also \eqref{drift1}, \eqref{coer0}, and \eqref{drift2}), we have
\begin{align}
\begin{split}
- \log &  \bigg( \int \exp(- G(u,w) - \mathcal{H}_N^\dia(u,w)) d \muu (u,w)\bigg) \\
&\quad \ge   \inf_{\dot {\vec \Upsilon}^N \in  \vec {\mathbb {H}}_a^1}    \E \bigg[ 
G(\vec {\<1>}+ \vec \Upsilon^N-\ld \vec \ZZ_N^r)
+ \frac1{20} \int_0^1 \| \dot {\vec \Upsilon}^N (t) \|_{\vec H^1_x}^2dt  \bigg]
-C_1
\end{split}
\label{BBB0}
\end{align}

\noi
for some constant $C_1 > 0$.
For $  \dot {\vec \Upsilon}^N  \in  \vec{\mathbb{H}}^1_a$, let $\vec \Dr^N$ be the solution to \eqref{MPI0} with $\dot {\vec \Upsilon} $ replaced by $\dot {\vec \Upsilon}^N $.
For any $M>0$, define the stopping time $\tau_M$ as 
\begin{align}
\begin{split}
\tau_M
&=\min\bigg(1,  \, \min \bigg\{ \tau : \int_0^\tau \| \dot {\vec \Upsilon}^N (t) \|_{\vec H^1_x}^2dt   = M \bigg\}, \\
&\hphantom{XXXXXX}
\min \bigg\{ \tau : \int_0^\tau \|\dot {\vec \Dr}^N (t)  \|_{\vec H^1_x}^2 dt = 2C(M,2)\bigg\}\bigg),
\end{split}
\label{stoptime0}
\end{align}

\noi
where $C(M,2)$ is the constant appearing in \eqref{MP2} with $p=2$.
Let
\begin{align}
\vec \Dr^N_M(t) := \vec \Dr^N (\min(t,\tau_M)).
\label{stop}
\end{align}

\noi
It follows from \eqref{CenGauss} that we have  $\vec {\<1>}(0) = 0$, while 
$\vec \ZZ^N(0) = 0$ by definition.
Then, from the change of variables~\eqref{Divv4} and \eqref{Divv5} with $\vec \Dr (0) = 0$, 
we see that $\vec \Upsilon^N(0) = 0$.
We also have $\vec \W_N(0) = 0$ from \eqref{WSN00} and \eqref{WSN01}.
Then, substituting  \eqref{Jb0} and \eqref{Jb00} into \eqref{MPI0} 
and integrating from $t = 0$ to $1$ gives 
\begin{align}
\vec {\<1>} + \vec \Upsilon^N -\ld  \vec\ZZ_N^r
= \vec {\<1>} +  \vec  \Dr^N_M -\ld  \vec \ZZ_N^r(\vec {\<1>} + \vec \Dr^N_M) + \vec \W_N(\vec {\<1>} + \vec \Dr^N_M)
\label{linrela0}
\end{align}

\noi
on the set $\{\tau_M = 1\}$.

Thanks to the definition \eqref{stop} with 
\eqref{stoptime0}, we have 
\begin{align}
\| \dot {\vec \Dr}^N_M(t) \|_{L^2_t([0,1];H^1_x)}^2 \le 2C(M,2)
\label{BB0}
\end{align}

\noi
and thus the Novikov condition is satisfied.
Then, Girsanov's theorem \cite[Theorem 10.14]{DZ}
yields that $\Law(\vec {\<1>} + \vec \Theta_M^N)$ is absolutely continuous with respect to $\Law (\vec {\<1>})$; see \eqref{Girsano0} below. Let $\Q = \Q^{\dot {\vec \Theta}_M^N}$  the probability measure whose  Radon-Nikodym derivative with respect to $\PP$ is given by the following stochastic exponential:
\begin{equation}
\frac{d\Q}{d\PP} = e^{ - \int_0^1 \jb{ \dot {\vec \Theta}_M^N (t),  d\vec {\<1>}(t)}_{\vec H^1_x} 
- \frac{1}{2} \int_0^1 \| \dot  {\vec \Theta}_M^N (t) \|_{\vec H^1_x}^2dt}
\label{expoproc}
\end{equation}

\noi
such that, under this new measure $\Q$, 
the process  
\begin{align*}
\vec W^{ \dot { \vec {\Dr}}_M^N}(t) = \vec W(t) +   \jb{\nb}  \dot {  \vec {\Dr}}_M^N(t)
=   \jb{\nb} (\vec {\<1>} + \dot  {\vec  \Dr}_M^N)(t)
\end{align*}

\noi
is a cylindrical Wiener process on $\vec L^2(\T^3)$.
By setting  $\vec {\<1>}^{ \dot  {\vec \Theta}_M^N} (t) = \jb{\nb}^{-1}\vec W^{ \dot {\vec   \Dr}_M^N} (t)$, we have 
\begin{align}
 \vec {\<1>}^{ \dot {\vec \Dr}_M^N}(t) = \vec {\<1>}(t) + \dot {\vec \Dr}_M^N(t).
\label{BB1}
\end{align}

\noi
Moreover, from Cauchy-Schwarz inequality with \eqref{expoproc}
and the bound~\eqref{BB0}, and then~\eqref{BB1}, we have 
\begin{align}
\begin{split}
\PP\big(\{\vec {\<1>} + \vec \Dr_M^N \in E\}\big) 
& = \int \ind_{\{ \vec {\<1>}+  \vec \Dr_M^N\in E\}} \frac{d \PP}{d\Q} d\Q
 \le C_M \Big(\Q\big(\{  \vec {\<1>}^{ \dot {\vec  \Dr}_M^N} \in E\}\big)\Big)^\frac12 \\
& =   C_M \Big(\PP\big(\{\vec {\<1>} \in E\}\big)\Big)^\frac12
\end{split}
\label{Girsano0}
\end{align}

\noi
for any measurable set $E$.

From \eqref{BB0}, \eqref{linrela0}, and  the non-negativity of $G$, 
we have
\begin{align}
\eqref{BB0} &\ge \inf_{\dot {\vec \Upsilon}^N \in  \vec {\mathbb {H}}_a^1}\E \bigg[ 
\Big( G \big( \vec {\<1>}+  \vec \Dr^N_M -\ld \vec \ZZ_N^r(\vec {\<1>} + \vec \Dr^N_M) + \vec \W_N(\vec {\<1>} + \vec \Dr^N_M) \big)\notag \\
&\phantom{XXXXXX}
+ \frac1{20} \int_0^1 \| \dot {\vec \Upsilon}^N (t) \|_{\vec H^1_x}^2dt\Big)\ind_{\{ \tau_M = 1 \}}\phantom{ ]} \notag \\
&\phantom{XXXXXX}
+\Big( G(\vec {\<1>}+ \vec \Upsilon^N-\ld \vec \ZZ_N^r)
+ \frac1{20} \int_0^1 \| \dot {\vec \Upsilon}^N (t) \|_{\vec H^1_x}^2dt\Big)\ind_{\{ \tau_M < 1 \}} \bigg]
-C_1 \notag \\
&\ge \inf_{\dot {\vec \Upsilon}^N \in  \vec {\mathbb {H}}_a^1}   \E \bigg[ 
G \big( \vec {\<1>}+  \vec \Dr^N_M -\ld \vec \ZZ_N^r(\vec {\<1>} + \vec \Dr^N_M) + \vec \W_N(\vec {\<1>} + \vec \Dr^N_M) \big) \cdot \ind_{\{ \tau_M = 1 \}}\phantom{]} \notag \\
&\phantom{XXXXXX}
+ \frac1{20} \int_0^1 \| \dot {\vec \Upsilon}^N (t) \|_{\vec H^1_x}^2dt \cdot \ind_{\{ \tau_M < 1 \}} \bigg] - C_1.
\notag
\end{align}

\noi
By using the definition \eqref{stoptime0} of the stopping time $\tau_M$ 
and exploiting \eqref{Girsano0} and \eqref{assup00}, we have 
\begin{align}
\eqref{BB0} &\ge \inf_{\dot {\vec \Upsilon}^N \in  \vec {\mathbb {H}}_a^1}
\E \bigg[ L \cdot \ind_{\{\tau_M = 1\} 
\cap \{ G(\vec {\<1>}+  \vec \Dr^N_M -\ld \vec \ZZ_N^r(\vec {\<1>} + \vec \Dr^N_M) + \vec \W_N(\vec {\<1>} + \vec \Dr^N_M)) \ge L \}} \notag \\
&\phantom{XXXXXX}
+ \frac M {20} \cdot \ind_{\{\tau_M < 1\} \cap \{\int_0^1 \| \dot {\vec \Dr}^N_M (t) \|_{H^1_x}^2 dt < 2C(M,2)\}} \bigg]  - C_1 \notag \\
&\ge \inf_{\dot {\vec \Upsilon}^N \in  \vec {\mathbb {H}}_a^1}
\Bigg\{ L \Big( \PP(\{\tau_M = 1\}) - C_M \delta (L)^ \frac 12 \Big) \notag \\
&\hphantom{XXXXX}
+ \frac M {20} \PP\bigg(\{\tau_M < 1\} \cap \bigg\{\int_0^1 \| \dot {\vec \Dr}^N_M (t) \|_{H^1_x}^2 dt < 2C(M,2)\bigg\}\bigg) \Bigg\} - C_1.
\label{BBB1}
\end{align}

\noi
In view of  \eqref{MP2} with \eqref{stoptime0} and \eqref{stop}, 
Markov's inequality gives 
\begin{align*}
\PP \bigg( \int_0^1 \| \dot {\vec \Dr}^N_M (t) \|_{\vec H^1_x}^2 dt
= \int_0^{\tau_M} \| \dot {\vec \Dr}^N_M (t) \|_{\vec H^1_x}^2 dt \ge 2C(M,2) \bigg) \le  \frac 12,
\end{align*}

\noi
which yields
\begin{align}
\PP \bigg( \{\tau_M < 1\} \cap \bigg\{\int_0^1 \| \dot {\vec \Dr}^N_M (t) \|_{H^1_x}^2dt < 2C(M,2) \bigg\} \bigg) \ge \PP(\{\tau_M < 1\})- \frac 12. 
\label{BBB2}
\end{align}

\noi
Now, we set   $M = 20L$.
Note  from \eqref{stoptime0} that $\PP(\{\tau_M = 1\})+ \PP(\{\tau_M < 1\}) = 1$.
Then, from~\eqref{BBB1} and \eqref{BBB2}, we obtain
\begin{align*}
- \log & \bigg( \int \exp(- G(u,w)- \H^{\dia}_N(u,w)) d \muu (u,w) \bigg) \\
&\ge  \inf_{\dot {\vec \Upsilon}^N \in  \vec {\mathbb {H}}_a^1} \bigg\{
L \Big( \PP(\{\tau_M = 1\}) - C'_{L} \delta(L)^ \frac 12 \Big)
+ L \Big( \PP(\{\tau_M < 1\})- \frac 12 \Big) \bigg\} - C_1 \\
&= L \Big( \frac 12 - C_{L}' \delta(L)^\frac 12 \Big) - C_1.
\end{align*}

\noi
Therefore, 
by choosing $\delta(L)>0$ such that $C'_{L} \delta(L)^\frac12 \to 0$ as $L \to \infty$,
this shows \eqref{result0} with 
\begin{align*}
R(L) = L \Big( \frac 12 - C'_{L} \delta(L)^\frac 12 \Big) - C_1 + \log Z, 
\end{align*}

\noi
where $Z = \lim_{N \to \infty} Z_N$ denotes the 
limit of the partition functions  for the truncated Gibbs measures $\rhoo_N$.

\smallskip

We conclude this subsection  by presenting the proof of Lemma \ref{LEM:coupeqn}.

\begin{proof}[Proof of Lemma \ref{LEM:coupeqn}]
By Lemma \ref{LEM:KKS0} (ii), Sobolev's inequality, and Young's inequality, we have
\begin{align}
\begin{split}
\| (1-&\Delta)^{-1} (\vec {\<1>}_N \cdot \vec \Dr^r_N+\Dr_{S,N}\Dr_{W,N}) \|_{H^{1}_x}\\
&\les \| (\<1>_{S,N}\Dr_{W,N}+\<1>_{W,N}\Dr_{S,N}+\Dr_{S,N}\Dr_{W,N}) \|_{H^{-1}_x} \\
&\les \| \Dr_{S,N} (t) \|_{H^{\frac 12+\eps}_x} \| \<1>_{W,N} (t) \|_{W^{-\frac 12-\eps, \infty}_x}+ \| \Dr_{W,N} (t) \|_{H^{\frac 12+\eps}_x} \| \<1>_{W,N} (t) \|_{W^{-\frac 12-\eps, \infty}_x}\\
&\hphantom{XXXXXXXXXXXXXXXXXXXX}+\|\Dr_{S,N} \|_{L^4} \| \Dr_{W,N}\|_{L^4} \\
&\les  \| \Dr_{S,N} (t) \|_{H^{1}_x} \| \<1>_{W,N} (t) \|_{W^{-\frac 12-\eps, \infty}_x}+ \| \Dr_{W,N} (t) \|_{H^{1}_x} \| \<1>_{W,N} (t) \|_{W^{-\frac 12-\eps, \infty}_x}\\
&\hphantom{XXXXXXXXXXXXXXXXXXXX}+\| \Dr_{S,N} (t) \|_{H_x^1}^2+\| \Dr_{W,N} (t) \|_{H_x^1}^2
\label{MI1}
\end{split}
\end{align}

\noi
and
\begin{align}
\begin{split}
\| (1-\Delta)^{-1} ( 2\Theta_{S,N} \<1>_{S,N} + \Dr_{S,N}^2)(t) \|_{H^{1}_x}
&\les 
\| ( 2\Dr_{S,N} \<1>_{S,N} + \Dr_{S,N}^2)(t) \|_{H^{-1}_x} \\
&\les
\| \Dr_{S,N} (t) \|_{H^{\frac 12+\eps}_x} \| \<1>_{S,N} (t) \|_{W^{-\frac 12-\eps, \infty}_x}
+ \| \Dr_{S,N}^2 (t) \|_{L_x^{4}}^2 \\
&\les
\| \Dr_{S,N} (t) \|_{H^1_x} \| \<1>_{S,N} (t) \|_{W^{-\frac 12-\eps, \infty}_x}
+ \| \Dr_{S,N} (t) \|_{H_x^1}^2
\end{split}
\label{MI0}
\end{align}

\noi
for small $\eps>0$. Moreover, from \eqref{W0a} and \eqref{W0b}, we have
\begin{align}
\begin{split}
\| \dot \W_{S,N}(\<1>_S(t)+\Dr_S(t)) \|_{H^1_x}
&\les
\| \jb{\nb}^{-\frac 12-\eps}\<1>_{S,N}(t) \|_{L^\infty_x}^{5}
+ \| \jb{\nb}^{-\frac 12-\eps} \Dr_{S,N} (t)\|_{L^\infty_x}^{5} \\
&\les 
\| \<1>_{S,N}(t) \|_{W^{-\frac 12-\eps, \infty}_x}^{5}
+ \| \Dr_{S,N} (t)\|_{H^1_x}^{5}
\end{split}
\label{MI2}
\end{align}

\noi
and
\begin{align}
\begin{split}
\| \dot \W_{W,N}(\<1>_W(t)+\Dr_W(t)) \|_{H^1_x}
&\les \| \jb{\nb}^{-\frac 12-\eps}\<1>_{W,N}(t) \|_{L^\infty_x}^{5}
+ \| \jb{\nb}^{-\frac 12-\eps} \Dr_{W,N} (t)\|_{L^\infty_x}^{5} \\
&\les  \| \<1>_{W,N}(t) \|_{W^{-\frac 12-\eps, \infty}_x}^{5}
+ \| \Dr_{W,N} (t)\|_{H^1_x}^{5}.
\end{split}
\label{MI3}
\end{align}

\noi
Therefore, by studying the integral formulation of \eqref{MPI0}, a contraction argument 
in $L^\infty([0,T]; \vec H^1(\T^3))$ for small $T>0$ with  \eqref{MI1}, \eqref{MI0},  \eqref{MI2}, and \eqref{MI3} yields local well-posedness. Here, the local existence time $T$ depends on  $\|\dot {\vec  \Upsilon} \|_{L^2_T  \vec H^1_x}$,   and 
$\| \vec {\<1>}_N \|_{L^6_T \vec W^{-\frac 12-\eps, \infty}_x}$,  where the last term is almost surely bounded in view of Lemma \ref{LEM:Cor00} and \eqref{embed}.

Next, we prove global existence on $[0, 1]$ by establishing an a priori bound on the $\vec H^1$-norm of a solution. From  \eqref{MPI0} with \eqref{WSN00} and \eqref{WSN01}, we have
\begin{align}
\begin{split}
\frac 12 \frac{d}{dt} \| \Dr_S(t) \|_{H^1}^2
&=  \ld \int_{\T^3} \big( \vec {\<1>}_N(t) \cdot \vec \Dr^r_N(t) + \Dr_{S,N}(t)\Dr_{W,N}(t) \big) \Dr_{W,N}(t) dx   \\
&\quad
- \int_{\T^3} \big(\jb{\nabla}^{-\frac 12 - \eps} 
(\<1>_{S,N}(t)+ \Dr_{S,N}(t))\big)^{5}\cdot \jb{\nb}^{-\frac 12 - \eps} \Dr_{S,N}(t) dx \\
&\quad
+ \int_{\T^3} \jb{\nabla} \Dr_S(t) \cdot \jb{\nb} \dot \Upsilon_S(t) dx
\label{MI4}
\end{split}
\end{align}

\noi
and
\begin{align}
\begin{split}
\frac 12 \frac{d}{dt} \| \Dr_W(t) \|_{H^1}^2
&= \ld \int_{\T^3} ( 2\Theta_{S,N}(t) \<1>_{S,N}(t) + \Dr_{S,N}^2(t)) \Dr_{S,N}(t) dx \\
&\quad - \int_{\T^3} \big(\jb{\nabla}^{-\frac 12 - \eps} 
(\<1>_{W,N}(t)+ \Dr_{W,N}(t))\big)^{5}\cdot \jb{\nb}^{-\frac 12 - \eps} \Dr_{W,N}(t) dx \\
&\quad
+ \int_{\T^3} \jb{\nabla} \Dr_W(t) \cdot \jb{\nb} \dot \Upsilon_W(t) dx. 
\label{MI5}
\end{split}
\end{align}

\noi
The second term on the right-hand side of \eqref{MI4} and \eqref{MI5}, coming from $\vec \W=(\W_S,\W_W)$ is a coercive term, allowing us to hide part of the first term on the right-hand side.

From Lemma \ref{LEM:KCKON0}
and  Young's inequality, we have
\begin{align}
\begin{split}
\bigg|\int_{\T^3} & \big( \vec {\<1>}_N \cdot \vec \Dr_N^r+ \Dr_{S,N}(t)\Dr_{W,N}(t) \big) \Dr_{W,N}(t) dx \bigg|\\
& \les \|\Dr_{S,N}(t)\|_{L^3}^3+\|\Dr_{W,N}(t)\|_{L^3}^3+\| \<1>_{S,N} (t) \|_{\mathcal{C}^{-\frac 12-\eps}}^c+\| \<1>_{W,N} (t) \|_{\mathcal{C}^{-\frac 12-\eps}}^c
\label{MI7}
\end{split}
\end{align}

\noi
and
\begin{align}
\begin{split}
\bigg| \int_{\T^3} &   ( 2\Theta_{S,N}(t) \<1>_{S,N}(t) + \Dr_{S,N}^2(t)) \Dr_{S,N}(t)  dx \bigg|\\
& \les
\| \Theta_{S,N} (t) \|_{H^1}^2 + \|\Dr_{S,N}(t)\|_{L^3}^3
+ \| \<1>_{S,N} (t) \|_{\mathcal{C}^{-\frac 12-\eps}}^c
\end{split}
\label{MI6}
\end{align}

\noi
for small $\eps > 0$ and some $c>0$.
We now estimate $L^3$-norm on the right-hand side of~\eqref{MI7} and \eqref{MI6}.
Thanks to  \eqref{INT0P}, we have
\begin{align}
\begin{split}
\| \Dr_{S,N}(t) \|_{L^3}^3
& \les \| \Dr_{S,N}(t) \|_{H^1}^{\frac{3 + 6\eps}{3 + 2\eps}}
 \|  \Dr_{S,N}(t) \|_{W^{-\frac 12 - \eps, 6}}^{\frac{6}{3 + 2\eps}}
\\
&\le \| \Dr_{S,N}(t) \|_{H^1}^2 + \eps_0 \| \Dr_{S,N}(t) \|_{W^{-\frac 12 - \eps, 6}}^{6} 
+ C_{\eps_0}
\end{split}
\label{MI8}
\end{align}

\noi
and
\begin{align}
\begin{split}
\| \Dr_{W,N}(t) \|_{L^3}^3
&\le \| \Dr_{W,N}(t) \|_{H^1}^2 + \eps_0 \| \Dr_{W,N}(t) \|_{W^{-\frac 12 - \eps, 6}}^{6} 
+ C_{\eps_0}
\end{split}
\label{MI9}
\end{align}

\noi
for small $\eps, \eps_0 > 0$. As for the coercive term, from \eqref{YoungaJ} and Young's inequality,  we have 
\begin{align}
\begin{split}
\int_{\T^3}&  \big(\jb{\nabla}^{-\frac 12 - \eps} 
(\<1>_{S,N}(t)+ \Dr_{S,N}(t))\big)^{5}\jb{\nb}^{-\frac 12 - \eps} \Dr_{S,N}(t) dx \\
&\ge
\frac 12 \int_{\T^3} (\jb{\nb}^{-\frac 12 - \eps} \Dr_{S,N}(t))^{6} dx
- c \int_{\T^3} \big|(\jb{\nb}^{-\frac 12 - \eps} \<1>_{S,N}(t))^{5}
\jb{\nb}^{-\frac 12 - \eps} \Dr_{S,N}(t)\big| dx \\
& \ge 
\frac 12 \| \Dr_{S,N}(t) \|_{W^{-\frac 12 - \eps, 6}}^{6}
- c \| \<1>_{S,N}(t) \|_{W^{-\frac12-\eps, 6}}^{5} \| \Dr_{S,N}(t) \|_{W^{-\frac12-\eps, 6}}\\
& \ge \frac 14  \| \Dr_{S,N}(t) \|_{W^{-\frac 12 - \eps, 6}}^{6}
- c \| \<1>_{S,N}(t) \|_{W^{-\frac12-\eps, 6}}^{6},
\end{split}
\label{MI10}
\end{align}

\noi
and
\begin{align}
\int_{\T^3}& \big(\jb{\nabla}^{-\frac 12 - \eps} 
(\<1>_{W,N}(t)+ \Dr_{W,N}(t))\big)^{5}\cdot \jb{\nb}^{-\frac 12 - \eps} \Dr_{W,N}(t) dx \notag \\
& \ge \frac 14  \| \Dr_{W,N}(t) \|_{W^{-\frac 12 - \eps, 6}}^{6} - c \| \<1>_{W,N}(t) \|_{W^{-\frac12-\eps, 6}}^{6}.
\label{MI11}
\end{align}

\noi
Therefore, putting \eqref{MI4}, \eqref{MI5}, \eqref{MI7}, \eqref{MI6}, \eqref{MI8},  \eqref{MI9}, \eqref{MI10}, and \eqref{MI11} together, we obtain
\begin{align*}
\frac{d}{dt} \|  \vec \Dr(t) \|_{\vec H^1}^2
& \les \| \vec \Dr (t) \|_{\vec H^1}^2
+ \|\dot  {\vec \Upsilon}(t)\|_{\vec H^1}^2
 +  \| \vec {\<1>} (t) \|_{\vec {\mathcal C}^{-\frac 12-\eps}}^c + \| \vec {\<1>}(t) \|_{\vec W^{-\frac12-\eps, 6}}^{6} + 1.
\end{align*}

\noi
By Gronwall's inequality, we then obtain
\begin{align}
\| \vec \Dr(t) \|_{\vec H^1}^2 \les \|\dot {\vec  \Upsilon}\|_{L^2([0,1];\vec H^1_x)}^2
+ \| \vec  {\<1>}_N \|_{L^c([0,1]; \vec {\mathcal C}^{-\frac 12-\eps}_x)}^c
+ \| \vec {\<1>} \|_{L^{6}([0,1];\vec W^{-\frac12-\eps, 6}_x)}^{6} + 1, 
\label{MI12}
\end{align}

\noi
uniformly in $ 0 \le t \le 1$.
The a priori bound \eqref{MI12} together with Lemma \ref{LEM:Cor00} allows us to iterate
the local well-posedness argument, 
guaranteeing  the global existence of the solution $\vec \Dr$ on $[0, 1]$.

Lastly, we prove the bound \eqref{MP2}. From \eqref{MI1}, \eqref{MI0}, \eqref{MI2}, \eqref{MI3}, and \eqref{MI12}, we have 
\begin{align}
\begin{split}
\| \ld (1 & -\Delta)^{-1} ( \vec {\<1>}_N \cdot \vec \Dr_N^r+\Dr_{S,N}\Dr_{W,N}) + \dot \W_{S,N}(\<1>_{S}+\Dr_S)  \|_{L^2([0,\tau];H^1_x)}\\
& \les \| \dot {\vec  \Upsilon} \|_{L^2 ([0,\tau]; \vec H^1_x)}^5 +
\| \vec {\<1>}_N \|_{L^q([0,1]; \vec{ \mathcal{C}}^{\hspace{1mm}-\frac 12-\frac 12 \eps}_x)}^{c_0} + 1
\end{split}
\label{MI14}
\end{align}

\noi
and
\begin{align}
\begin{split}
\| \ld (1 & -\Delta)^{-1} ( 2\Dr_{S,N} \<1>_{S,N} + \Dr_{S,N}^2) + \dot \W_{W,N}(\<1>_W+\Dr_W) \|_{L^2([0,\tau];H^1_x)}\\
& \les \| \dot {\vec  \Upsilon} \|_{L^2 ([0,\tau]; \vec H^1_x)}^5 +
\| \vec {\<1>}_N \|_{L^q([0,1]; \vec{ \mathcal{C}}^{\hspace{1mm}-\frac 12-\frac 12 \eps}_x)}^{c_0} + 1
\end{split}
\label{MI13}
\end{align}

\noi
for some finite $q, c_0 \geq 1$
and for any $0 \le \tau \leq 1$.
Then, using the equation \eqref{MPI0}, the bound~\eqref{MP2} follows from \eqref{MI14} and \eqref{MI13}, the bound on $\dot \Upsilon$, and the following corollary to  Lemma \ref{LEM:Cor00}:
\begin{align*}
\E \Big[ \|  \vec {\<1>}_{N} \|_{L^q([0,1]; \vec{\mathcal{C}}^{\hspace{1mm}-\frac 12-\frac 12 \eps}_x)}^p \Big]
< \infty
\end{align*}

\noi
for any finite $p, q \geq 1$,  uniformly in $N \in \N$.
\end{proof}

\section{Non-construction of Gibbs measures in the strong coupling regime}
\label{SEC:non}

In this section, we study non-construction of Gibbs measures in the strong coupling regime $|\ld| \gg 1$, as stated in Theorem \ref{THM:Gibbs}\,(ii). As explained in the introduction (subsection \ref{SUBSEC:renorm}), the singularity of the Gibbs measure $\rhoo$ with respect to the Gaussian free field $\muu$ complicates the formulation of the non-construction  statement in Theorem \ref{THM:Gibbs}(ii) because there is no density for the Gibbs measure. If the truncated density $e^{- \mathcal{H}_N^\dia(u,w)}$ converges to the limiting density under the Gaussian free field $\muu$, as in the one- and two-dimensional cases, it is possible to define the $\sigma$-finite version of the Gibbs measures using the density as follows
\begin{align}
d\rhoo(u,w)=e^{- \H^\dia  (u,w)} \ind_{ \big\{|\int_{\T^3} \;: | u|^2 :\;   dx| \le K\big\}} d\muu(u,w).
\label{JY1M}
\end{align}

\noi
Then, it suffices to prove 
\begin{align}
\sup_{N \in \N} \E_{\muu} \Big[ e^{- \H_N^\dia  (u,w)} \ind_{ \big\{|\int_{\T^3} \; : | u|^2 : \;   dx| \le K\big\}} \Big] = \infty, 
\label{JY0M}
\end{align}

\noi
because \eqref{JY0M} indicates that there is no normalization constant that can make the limiting measure \eqref{JY1M} into a probability measure. In the current situation, however, the corresponding density $e^{-\H_N^\dia(u,w)}$ does not converge to any limit under the Gaussian free field $\muu$.

The overall strategy of the proof of Theorem \ref{THM:Gibbs}\, (ii) is based on \cite{OOT2}, where the non-construction of the grand-canonical Gibbs measure like \eqref{GibbsH} was studied. In our situation, however, the argument in \cite{OOT2} provides the non-construction result (Theorem \ref{THM:Gibbs}\, (ii)) for $K \gg 1$. To extend the result to every small $K > 0$, we need a more delicate argument in constructing the profile causing the blow-up in the variational formulation of the partition function. See the explanation in Remark \ref{REM:diff0}. We first summarize the method and state corresponding propositions. The first step is to define a $\s$-finite version of the Gibbs measure. To consider the $\s$-finite version of the Gibbs measure, we first need  a reference measure $\nuu_\dl$ defined as a weak limit of the following tamed version  of Gibbs measures $\nuu_{N,\dl}$ 
\begin{align}
d \nuu_{N, \dl} (u,w) = Z_{N, \dl}^{-1} e^{-\dl F(\pi_N u, \pi_N w) - \mathcal{H}_N^\dia(u,w)} \ind_{ \{|\int_{\T^3} :  u_N^2 :   dx| \le K\}}  d\muu(u,w)
\label{nons00}
\end{align}

\noi
for some appropriate taming function $F \ge 0 $ and $\dl>0$. See Proposition \ref{PROP:nutame}.  This allows to construct a $\sigma$-finite version of the Gibbs measure as follows:
\begin{align}
d\rhoo_\dl &= e^{\dl F(u,w)} d \nuu_\dl \label{ship0}\\
&= \lim_{N \to \infty} Z_{N, \dl}^{-1} \, e^{\dl F(u,w)} 
e^{-\dl F(\pi_N u,\pi_N w) - \mathcal{H}_N^\dia(u,w)}\ind_{ \{|\int_{\T^3} :  u_N^2 :   dx| \le K\}}  d\muu(u,w). 
\label{rho1}
\end{align}

\noi
At a conceptual level, $\dl F(u,w)$ in the exponent of \eqref{ship0} and $-\delta F(\pi_N u, \pi_N w)$ in the exponent of \eqref{nons00} are 
cancelled when taking the limit as $N \to \infty$, and so the right-hand side of~\eqref{rho1} formally looks like $Z_\dl^{-1} \lim_{N\to \infty} e^{-\H^\dia_N(u,w)} \ind_{ \{|\int_{\T^3} :  u_N^2 :   dx| \le K\}} d\muu$. We then show that this  $\s$-finite version $ \rhoo_\dl$ of the Gibbs measure in \eqref{rho1} cannot be normalized into a probability measure in the strong coupling regime $|\ld| \gg 1$ (Proposition \ref{PROP:Nonnor}) 
\begin{align}
\int 1\, d  \rhoo_\dl=\infty
\label{nons0}
\end{align}

\noi
for every $\dl>0$. Then, as a corollary to this non-normalizability  result \eqref{nons0}, it is possible to prove that
the sequence $\{\rhoo_N\}_{N \in \N}$ of the truncated Gibbs measures defined in \eqref{truGibbsN} does not converge weakly, even up to a subsequence, in a natural space
$\vec \A$ for the Gibbs measure $\rhoo$. See Proposition \ref{PROP:result2}.

We now discuss the construction of the reference measure $\nuu_\dl$ for the $\s$-finite version of the Gibbs measure. Let $p_t$ be the kernel of the heat semigroup $e^{t\Dl}$.
We set the following space $\A = \A(\T^3)$ by the norm
\begin{align}
\|  u \|_{\A} := \sup_{0 < t \le 1} \Big( t^{\frac38}\| p_t \ast u \|_{L^3(\T^3)} \Big)
\label{normA}
\end{align}

\noi
and
\begin{align}
\| (u,w) \|_{\vec \A}^{20} :=\| u \|_{\A}^{20} +\| w \|_{\A}^{20}.
\label{vecA}
\end{align}

\noi
Thanks to the Schauder estimate on $\T^3$ 
\begin{align}
\| p_t*  u \|_{L^q (\T^3)} \le C_{\al, p, q} \,  t^{-\frac{\al}2-\frac 32 (\frac1p-\frac1q)} \|\jb{\nabla}^{-\al} u \|_{L^p(\T^3)}
\label{Schau}
\end{align}

\noi
for any  $\al \geq 0$ and $1\le p\le q\le \infty$, we can check that $W^{-\frac{3}{4}, 3}(\T^3) \subset \A$. In particular, the space $\A$ contains the support of the  Gaussian free field $\mu$ on $\T^3$ and thus we have 
$\|  u \|_{\A} < \infty$, $\mu$-almost surely.
See also Lemma \ref{LEM:gaussA} below. To be motivated for the use of the norm $\A$, see Remark \ref{REM:natu} (i).

Given $N \in \N$, we set $(u_N, w_N) = (\pi_N u, \pi_N w)$. Then, given $\dl > 0$, $K>0$, and $N \in \N$, we define the tamed version $\nuu_{N,\delta}$ of the truncated Gibbs measure $\rhoo_N$ as follows 
\begin{align}
d \nuu_{N,\delta} (u,w)
=Z_{N,\delta}^{-1} \exp\Big( -\delta \| (u_N,w_N) \|_{\vec \A}^{20} - \mathcal{H}_N^{\dia}(u,w) \Big)  \ind_{ \{|\int_{\T^3} :  u_N^2 :   dx| \le K\}}   d \muu (u,w)
\label{nutame}
\end{align}

\noi
where $\mathcal{H}_N^\dia$ is as in \eqref{density11} and
\begin{align}
Z_{N,\delta}
= \int \exp \Big( -\delta \| (u_N,w_N) \|_{\vec \A}^{20} - \mathcal{H}_N^{\dia}(u,w) \Big)\ind_{ \{|\int_{\T^3} :  u_N^2 :   dx| \le K\}} d \muu (u,w).
\label{ZNtame}
\end{align}

\noi
Thanks to the taming $-\delta \| (u_N,w_N) \|_{\vec \A}^{20}$ effect, we can prove that  $\{\nuu_{N, \dl}\}_{N \in \N}$ is tight. 


\begin{proposition} \label{PROP:nutame}
Let $\ld \in \R\setminus\{0\}$ and $K>0$.
Then, given any  $\delta > 0$, the sequence of measures $\{ \nuu_{N,\delta} \}_{N \in \N}$ is tight, and thus there exists a subsequence $\{ \nuu_{N_k, \delta} \}_{k \in \N}$ that converges weakly to a probability measure $\nuu_\delta$. Similarly $Z_{N_k,\delta}$ converges to $Z_\delta$. Moreover, $\| (u,w) \|_{\vec \A}$ is finite $\nuu_\delta$-almost surely. Therefore, we can define a $\s$-finite version of the Gibbs measure $ \rhoo_\dl $ with respect to the reference measure $\nuu_\dl$ as follows
\begin{align}
 d  \rhoo_\dl  =  e^{\dl \|(u,w)\|_{\vec \A}^{20} }d \nuu_\dl
\label{sigmafine}
\end{align}

\noi
for any $\dl > 0$.

\end{proposition}

As explained above, $\dl \|(u,w)\|_{\vec \A}^{20}$ in the exponent of \eqref{sigmafine} and $-\delta \| (u_N,w_N) \|_{\vec \A}^{20}$ in the exponent of \eqref{nutame} are cancelled from each other when taking the limit as $N \to \infty$. Therefore, at a conceptual level, the right-hand side of~\eqref{sigmafine} formally looks like
\begin{align*}
Z_\dl^{-1} \lim_{N\to \infty} e^{-\H^\dia_N(u,w)} \ind_{ \{|\int_{\T^3} :  u_N^2 :   dx| \le K\}} d\muu.
\end{align*}

\noi
Consequently, we can refer to the measure $\rhoo_\delta$ in \eqref{sigmafine} as a $\sigma$-finite version of the Gibbs measure $\rhoo$.

We now turn to the next proposition, which shows that the $\sigma$-finite version $\rhoo_\delta$ of the Gibbs measure $\rhoo$, constructed in \eqref{sigmafine}, cannot be normalized into a probability measure when the coupling effect is strong  $|\lambda| \gg 1$.

\begin{proposition} \label{PROP:Nonnor}
Let $|\ld| \gg 1$ and $K>0$. Given  $\delta > 0$, let  $\nuu_\delta$ be the measure constructed  in Proposition~\ref{PROP:nutame} and let $\rhoo_\dl$ be as in \eqref{sigmafine}.
Then, we have  
\begin{align}
\int 1\, d  \rhoo_\dl = \int e^{\dl \|(u,w)\|_{\vec \A}^{20} } d \nuu_\delta = \infty.
\label{BBB00}
\end{align}
\end{proposition}

Thanks to Proposition \ref{PROP:Nonnor}, we can obtain the following non-convergene result.

\begin{proposition}\label{PROP:result2}
Let $|\ld| \gg 1$,  $K>0$, and $\vec \A = \vec \A(\T^3)$ be as in \eqref{vecA}. Then, the sequence $\{\rhoo_N\}_{N \in \N}$ of the truncated Gibbs measures defined in \eqref{truGibbsN} does not converge weakly to any limit as probability measures on $\vec \A$. The same claim holds for any subsequence $\{\rhoo_{N_k}\}_{k \in \N}$.
\end{proposition}

We first present the proof of  Proposition \ref{PROP:nutame} in Subsection \ref{SUBSEC:refm}. 
After then, we prove that the $\s$-finite version $\rhoo_\dl$ of the Gibbs measure $\rhoo$ is not normalizable in Subsection \ref{SUBSEC:Non}. Once we obtain  Proposition \ref{PROP:nutame} and \ref{SUBSEC:refm}, we finally show the proof of Proposition \ref{PROP:result2} in Subsection \ref{SUBSEC:nonconv}.

\begin{remark}\rm\label{REM:natu}
(i) There is the following characterization of the Besov spaces.  For $s > 0$, we have 
\begin{align*}
\| u \|_{B^{-2s}_{p, \infty}(\T^3)} \sim \sup_{t > 0} t^s \| p_t * u\|_{L^p(\T^3)}
\end{align*}

\noi
for any mean-zero function $u$ on $\T^d$. See Theorem 2.34 in \cite{BCD} for the proof. Hence, the space $\A$ can be essentially regarded as the Besov space 
$B^{-\frac{3}{4}}_{3, \infty}(\T^3)$. Indeed, we have $B^{-\frac{3}{4}}_{3, \infty}(\T^3) \subset \A$.

\smallskip

\noi
(ii)
In the weak coupling regime $0<|\ld|\ll 1$, the support of the limiting Gibbs measure $\rhoo$,  established in Theorem \ref{THM:Gibbs}\,(i), is included in the space $\vec \A \supset \vec{\mathcal{C}}^{-\frac 34}(\T^3)$. Hence, the choice of the taming in \eqref{nutame} is natural.

\end{remark}

\begin{remark} \label{REM:ac2}\rm
(i) Notice that  the $\s$-finite version $ \rhoo_\dl$ of the Gibbs measure $\rhoo$, defined in \eqref{sigmafine}, and the Gaussian free field $\muu$ are mutually singular, like the Gibbs measure $\rhoo$, constructed in Theorem \ref{THM:Gibbs} when $0<|\ld|\ll 1$. A minor adjustment  of  the argument in Subsection \ref{SUBSEC:singular} combined with the computation in Subsection~\ref{SUBSEC:refm} presented below (Step 1 of the proof of Proposition \ref{PROP:nutame}) shows that the tamed version~$\nuu_\dl$ of the Gibbs measure, established in Proposition \ref{PROP:nutame}, and the Gaussian free field $\muu$ are mutually singular. Therefore, the mutual singularity of the $\s$-finite version $ \rhoo_\dl$ of the Gibbs measure $\rhoo$ and Gaussian free field $\muu$ comes as a corollary.
 
\smallskip

\noi
(ii) In Subsection \ref{SUBSEC:AC},
we prove that the Gibbs measure $\rhoo$, constructed in Theorem \ref{THM:Gibbs} (i) when $0<|\ld|\ll 1$, is absolutely continuous with respect to the shifted measure
$\nuu=\Law (\<1>(1) -\ld \vec \ZZ^r(1) + \vec \W(1))$. 
By following the argument in Subsection \ref{SUBSEC:AC}, one can show that both the tamed version~$\nuu_\dl$ and  $\s$-finite version $ d\rhoo_\dl= e^{\dl \|(u,w)\|_{\vec \A}^{20} }d \nuu_\dl$ of the Gibbs measure $\rhoo$, constructed in Proposition \ref{PROP:nutame}, are also  absolutely continuous
with respect to the same shifted measure, even in the case $|\ld|\gg 1$. Therefore, the measure $\rhoo_\delta$ in \eqref{sigmafine} is a suitable and natural option that can be regarded as a $\sigma$-finite version of the Gibbs measure $\rhoo$.

\end{remark}

\subsection{Construction of $\s$-finite Gibbs measure}
\label{SUBSEC:refm}

In this subsection, we prove Proposition \ref{PROP:nutame}.  Before we start the proof, we first present several  preliminary lemmas.

\begin{lemma} \label{LEM:AnormH}
Let the $\A$-norm be as in \eqref{normA}.
Then, we have 
\begin{align*}
 \|u\|_{\A} \les \|u\|_{H^{-\frac 14}}.
\end{align*}

\end{lemma}

\begin{proof}
Thanks to the Schauder estimate \eqref{Schau}, we obtain the result.
\end{proof}

\begin{lemma} \label{LEM:gaussA}
We have $W^{-\frac{3}{4}, 3}(\T^3) \subset \A$ and thus the quantity $\| u \|_{\A}$ 
is finite $\mu$-almost surely.
Moreover, given any  $1 \le p < \infty$, we have
\begin{align}
\E_\mu \Big[ \| \pi_N u \|_{\A}^p \Big] \le C_p < \infty, 
\label{BWW0}
\end{align}

\noi
uniformly in $N \in \N \cup \{\infty\}$ with the understanding that $\pi_\infty = \Id$.
\end{lemma}

\begin{proof}
The first claim follows from the Schauder estimate \eqref{Schau}.
Regarding the bound \eqref{BWW0}, thanks to the Schauder estimate \eqref{Schau}, 
Minkowski's integral inequality, and the Wiener chaos estimate (Lemma \ref{LEM:hyp}), we have
\begin{align*}
\E_\mu \Big[ \|  \pi_N u\|_{\A}^p \Big] 
& \les \E_\mu \Big[ \| u \|_{W^{-\frac 34, 3}}^p \Big] 
\les \Big\| \| \jb{\nb}^{-\frac{3}{4}} u (x)\|_{L^p(\mu)}\Big\|_{L^3_x}^p \\
& \le p^\frac{p}{2} \Big\| \| \jb{\nb}^{-\frac{3}{4}} u (x)\|_{L^2(\mu)}\Big\|_{L^3_x}^p\\
& \le p^\frac{p}{2}\bigg( \sum_{n \in \Z^3} \frac{1}{\jb{n}^\frac 72 }\bigg)^p < \infty,
\end{align*}

\noi
which proves \eqref{BWW0}.
\end{proof}

We are now ready to present the proof of Proposition \ref{PROP:nutame}.

\begin{proof}[Proof of Proposition \ref{PROP:nutame}]

For the proof of Proposition \ref{PROP:nutame}, we split the steps as follows:

\begin{itemize}
\item[\bf{(1)}] 
Tightness of the tamed measures  $\{ \nuu_{N,\dl} \}_{N \in \N}$ and  its weak convergence to $\nuu_\dl$ up to a subsequence.

\smallskip

\item[\bf{(2)}] $\| (u,w) \|_{\vec \A}$ is finite, $\nuu_\delta$-almost surely.

\end{itemize}

We now start with the proof of the first step.

\smallskip

\noi
{\bf  Step (1):} We first prove that $Z_{N,\delta}$ in \eqref{ZNtame} is uniformly bounded in $N\in \N$. Thanks to
\begin{align}
 \ind_{\{|\,\cdot \,| \le K\}}(x) \le \exp\big( -  A |x|^\gamma\big) \exp\big(A K^\g\big)
\label{cut0001}
\end{align}

\noi
for any $K>0$, $\g>0$, and $A>0$, we set 
\begin{align}
Z_{N,\dl} &= \int e^{-\delta \| (u_N,w_N) \|_{\vec \A}^{20}  - \mathcal{H}_N^\dia(u,w)} \ind_{ \{|\int_{\T^3} :  u_N^2 :   dx| \le K\}}  d\muu (u,w) \notag \\
&\les_{A,K}  \int e^{-\delta \| (u_N,w_N) \|_{\vec \A}^{20}- \mathcal{H}_N^\dia(u,w) - A\big|\int_{\T^3} :\,u_N^2: \, dx\big|^3} d\muu (u,w) =:\wt Z_{N,\dl}. 
\label{PARIS04}
\end{align}

\noi
Once we prove the uniform exponential integrability, that is, the uniform boundedness of $\wt Z_{N,\delta}$ in $N \in \N$, by repeating the arguments in Subsection \ref{SUBSEC:uniform}, we can obtain the tightness of $\{ \nuu_{N,\delta} \}_{N \in \N}$ and therefore omit the details of the proof of tightness.

Thanks to \eqref{PARIS04} and the Bou\'e-Dupuis variational formula (Lemma \ref{LEM:BoueDupu}) with the change of variables \eqref{Divv4} and \eqref{Divv5}, we have 
\begin{align}
-\log \wt Z_{N,\delta}
&= \inf_{\dot {\vec \Upsilon}^N \in   \vec{\mathbb H}_a^1} \E
\bigg[  \delta \| \vec{\<1>}_N+\vec \Dr_{N} \|_{\vec \A}^{20} 
+\ld \int_{\T^3}  \<1>_{S,N} \Dr_{S,N} \Dr_{W,N} dx \notag \\
&\hphantom{XXXXXXXXXXX}+\frac \ld2\int_{\T^3}  \<1>_{W,N} \Dr_{S,N}^2 dx + \frac \ld2\int_{\T^3} \Dr_{S,N}^2 \Dr_{W,N} dx \notag \\
&\hphantom{XXXXXXXXXXX} + A \bigg| \int_{\T^3} \Big(  \<2>_{S,N} + 2 \<1>_{S,N} \Dr_{S,N} + \Dr_{S,N}^2 \Big) dx \bigg|^3 \notag \\
&\hphantom{XXXXXXXXXXX} + \frac{1}{2}  \int_0^1 \| \dot {\vec \Upsilon}^N(t)  \|_{\vec H^1_x}^2 dt  \bigg],
\label{PARIS001}
\end{align}

\noi
where $\vec {\<1>}_N=(\<1>_{S,N}, \<1>_{W,N})$, $\dot {\vec \Upsilon}^N=(\dot \Upsilon^{N}_S, \dot \Upsilon^{N}_W)$, and $\vec \Dr_N=(\Dr_{S,N}, \Dr_{W,N})$. Moreover, we recall that 
\begin{align*}
\Dr_{S,N}= \Upsilon_{S,N} - \ld \ZZ_{W,N}\\
\Dr_{W,N}= \Upsilon_{W,N} -  \ld  \ZZ_{S,N}
\end{align*}

\noi
where $\ZZ_{S,N} = \pi_N\ZZ_S^N$ and $\ZZ_{W,N} = \pi_N\ZZ_W^N$ are as in \eqref{Divv2} and \eqref{Divv3}.

To obtain the uniform boundedness of $\wt Z_{N,\dl}$ in $N\in \N$, we establish a uniform lower bound on the right-hand side of~\eqref{PARIS001}. We point out that unlike Subsection \ref{SUBSEC:uniform}, the smallness of the coupling constant $|\ld|$ is not exploited here. Instead, the extra positive term $\delta \| \vec {\<1>}_N+\vec \Dr_N  \|_{\vec \A}^{20}$ in \eqref{PARIS001} plays a key role as a coercive term. By taking an expectation, choosing $A> 0$ sufficiently small, and using Lemmas \ref{LEM:Cor0} and \eqref{LEM:Cor1} with Lemma \ref{LEM:Cor00} and \eqref{estiZZ}, we have
\begin{align}
\E\Bigg[A \bigg|  \int_{\T^3} \Big(  \<2>_{S,N}  + 2 \<1>_{S,N} \Dr_{S,N} + \Dr_{S,N}^2 \Big) dx \bigg|^3\Bigg] \ge C_0 \E  \Big[ \| \Upsilon_{S,N} \|_{L^2}^6\Big]
 - C_1 \E \Big[\| \Upsilon_{S,N} \|_{H^1}^2\Big]- C
\label{PARIS002}
\end{align}

\noi
for some $C_0 > 0$ and $0 < C_1 \le \frac 14$.
Hence, by combining \eqref{PARIS001}, \eqref{PARIS002}, and Lemma~\ref{LEM:Cor0} together with Lemma \ref{LEM:Cor00} and \eqref{estiZZ}, we obtain that
there exists $C_2 > 0$ such that
\begin{align}
-\log \wt Z_{N,\delta}
&\ge  \inf_{ \dot {\vec \Upsilon}^N  \in   \vec{\mathbb H}_a^1}\E\bigg[
\delta \|  \vec{\<1>}_N+ \pi_N\vec{\Upsilon}^N  -\ld \vec{\ZZ}^r_N  \|_{\vec \A}^{20} +\frac \ld2 \int_{\T^3} \Dr_{S,N}^2\Dr_{W,N} dx \notag \\
&\hphantom{XXXXXXXXXX} 
+ C_2   \| \Upsilon_{S,N} \|_{L^2}^6 
+ C_2\| \Upsilon_{S,N} \|_{H^1}^2 +C_2\| \Upsilon_{W,N} \|_{H^1}^2    \bigg] - C.
\label{PARIS005}
\end{align}

\noi
where $\vec{\ZZ}^r_N=(\ZZ_{W,N}, \ZZ_{S,N})$. By Young's inequality, we have 
\begin{align}
\bigg| \int_{\T^3}& \Upsilon_{S,N}^2  \ZZ_{S,N} dx \bigg|+\bigg| \int_{\T^3} \Upsilon_{S,N}  \Upsilon_{W,N}  \ZZ_{W,N} dx \bigg| \notag \\
&\leq \| \Upsilon_{S,N} \|_{L^2}^2  \| \ZZ_{S,N} \|_{{\mathcal C}^{1-\eps}}+\| \Upsilon_{S,N} \|_{L^2}  \| \Upsilon_{W,N} \|_{L^2} \| \ZZ_{W,N} \|_{{\mathcal C}^{1-\eps}} \notag   \\
&\leq \frac{C_2}{100|\ld|} \| \Upsilon_{S,N} \|_{L^2}^6  + \frac{C_2}{100|\ld|} \| \Upsilon_{W,N} \|_{H^1}^2 + \| \ZZ_{S,N} \|_{\mathcal{C}^{1-\eps}}^c+\| \ZZ_{W,N} \|_{\mathcal {C}^{1-\eps}}^c + C_\ld,
\label{PARIS006}
\end{align}

\noi
and
\begin{align}
\bigg| \int_{\T^3}& \Upsilon_{W,N}  \ZZ_{W,N}^2 dx \bigg| +\bigg| \int_{\T^3} \Upsilon_{S,N}  \ZZ_{S,N}  \ZZ_{W,N} dx \bigg| \notag \\
& \le  \| \Upsilon_{W,N} \|_{L^2} \| \ZZ_{W,N} \|_{\mathcal{C}^{1-\eps}}^2+\| \Upsilon_{S,N} \|_{L^2} \| \ZZ_{S,N} \|_{\mathcal{C}^{1-\eps}}\| \ZZ_{W,N} \|_{\mathcal{C}^{1-\eps}} \notag \\
&\leq \frac{C_2}{100|\ld|} \| \Upsilon_{S,N} \|_{H^1}^2 + \frac{C_2}{100|\ld|} \| \Upsilon_{W,N} \|_{H^1}^2 + \| \ZZ_{S,N} \|_{\mathcal{C}^{1-\eps}}^c+\| \ZZ_{W,N} \|_{\mathcal {C}^{1-\eps}}^c + C_\ld.
\label{PARIS007}
\end{align}


\noi
Hence, from \eqref{PARIS005}, \eqref{PARIS006}, \eqref{PARIS007} with  \eqref{YoungaJ} (with $\g = 20$), \eqref{estiZZ}, and Lemma  \ref{LEM:gaussA},
we obtain  
\begin{equation} 
\begin{split}
-\log \wt Z_{N,\delta}
\ge  \inf_{(\dot \Upsilon^{N}_S, \dot \Upsilon^{N}_W)\in   \vec{\mathbb H}_a^1}
\E\bigg [ & \, \frac \delta2 \| (\Upsilon_{S,N}, \Upsilon_{W,N})  \|_{\vec \A}^{20}  - \frac{|\ld|}2 \int_{\T^3 } |\Upsilon_{S,N}^2 \Upsilon_{W,N} |dx  \\
&  + \frac {C_2} 2   \| \Upsilon_{S,N} \|_{L^2}^6 + \frac {C_2}2\| \Upsilon_{S,N} \|_{H^1}^2+\frac {C_2}2\| \Upsilon_{W,N} \|_{H^1}^2   \bigg] - C.
\end{split}
\label{PARIS008}
\end{equation}

\noi
It suffices to estimate the $L^3$-norm of $\Upsilon_{S,N}$ and $\Upsilon_{W,N}$.
It follows from H\"older's inequality, \eqref{normA}, Sobolev's inequality, and the mean value theorem $|1 - e^{-t|n|^2}| \les (t |n|^2)^\ta$ for any $0 \le \ta \leq 1$ (here with $\dr=\frac 14$) that we have 
\begin{align*}
\int_{\T^3 }& |\Upsilon_{S,N}^2 \Upsilon_{W,N} |dx\les \|\Upsilon_{S,N}\|_{L^3}^3+\|\Upsilon_{W,N}\|_{L^3}^3\\
&\les t^{-\frac 98} (\| \Upsilon_{S,N} \|_\A^3+\| \Upsilon_{W,N} \|_\A^3) + \|(\Upsilon_{S,N} - p_t \ast \Upsilon_{S,N}, \Upsilon_{W,N} - p_t \ast \Upsilon_{W,N})\|_{\vec H^\frac 12 }^3 \\
&\les t^{-\frac 98} (\| \Upsilon_{S,N} \|_\A^3+\| \Upsilon_{W,N} \|_\A^3) + t^\frac 34 (\|\Upsilon_{S,N} \|_{H^1}^3+  \|\Upsilon_{W,N} \|_{H^1}^3)
\end{align*}

\noi
for $0<t \le 1$.
By choosing $ t^\frac 34  \sim  \big(1 + \frac{100|\ld|}{C_2} \| \Upsilon_{S,N} \|_{H^1} +\frac{100|\ld|}{C_2} \| \Upsilon_{W,N} \|_{H^1} \big)^{-1}$
and using Young's inequality, we have
\begin{align}
&|\ld| \int_{\T^3 } |\Upsilon_{S,N}^2 \Upsilon_{W,N} |dx \notag \\
&\le C_{C_2, |\ld|}
\| \vec{\Upsilon}_N \|_{\vec H^1}^\frac 32\| \vec{\Upsilon}_N \|_{\vec \A}^3
+ \frac {C_2}{100}  \| \Upsilon_{S,N}\|_{H^1}^2+   \frac {C_2}{100} \| \Upsilon_{W,N}  \|_{ H^1}^2 + 1 \notag \\
&\le C_{C_2, |\ld|, \dl} + \frac \delta 4\| \vec{\Upsilon}_N \|_{\vec \A}^{20}
 + \frac {C_2}{50}\| \Upsilon_{S,N}\|^2_{H^1}+\frac {C_2}{50} \|\Upsilon_{W,N} \|_{H^1}^2,
\label{BS09}
\end{align}

\noi
where $\vec{\Upsilon}_N=\pi_N\vec{\Upsilon}^N=(\Upsilon_{S,N}, \Upsilon_{W,N})$. Therefore, thanks to  \eqref{PARIS008} and \eqref{BS09}, we obtain that 
\begin{align}
Z_{N,\dl}\les \wt Z_{N,\delta}\le C_\dl < \infty, 
\label{DB14}
\end{align}

\noi
uniformly in $N \in \N$. 
Hence, by repeating the argument in Subsection \ref{SUBSEC:uniform}, one can prove the tightness of the tamed measures $\{\nuu_{N,\delta}\}_{N\in \N}$.

\smallskip

\noi
{\bf Step (2):}
As a next step, we prove  that $\| (u,w) \|_{\vec \A}$ is finite $\nuu_\delta$-almost surely.
Let $ \eta$ be a smooth function with compact support with 
$\int_{\R^3} | \eta  (\xi)|^2  d\xi = 1$
and define
\begin{align*}
\ft \rho(\xi)= \int_{\R^3} \eta (\xi - \xi_1) \cj{\eta (-\xi_1)}d\xi_1.
\end{align*}

\noi
Given $\eps > 0$, define $\rho_\eps$ by 
\begin{align}
 \rho_\eps(x) = \sum_{n \in \Z^3} \ft \rho (\eps n) e^{i n \cdot x}. 
 \label{rhoeps}
\end{align}

\noi
Notice that the support of $\ft \rho$ is compact and so the summation on the right-hand side of \eqref{rhoeps} is over finitely many frequencies. Thus, given  any $\eps>0$, there exists $N_0(\eps) \in \N$ such that 
\begin{align}
\rho_\eps \ast u = \rho_\eps \ast u_N
\label{BY02}
\end{align}

\noi
for any $N \ge N_0(\eps)$. Thanks to the Poisson summation formula, we have 
\begin{align*}
\rho_\eps (x) = \sum_{n \in \Z^3} \eps^{-3} \big|\F_{\R^3}^{-1}(\eta)(\eps^{-1} x + 2\pi n)\big|^2 \ge 0,
\end{align*}

\noi
where $\F_{\R^3}^{-1}$ denotes the inverse Fourier transform on $\R^3$.
By observing that 
\begin{align*}
\| \rho_\eps\|_{L^1(\T^3)} = 
\int_{\T^3} \rho_\eps (x) dx
= \ft \rho (0)
= \| \eta \|_{L^2(\R^3)}^2=1,
\end{align*}

\noi
we have, from Young's inequality, that 
\begin{align}
\| \rho_\eps \ast u \|_{\A} \le \| u \|_{\A}.
\label{BY03}
\end{align}

\noi
Moreover, $\{\rho_\eps\}$ defined above
is an approximation to the identity on $\T^3$
and thus for any distribution $u$ on $\T^3$, 
$\rho_\eps \ast u \to u$ in the $\A$-norm,  as $\eps \to 0$.

Let $\dl>\delta'>0$ and $\{\nuu_{N_k,\dl}\}_{k\in \N}$ be a subsequence of $\{\nuu_{N,\dl}\}_{N \in \N}$, obtained in {\bf Step 1} from the tightness, converging to the reference measure $\nuu_\dl$. Then, by repeating the argument in {\bf Step (1)}, we can obtain the tightness of $\{\nuu_{N_k,\dl'}\}_{k\in \N}$, which allows us to get another subsequence whose limit corresponds to $\nuu_{\dl'}$. By Fatou's lemma, the weak convergence of $\{ \nuu_{N_k, \dl} \}_{N \in \N}$ and $\{ \nuu_{N_k, \dl'} \}_{N \in \N}$\footnote{By allowing us to choose another subsequences $\{ \nuu_{N_{k_\l}, \dl} \}_{\l \in \N}$ and $\{ \nuu_{N_{k_\l}, \dl'} \}_{\l \in \N}$, which converges to $\nuu_\dl$ and $\nuu_{\dl'}$, respectively } from Step (1) with \eqref{BY02}, \eqref{BY03}, and the definition \eqref{nutame} of $\nuu_{N_k, \dl}$, we have 
\begin{align}
\int \exp\Big( (\delta - \delta') \|(u,w) \|_{\vec \A}^{20} \Big) d \nuu_\delta
&\le \liminf_{\eps \to 0} \int \exp\Big( (\delta - \delta') \|(\rho_\eps \ast u, \rho_\eps \ast w) \|_{\vec \A}^{20} \Big) d \nuu_\delta \notag  \\
&= \liminf_{\eps \to 0} \lim_{k \to \infty}  \int \exp\Big( (\delta - \delta') \|(\rho_\eps \ast u_{N_k}, \rho_\eps \ast w_{N_k}) \|_{\vec \A}^{20} \Big) d \nuu_{N_k,\delta} \notag \\
&\le \lim_{k \to \infty}  \int \exp\Big( (\delta - \delta') \| (u_{N_k}, w_{N_k}) \|_{\vec \A}^{20} \Big) d \nuu_{N_k,\delta} \notag \\
&= \lim_{k \to \infty} \frac{Z_{N_k,\delta'}}{Z_{N_k,\delta}} \int 1 \, d \nuu_{N_k,\delta'} \notag \\
&= \frac{Z_{\delta'}}{Z_\delta}.
\label{NICE0}
\end{align}

\noi
Therefore, we have 
\begin{align*}
\int \exp\Big((\delta - \delta') \|(u,w) \|_{\vec \A}^{20} \Big) d \nuu_\delta < \infty
\end{align*}

\noi
for any $\dl>\dl'>0$.
By choosing $\delta' = \frac \delta2$,
we obtain
\begin{align*}
\int \exp\Big(\frac \delta2 \|(u,w) \|_{\vec \A}^{20} \Big) d \nuu_\delta < \infty, 
\end{align*}

\noi
which shows that $\|(u,w)\|_{\vec \A}$ is finite almost surely with respect to $\nuu_\dl$.
Therefore, we conclude the proof of Proposition \ref{PROP:nutame}.
\end{proof}

\subsection{Non-normalizability of $\s$-finite Gibbs measure}
\label{SUBSEC:Non}
In this subsection, we present the proof of Proposition \ref{PROP:Nonnor}.

Let us first recall $\rho_\eps$ in \eqref{rhoeps}. It follows from  \eqref{BY03}, the weak convergence of $\{ \nuu_{N,\dl} \}_{N \in \N}$\footnote{Up to a subsequence}
(Proposition \ref{PROP:nutame}),  \eqref{BY02}, and \eqref{nutame}, we have
\begin{align*}
\int & \exp\Big( \delta  \|(u,w) \|_{\vec \A}^{20} \Big) d \nuu_\delta
\ge \int \exp\Big( \delta  \|(\rho_\eps \ast u, \rho_\eps \ast w)  \|_{\vec \A}^{20} \Big) d \nuu_\delta\\
&\ge \limsup_{L \to \infty} \int \exp\Big( \dl \min\big(  \|(\rho_\eps \ast u, \rho_\eps \ast w) \|_{\vec \A}^{20}, L\big) \Big) d \nuu_\delta \\
&=  \limsup_{L \to \infty} \lim_{N \to \infty}
\int \exp\Big( \dl \min\big(  \| (\rho_\eps \ast u_N, \rho_\eps \ast w_N)  \|_{\vec \A}^{20}, L\big) \Big) d \nuu_{N,\delta}\\
&=  \limsup_{L \to \infty} \lim_{N \to \infty}
Z_{N,\delta}^{-1} \int \exp\Big(\delta  \min\big(\| (\rho_\eps \ast u_N, \rho_\eps \ast w_N)     \|_{\vec \A}^{20},L\big)\\
&\hphantom{XXXXXXXXXXXXXXX} -\dl \| (u_N,w_N) \|_{\vec \A}^{20}  - \mathcal{H}_N^{\dia}(u,w)\Big) \ind_{ \{|\int_{\T^3} :  u_N^2 :   dx| \le K\}}   d \muu (u,w).
\end{align*}

\noi
Therefore, in order to prove \eqref{BBB00}, it suffices to show that
\begin{align}
\limsup_{L \to \infty} \lim_{N \to \infty}
\E_{\muu} \Big[ \exp\Big(&\delta  \min\big(\| (\rho_\eps \ast u_N, \rho_\eps \ast w_N)     \|_{\vec \A}^{20},L\big) \notag \\
&-\dl \| (u_N,w_N) \|_{\vec \A}^{20}  - \mathcal{H}_N^{\dia}(u,w)   \Big) \ind_{ \{|\int_{\T^3} :  u_N^2 :   dx| \le K\}}  \Big]
= \infty.
\label{div0}
\end{align}

\noi
Noting that 
\begin{align}
\begin{split}
\E_{\muu}\Big[ & \exp\big(G(u,w) \big) 
\cdot\ind_{\{ |\int_{\T^3} \, : u_N^2 :\, dx | \le K\}} \Big]\\
&\ge \E_{\muu} \Big[\exp\big( G(u,w) 
\cdot \ind_{\{ |\int_{\T^3} \, : u_N^2 :\, dx | \le K\}}\big)   \Big]- 1,
\end{split}
\label{pax1}
\end{align}

\noi
the  divergence \eqref{div0} (and thus \eqref{BBB00}) follows once we prove
\begin{align}
\lim_{L \to \infty} \liminf_{N \to \infty}  \E_{\muu} \Bigg[ \exp\bigg\{\Big(&\delta  \min\big(\| (\rho_\eps \ast u_N, \rho_\eps \ast w_N)     \|_{\vec \A}^{20},L\big) \notag \\
&-\dl \| (u_N,w_N) \|_{\vec \A}^{20}  - \mathcal{H}_N^{\dia}(u,w)   \Big) \ind_{ \{|\int_{\T^3} :  u_N^2 :   dx| \le K\}}  \bigg\} \Bigg]=\infty.
\label{div1}
\end{align}

\noi
Thanks to the Bou\'e-Dupuis variational formula (Lemma \ref{LEM:BoueDupu}) with recalling 
$\vec{\<1>}_N= (\<1>_{S,N},\<1>_{W,N})$, $\vec \Dr_N=(\Dr_{S,N}, \Dr_{W,N})$, and $\vec \dr=(\dr_S,\dr_W)$, we have
\begin{align}
\begin{split}
-& \log \E_{\muu} \Bigg[ \exp\bigg\{\Big(\delta  \min\big(\| (\rho_\eps \ast u_N, \rho_\eps \ast w_N)     \|_{\vec \A}^{20},L\big)  \\
&\hphantom{XXXXXXXXXX}-\dl \| (u_N,w_N) \|_{\vec \A}^{20}  - \mathcal{H}_N^{\dia}(u_N,w_N)   \Big) \ind_{ \{|\int_{\T^3} :  u_N^2 :   dx| \le K\}}  \bigg\} \Bigg] \\
&=  \inf_{\vec \dr \in \vec{\mathbb{H}}_a }
\E\Bigg[ \bigg\{
-\delta   \min\big(\|\rho_\eps \ast (\vec {\<1>}_N+\vec \Dr_N) \|_{\vec \A}^{20},L\big)  + \dl \big\| \vec {\<1>}_N+\vec \Dr_N \big\|_{\vec \A}^{20} \\
&\hphantom{XXXXXXXXXXXXX}
+ \mathcal{H}_N^\dia (\vec {\<1>} +\vec \Dr) \bigg\} \ind_{ \{|\int_{\T^3} :  (\<1>_{S,N}+\Dr_{S,N})^2 :   dx| \le K\}}  +  \frac{1}{2}  \int_0^1 \| \vec {\dr}(t) \|_{\vec L^2_x }^2 dt \Bigg],
\label{genev0}
\end{split}
\end{align}

\noi
where we used the following definition 
\begin{align*}
\rho_\eps \ast \vec {\<1>}_N:&=(\rho_\eps \ast\<1>_{S,N},\rho_\eps \ast\<1>_{W,N}),\\
\rho_\eps \ast \vec \Dr_N:&=(\rho_\eps \ast\Dr_{S,N},\rho_\eps \ast \Dr_{W,N}).
\end{align*}

\noi
With the change of variables \eqref{chAa0}, \eqref{chAa1}, and similar arguments in \eqref{REQQ000}, \eqref{MCB012}, we have 
\begin{align}
\begin{split}
\bigg|& \int_{\T^3}  \<1>_{W,N} \Dr_{S,N}^2 dx \bigg|=
\bigg| \int_{\T^3} \<1>_{W,N} ( \Upsilon_{S,N}^2 - 2 \ld \Upsilon_{S,N}  \ZZ_{W,N} + \ld^2  \ZZ_{W,N}^2) dx \bigg| \\
&\le C_\ld\Big( 1+ \| \<1>_{W,N} \|_{\mathcal{C}^{-\frac 12-\eps}}^c +  \| \ZZ_{W,N} \|_{\mathcal{C}^{1-\eps}}^c\Big)
+ \frac 1{100 |\ld|} \Big( \| \Upsilon_{S}^N \|_{L^2}^3
+ \| \Upsilon_{S}^N \|_{H^1}^2 \Big),
\end{split}
\label{genev1}
\end{align}

\noi
and
\begin{align}
&\bigg| \int_{\T^3}  \<1>_{S,N} \Dr_{S,N}   \Dr_{W,N} dx \bigg| \notag \\
&=\bigg| \int_{\T^3} \<1>_{S,N}(\Upsilon_{W,N}\Upsilon_{S,N}-\ld \Upsilon_{S,N}\ZZ_{S,N}-\ld \Upsilon_{W,N}\ZZ_{W,N}  +\ld^2  \ZZ_{W,N}  \ZZ_{S,N} )  dx\bigg| \notag \\
&\le C_\ld \Big(1+\| \<1>_{S,N} \|_{\mathcal{C}^{-\frac 12-\eps}}^c +  \| \ZZ_{S,N} \|_{\mathcal{C}^{1-\eps}}^c+\| \ZZ_{W,N} \|_{\mathcal{C}^{1-\eps}}^c  \Big) \notag \\
&\hphantom{XXXXXXXXX}+ \frac 1{100 |\ld|} \Big( \| \Upsilon_{S}^N \|_{L^2}^3+ \| \Upsilon_{W}^N \|_{L^2}^3+ \| \Upsilon_{S}^N \|_{H^1}^2+\| \Upsilon_{W}^N \|_{H^1}^2 \Big).
\label{genev2}
\end{align}

\noi
From Young's inequality, we have
\begin{align}
\begin{split}
&\bigg|  \int_{\T^3} \Dr_{S,N}^2 \Dr_{W,N} dx - \int_{\T^3} \Upsilon_{S,N}^2 \Upsilon_{W,N} dx \bigg|\\
&=\bigg| \int_{\T^3} \ld \Upsilon_{S,N}^2 \ZZ_{S,N}+2\ld \Upsilon_{S,N} \ZZ_{W,N}  \Upsilon_{W,N}-2\ld^2 \Upsilon_{S,N} \ZZ_{W,N} \ZZ_{S,N}-\ld^2 \ZZ_{W,N}^2 \Upsilon_{W,N}+\ld^3 \ZZ_{W,N}^2 \ZZ_{S,N}  dx \bigg| \\
&\le C_\ld \Big( \| \ZZ_{S,N} \|_{\mathcal{C}^{1-\eps}}^c+\| \ZZ_{W,N} \|_{\mathcal{C}^{1-\eps}}^c  \Big) + \frac 1{100 |\ld|} \Big( \| \Upsilon_{S}^N \|_{L^2}^3+ \| \Upsilon_{W}^N \|_{L^2}^3+ \| \Upsilon_{S}^N \|_{H^1}^2+\| \Upsilon_{W}^N \|_{H^1}^2 \Big).
\label{genev3}
\end{split}
\end{align}

\noi
It follows from \eqref{genev0}, \eqref{genev1}, \eqref{genev2}, \eqref{genev3}, Lemma \ref{LEM:Cor00}, and \eqref{estiZZ} that 
\begin{align}
\begin{split}
-& \log \E_{\muu} \Bigg[ \exp\bigg\{\Big(\delta  \min\big(\| (\rho_\eps \ast u_N, \rho_\eps \ast w_N)     \|_{\vec \A}^{20},L\big)  \\
&\hphantom{XXXXXXXXXX}-\dl \| (u_N,w_N) \|_{\vec \A}^{20}  - \mathcal{H}_N^{\dia}(u,w)   \Big) \ind_{ \{|\int_{\T^3} :  u_N^2 :   dx| \le K\}}  \bigg\} \Bigg] \\
&\le  \inf_{\dot {\vec \Upsilon}^N \in  \vec {\mathbb {H}}_a^1} \E \Bigg[ \bigg\{
-\delta   \min\big(\|\rho_\eps \ast (\vec {\<1>}_N+ \pi_N\vec \Upsilon^N-\ld \vec \ZZ^r_N)  \|_{\vec \A}^{20},L\big)  + \dl \big\| \vec {\<1>}_N+ \pi_N\vec \Upsilon^N-\ld \vec \ZZ^r_N
\big\|_{\vec \A}^{20} \\
&\hphantom{XXXXXXXXXX} +\frac \ld2 \int_{\T^3} \Upsilon_{S,N}^2 \Upsilon_{W,N} dx \bigg\}\cdot \ind_{ \{|\int_{\T^3} :  (\<1>_{S,N}+\Dr_{S,N})^2 :   dx| \le K\}}
+ \| \Upsilon^N_S \|_{L^2}^3+ \| \Upsilon_{W}^N \|_{L^2}^3\\
&\hphantom{XXXXXXXXXX}
+ \frac 3{4} \int_0^1 \|\dot {\vec \Upsilon}^N(t) \|_{\vec H^1_x}^2 dt \Bigg]
+C_\ld, 
\end{split}
\label{CSY2}
\end{align}

\noi
where $\dot {\vec \Upsilon}^N=(\dot \Upsilon_S^N, \dot \Upsilon_W^N)$, $\vec{\ZZ}^r_N=(\ZZ_{W,N}, \ZZ_{S,N})$, and $\Dr_{S,N}=\Upsilon_{S,N} - \ld  \ZZ_{W,N}$.

To prove the  divergence \eqref{div0} (and thus \eqref{BBB00}), our goal is to prove that
the right-hand side of \eqref{CSY2} tends to $-\infty$ when $N, L \to \infty$, provided that the size of the coupling constant $|\ld| >0$ is sufficiently large.

\begin{remark}\rm\label{REM:diff0}
To create the desired divergence in the variational formulation  \eqref{genev0}, at the formal level, the main idea is to construct a drift  $\vec \Dr$ such that $\vec \Dr$ looks like\footnote{This naive choice $\vec \Dr=-\vec {\<1>}+\text{``blowing up profile"}$ actually does not work in the variational setting because of the failure of the measurability condition in $\vec{\mathbb{H}}_a$ (i.e. $\vec \Dr$ is not an adapted process).} 
\begin{align*}
``\vec \Dr=- \vec{\<1>} + \text{blowing up perturbation}''
\end{align*}

\noi
which makes the potential energy $\H^\dia_N(\vec{\<1>} + \vec \Dr)$ blow up as in \cite{OOT1,OSeoT,OOT2}.
Note that, in \cite{OOT1,OOT2} where generalized grand-canonical Gibbs measures  like \eqref{GibbsH} were studied, such an approximation of $\vec {\<1>}=\vec{\<1>}(1)$ was constructed by $\vec{\<1>}(\frac 12)$ (or $\vec{\<1>}(1-\eps)$). We point out that if the approximation in our setting were $\vec{\<1>}(1)+\vec {\Dr}:=\big(\vec{\<1>}(1)-\vec{\<1>}(\frac 12) \big)+\text{perturbation}$ (or $\vec{\<1>}(1)+\vec \Dr:=\big(\vec{\<1>}(1)-\vec{\<1>}(1-\eps) \big)+\text{perturbation}$) as in \cite{OOT1,OOT2}, then it would not be possible to obtain Theorem \ref{THM:Gibbs}'s (ii) for any $K>0$ because of the presence of the gap between $\vec{\<1>}(\frac 12)$ and $\vec{\<1>}(1)$ $\big(  \vec{\<1>}(1-\eps)$ and $\vec{\<1>}(1)\big)$. Namely, in our setting, the Wick-ordered $L^2$-cutoff $\ind_{\{ |\int_{\T^d} \, : (\vec{\<1>}_N+\vec \Dr_N)^2 :\, dx | \le K\}}$ excludes the blowup profile for any \textit{small} cutoff size $K >0$. Therefore, the method in \cite{OOT1, OOT2} only allows to prove the non-normalizability result for  large $K \gg 1$ and thus we need to refine the argument to prove the divergence for {\it any} $K > 0$.  In \cite{OSeoT} where the Gibbs measure is absolutely continuous with respect to the Gaussian free field, a more refined approximation $``Z_M(t)"$ to $\vec{\<1>}(t)$ was introduced in order to prove the divergence for any $K > 0$, where $Z_M(t)$ is a Gaussian process converging to $\vec{\<1>}(t)$.
More precisely, by setting $\Dr(t):=-Z_M(t)+\text{perturbation}$, $\<1>_N(t) + \Theta_N(t)$ approximates the blowing up profile as $N\ge M\to \infty$. More remarkably, due to the construction of the Gaussian process $Z_M(t)$, the Wick-ordered $L^2$ norm of this process $\<1>_N + \Theta_N$ can be made as small as possible as $M\to \infty$, i.e. the cutoff in the Wick-ordered $L^2$ norm $\ind_{\{ |\int_{\T^d} \, : (\<1>_N+\Dr_N)^2 :\, dx | \le K\}}$ does not exclude the approximation $\<1>_N + \Theta^0_N$ even if $K$ is close to $0$. 

In the current situation, however, the Gibbs measure $\rhoo$ is singular with respect to the Gaussian free field. Therefore, a proper approximation $Z_M(t)$ to $<1>_N(t)$ is no longer a Gaussian process. The singularity forces the construction of an approximation process $Z_M(t)$ to be a non-Gaussian process $A_N(t)$ in the second-order Wiener chaos 
\begin{align*}
A_N(t):=\<1>_{S,N}(t)-\ld \ZZ_{W,N}(t)
\end{align*}

\noi
where $\ZZ_{W,N}(t)$ is in \eqref{Divv3}. Then, by setting $\Upsilon^N_S$ as follows
\begin{align*}
\Upsilon^N_S=-Z_M+\text{blowing up perturbation},
\end{align*}

\noi
at the formal level, one can see that the interaction potential $\mathcal{H}^{\dia}_{N}(\<1>_S- \ld  \ZZ_{W,N}+ \Upsilon^{N}_S , \<1>_W -  \ld  \ZZ_{S,N}+ \Upsilon^{N}_W)$ blows up regardless of the $L^2$-cutoff size $K$ i.e.~$\ind_{\big\{ |\int_{\T^d} \, : (\<1>_S- \ld  \ZZ_{W,N}+ \Upsilon^{N}_S )^2 :\, dx | \le K\big \}}$ as $N\ge M \to \infty$. In particular, unlike in \cite{OOT1,OOT2,OSeoT}, the desired process $Z_M(t)$ is no longer an independent Gaussian process for each Fourier mode. Hence, the explicit independence structure cannot be used as before.
\end{remark}

\medskip

In the following lemma, we construct an approximation process $Z_M(t)$ to the following process in $\H_{\le 2}$ (i.e.~the second order Wiener chaos) 
\begin{align}
A_N(t):=\<1>_{S,N}(t)-\ld \ZZ_{W,N}(t)
\label{UFR0}
\end{align}

\noi
by solving a stochastic differential equation. 

\begin{lemma} \label{LEM:approx}
Given $ M\gg 1$,  define $Z_M$ by its Fourier coefficients
as follows.
For $|n| \leq M$, $\ft Z_M(n, t)$ is a solution of the following  differential equation\textup{:}
\begin{align}
\begin{cases}
d \ft Z_{M}(n, t) = \jb{n}^{-1} M (\ft A(n, t)- \ft Z_{M}(n, t)) dt \\
\ft Z_{M}|_{t = 0} =0, 
\end{cases}
\label{ZZZ}
\end{align}

\noi
where $\ft A(n,t)=\ft {\<1>}_{S,N}(n,t)-\ld \ft \ZZ_{W}(n,t)$ and  we set $\ft Z_{M}(n, t)  \equiv 0$ for $|n| > M$. Then, $Z_M(t)$ is a centered stochastic process in $\H_{\le 2}$\footnote{i.e.~the second order Wiener chaos}, which is frequency localized on $\{|n| \le M \}$, 
satisfying 
\begin{align}
&\E \big[ |Z_M(x)|^2 \big]  \sim M,\label{CORN0}\\
&\E\bigg[ \|Z_M\|_{\A}^p \bigg] \les 1 \label{CORN4},\\
&\E\bigg[  2 \int_{\T^3} A_{N} Z_M dx - \int_{\T^3} Z_M^2 dx   \bigg]  \sim M, \label{CORN02}\\
&\E \bigg[  \Big|   \int_{\T^3} :\! (  A_{N}-Z_M)^2 \!: dx  \Big|^2      \bigg] \les M^{-1},    \label{CORN1}\\
&\E\bigg[\Big( \int_{\T^3} A_{N}  f_M dx \Big)^2\bigg] 
+ \E\bigg[\Big( \int_{\T^3} Z_M  f_M dx \Big)^2\bigg] \les M^{-2},   \label{CORN2}\\
&\E\bigg[\int_0^1 \Big\| \frac {d}{ds} Z_M(s) \Big\|^2_{H^1}ds\bigg] \les M^3 \label{Lavo}.
\end{align}
	
\noi
for any $N \ge M \gg 1$ and finite $p\ge 1$,  where  $Z_M =Z_M|_{t = 1}$, $A_N=\pi_N A|_{t = 1}$,
and
\begin{align}
 :\! ( A_N-Z_M)^2 \!: \, = ( A_N-Z_M)^2 - \E\big[ ( A_N-Z_M)^2 \big].
\label{ZZZ2}
\end{align}

\noi
Here, \eqref{CORN0} is independent of $x \in \T^d$. 
	
\end{lemma}

We present the proof of Lemma \ref{LEM:approx} in Subsection \ref{SUBSEC:approx}.  Next, we consider the construction of the blowing up profile $f_M$ given below. We first fix a parameter $M \gg 1$. Let $f: \R^3 \to \R$ be a real-valued Schwartz function
such that 
the Fourier transform $\ft f$ is a smooth, radial, and non-negative function 
supported on $\big\{\frac 12 <  |\xi| \le 1 \}$ such that  $\int_{\R^3} |\ft f (\xi)|^2 d\xi = 1$.
Define a function $f_M$  on $\T^3$ by 
\begin{align}
f_M(x) &:= M^{-\frac 32} \sum_{n \in \Z^3} \ft f\Big( \frac n M\Big) e_n, 
\label{JJWW01} 
\end{align}

\noi
where $e_n(z):=e^{in\cdot z}$ and $\ft f$ denotes the Fourier transform on $\R^3$ defined by 
\begin{align}
\ft f(\xi) = \frac{1}{(2\pi)^\frac{3}{2}} \int_{\R^3} f(x) e^{-i\xi \cdot x} dx.
\notag
\end{align}

\noi
Then, a direct calculation shows the following lemma.

\begin{lemma} \label{LEM:CN0}
For any $M \in \N$ and $\al>0$, we have
\begin{align}
\int_{\T^3} f_M^2 dx &= 1 + O(M^{-\alpha}), \label{JJWW02} \\
\int_{\T^3} (\jb{\nabla}^{-\al} f_M)^2 dx &\les M^{-2\al}, \label{JJWW03} \\
\int_{\T^3} |f_M|^3 dx  \sim \int_{\T^3} f_M^3 dx &
\sim M^{\frac 32}.
\label{JJWW04}
\end{align}
\end{lemma}

\begin{proof}
Regarding \eqref{JJWW02} and \eqref{JJWW03}, see the proof of Lemma 5.12 in \cite{OOT1}. 
It follows from \eqref{JJWW01} and the fact that $\ft f$ is supported on $\{ \frac 12 < |\xi| \le 1 \}$ that we have 
\begin{align}
\int_{\T^3} f_M^3 dx
= M^{-\frac 92} \sum_{n_1,n_2 \in \Z^3}
\ft f \Big( \frac{n_1} M \Big)
\ft f \Big( \frac{n_2} M \Big)
\ft f \Big( -\frac{n_1+n_2} M \Big)
\sim M^{\frac 32}.
\label{JJWW05}
\end{align}

\noi
The bound $\| f_M\|_{L^3}^3 \ges M^\frac{3}{2}$
follows from \eqref{JJWW05}, 
while $\| f_M\|_{L^3}^3 \les M^\frac{3}{2}$
follows from Hausdorff-Young's inequality, which proves \eqref{JJWW04}.
\end{proof}

\begin{remark}\rm
We emphasize that the blow-up profile $f_M$ in \eqref{JJWW01} looks like a highly concentrated profile whose $L^p$-norm, $p> 2$ blows up but it $L^2$-norm is preserved (like $L^2$-scaled soliton). 
Note that \eqref{JJWW04} (blow-up of $L^3$-norm) was crucially used in \eqref{CUS1}, which implies the desired rate $|\ld|M^3$ in \eqref{CNUV2}. Also, \eqref{JJWW02} ($L^2$-preserving property) played an essential role in \eqref{CNUV0} to guarantee that the divergence rate in~\eqref{CNUV2} is not faster than $M^3$.
\end{remark}

Before we start the proof of Proposition \ref{PROP:Nonnor}, define $ \al_{M, N}$ by
\begin{align} 
\al_{M, N}= \frac {\E \bigg[ 2 \int_{\T^3}A_NZ_M dx-\int_{\T^3}Z_M^2  dx\bigg]+\E  \bigg[2\ld\int_{\T^3}A_N\ZZ_{W,N}dx+\ld^2\int_{\T^3} \ZZ_{W,N}^2 dx  \bigg]    }{\int_{\T^3} f_M^2 dx}.
\label{almn}
\end{align}

\noi
for $N\ge M \gg 1$. To be motivated for the definition of \eqref{almn}, see \eqref{TTadple} and \eqref{MENly0}. In particular, thanks to \eqref{JJWW02}, \eqref{CORN02}, \eqref{UFR0}, and \eqref{estiZZ}, we have
\begin{align}
 \al_{M, N} \sim M
\label{logM}
\end{align}

\noi
for any $N \ge M\gg 1$. We now present the proof of Proposition \ref{PROP:Nonnor}.

\begin{proof}[Proof of Proposition \ref{PROP:Nonnor}]

As we already pointed out, it suffices to show \eqref{div1}.

Fix $N \in \N$, appearing in \eqref{CSY2}. For any $M \gg 1$ with $N\ge M$,
we set $f_M$, $Z_M$, and $\al_{M,N}$ as in \eqref{JJWW01}, Lemma \ref{LEM:approx}, and~\eqref{almn}. In the variational problem \eqref{CSY2},  we choose a drift $\dot \Upsilon_S^N$ and $\dot \Upsilon_W^N$  by setting
\begin{align}
\dot \Upsilon_S^N (t) &=  -\frac d{dt} Z_M(t) - \sgn(\ld) \sqrt{\al_{M,N} } f_M, \label{UNS0}\\
\dot \Upsilon_W^N (t) &= -\frac d{dt} Z_M(t) - \sgn(\ld) \sqrt{\al_{M,N}} f_M 
\label{UNW0}
\end{align}

\noi
where $\sgn(\ld)$ is the sign of $\ld \ne 0$.
Then, we have 
\begin{align}
\Upsilon^N_S :&
= \int_0^1  \dot \Upsilon_S^N(t) dt
= - Z_M - \sgn(\ld) \sqrt{\al_{M,N} } f_M, \label{UNS}\\
\Upsilon^N_W :&= \int_0^1 \dot \Upsilon_W^N(t) dt
= - Z_M - \sgn(\ld) \sqrt{\al_{M,N} } f_M \label{UNW}.
\end{align}

\noi
Note that for $N \ge M \ge 1$, 
we have $\Upsilon_{S,N} = \pi_N \Upsilon^N_S = \Upsilon^N_S$ and $\Upsilon_{W,N} = \pi_N \Upsilon^N_W = \Upsilon^N_W$ 
since $Z_M$ and $f_M$ are supported on the frequencies $\{|n|\le M\}$.

We start with the first two terms under the expectation in \eqref{CSY2}. From the definition \eqref{vecA}, we have
\begin{align}
\begin{split}
- & \delta  \min\Big(\big\| \rho_\eps \ast (\vec {\<1>}_N+ \pi_N\vec \Upsilon^N-\ld \vec \ZZ^r_N)   \big\|_{\vec \A}^{20},L \Big) + \dl \big\| \vec {\<1>}_N+ \pi_N\vec \Upsilon^N-\ld \vec \ZZ^r_N \big\|_{\vec \A}^{20} \\
& = -\delta   \min\Big\{ \big\| \rho_\eps \ast (\vec {\<1>}_N+ \pi_N\vec \Upsilon^N-\ld \vec \ZZ^r_N)   \big\|_{\vec \A}^{20}- \big\| \vec {\<1>}_N+ \pi_N\vec \Upsilon^N-\ld \vec \ZZ^r_N \big\|_{\vec \A}^{20}, \\
& \hphantom{XXXXXXX} L-  \big\| \vec {\<1>}_N+ \pi_N\vec \Upsilon^N-\ld \vec \ZZ^r_N \big\|_{\vec \A}^{20}  \Big\} \\
& =:  -\dl \min(\text{I}, \II).
\end{split}
\label{CNUVV0}
\end{align}

\noi
We first consider $\II$. It follows from Lemma \ref{LEM:AnormH}, \eqref{INT0P}, and Lemma \ref{LEM:CN0} that we have 
\begin{align}
\| f_M \|_\A
\les \| f_M \|_{H^{-\frac 14}}
\les \| f_M \|_{L^2}^{\frac 34} \| f_M \|_{H^{-1}}^{\frac 14}
\les M^{-\frac 14}.
\label{MPCN0}
\end{align}

\noi
Thanks to \eqref{UNS}, \eqref{UNW}, \eqref{logM}, and \eqref{MPCN0}, we have
\begin{align}
\begin{split}
\II & \geq  L-  2 \al_{M,N}^{10} \|  f_M \|_{\A}^{20} 
- C\Big( \|\vec {\<1>}_N \|_{\vec \A}^{20}+ \|Z_M\|_{\A}^{20} + |\ld| \| \vec \ZZ^r_N \|_{\vec \A}^{20}\Big)\\
 & \geq  L-  C_0 M^5 
- C\Big( \|\vec {\<1>}_N \|_{\vec \A}^{20}+ \|Z_M\|_{\A}^{20} + |\ld| \| \vec \ZZ^r_N \|_{\vec \A}^{20}\Big)\\
 & \geq  \frac 12 L
- C\Big( \|\vec {\<1>}_N \|_{\vec \A}^{20}+ \|Z_M\|_{\A}^{20} + |\ld| \| \vec \ZZ^r_N \|_{\vec \A}^{20} \Big)
\end{split}
\label{CNUVV1}
\end{align}

\noi
for $L \gg M^5$.
Note that the second term on the right-hand side is bounded under an expectation.

We now consider $\text{I}$ in \eqref{CNUVV0}. Let $\dl_0$ denote the Dirac delta on $\T^3$. 
Then, by using \eqref{UNS}, \eqref{UNW},  Young's inequality, Lemma \ref{LEM:AnormH}, 
\eqref{logM}, and \eqref{JJWW02} in Lemma~\ref{LEM:CN0} and by choosing $\eps = \eps(M) > 0$ sufficiently small, we have 
\begin{align}
\begin{split}
\text{I} & \ge - \Big|  \big\| \rho_\eps \ast (\vec {\<1>}_N+ \pi_N\vec \Upsilon^N-\ld \vec \ZZ^r_N)   \big\|_{\vec \A}^{20}- \big\| \vec {\<1>}_N+ \pi_N\vec \Upsilon^N-\ld \vec \ZZ^r_N \big\|_{\vec \A}^{20} \Big|\\ 
& \ge - C \| (\rho_\eps - \dl_0)  \ast (\vec {\<1>}_N+ \pi_N\vec \Upsilon^N-\ld \vec \ZZ^r_N) \|_{\vec \A} \| \vec {\<1>}_N+ \pi_N\vec \Upsilon^N-\ld \vec \ZZ^r_N   \|_{\vec \A}^{19} \\
& \ge - C \al_{M,N}^{10}\big( \| (\rho_\eps - \dl_0)  \ast f_M \|_{H^{-\frac 14}}^{20}- C\Big( \|\vec {\<1>}_N \|_{\vec \A}^{20}+ \|Z_M\|_{\A}^{20} + |\ld| \| \vec \ZZ^r_N \|_{\vec \A}^{20} \Big) \\
& \ge - C \eps^5  M^{10} 
-  C\Big( \|\vec {\<1>}_N \|_{\vec \A}^{20}+ \|Z_M\|_{\A}^{20} + |\ld| \| \vec \ZZ^r_N \|_{\vec \A}^{20} \Big) \\
& =  - C_0
-  C\Big( \|\vec {\<1>}_N \|_{\vec \A}^{20}+ \|Z_M\|_{\A}^{20} + |\ld| \| \vec \ZZ^r_N \|_{\vec \A}^{20} \Big).
\end{split}
\label{CNUVV2}
\end{align}

\noi
Hence, from \eqref{CNUVV0}, \eqref{CNUVV1}, and \eqref{CNUVV2}
together with Lemma \ref{LEM:gaussA}, Lemma \ref{LEM:approx}, and \eqref{estiZZ}, we obtain 
\begin{align}
\E\Big[-\dl \min(\text{I}, \II)\Big] \le C(\dl, \ld).
\label{CUS0}
\end{align}

Next, we deal with the third term under the expectation in \eqref{CSY2}. We point out that this term gives the main contribution for the divergence in \eqref{div1}. From \eqref{UNS}, \eqref{UNW}, and Young's inequality with  Lemma \ref{LEM:CN0}, we have 
\begin{align}
\begin{split}
& -\ld  \int_{\T^3} \Upsilon_{S,N}^2 \Upsilon_{W,N} dx - |\ld| \al_{M,N}^{\frac 32} \int_{\T^3} f_M^3 dx \\
&= \ld \int_{\T^3} Z_M^3 dx
+3 |\ld| \int_{\T^3} Z_M^2 \sqrt{\al_{M,N}} f_M dx
+ 3\ld \int_{\T^3} Z_M \al_{M,N} f_M^2 dx \\
&\ge - \eta |\ld| \al_{M,N}^{\frac 32} \int_{\T^3} f_M^3 dx
- C_{\eta} |\ld| \int_{\T^3} |Z_M|^3 dx
\end{split}
\label{pa1}
\end{align}

\noi
for any $0<\eta<1$. Then, it follows from \eqref{pa1} with $\eta = \frac 12$ and Lemmas \ref{LEM:CN0} and \ref{LEM:approx}
that for any  measurable set  $E$ with $\PP(E) >0$, we have
\begin{align}
\begin{split}
\E \bigg[-\ld  \int_{\T^3} \Upsilon_{S,N}^2 \Upsilon_{W,N} dx \cdot \ind_E \bigg]&\ge (1-\eta) |\ld| \al_{M,N}^{\frac 32} \int_{\T^3} f_M^3 dx \cdot \PP(E) - C_\eta |\ld| \E \bigg[ \int_{\T^3} |Z_M|^3 dx \bigg] \\
&\ges |\ld| M^3 - |\ld| M^{\frac 32} \\
&\ges |\ld| M^3
\end{split}
\label{CUS1}
\end{align}

\noi
for $M \gg 1$, where we used the fact that $Z_M$ is in $\H_{\le 2}$  and so the Wiener chaos estimate (Lemma \ref{LEM:hyp}) with \eqref{CORN0} implies  
\begin{align*}
\E\Big[ \|Z_M \|_{L^3_x}^3  \Big]\les  \big\|  \|Z_M \|_{L^2_\o} \big\|_{L^3_x}^3\sim M^{\frac 32}. 
\end{align*}

We now handle the fourth, fifth, and sixth terms  under the expectation in \eqref{CSY2}.
Notice that $\Upsilon_{S,N}, \Upsilon_{W,N} \in \H_{\le 2}$ from \eqref{UNS}, \eqref{UNW}.
Hence, by using the Wiener chaos estimate (Lemma \ref{LEM:hyp}) with \eqref{CORN0} and Lemma \ref{LEM:CN0} with \eqref{logM}, we have
\begin{align}
\E \Big[ \| \Upsilon_{S,N} \|_{L^2_x}^3 \Big]&\les \E \Big[ \| \Upsilon_{S,N} \|_{L^2_x}^2 \Big]^{\frac 32} \les M^{\frac 32}, \label{YSN000}\\
\E \Big[ \| \Upsilon_{W,N} \|_{L^2_x}^3 \Big]&\les \E \Big[ \| \Upsilon_{W,N} \|_{L^2_x}^2 \Big]^{\frac 32} \les M^{\frac 32}.
\label{YWN000}
\end{align}

\noi
As for the sixth term in \eqref{CSY2}, recall that  both $\ft Z_M$ and $\ft f_M$ are supported on $\{|n|\leq M\}$. Hence, from  \eqref{UNS0}, \eqref{UNW0}, \eqref{Lavo}, \eqref{logM}, and \eqref{JJWW02},  we have 
\begin{align} 
\E \bigg[ \int_0^1 \|\big( \dot \Upsilon^N_{S} (t), \dot \Upsilon^N_{W} (t) \big) \|_{H^1_x\times H^1_x}^2 dt \bigg] \les \E\bigg[\int_0^1 \Big\| \frac d {ds} Z_M(s) \Big\|^2_{H^1_x}ds\bigg]
 + M^2  \alpha_{M,N} \|f_M \|_{L^2}^2\les M^3. 
\label{CNUV0}
\end{align}

Suppose that for any $K>0$ and small $\dl_1 >0$, there exists $M_0=M_0(K,\dl_1) \geq 1$ such that 
\begin{align}
\PP\bigg\{ \Big|\int_{\T^3} ( \<2>_{S,N} + 2 \<1>_{S,N} \Dr_{S,N} + \Dr_{S,N}^2) dx \Big| \le K \bigg\} 
\ge 1 - \dl_1, 
\label{prob}
\end{align}
	
\noi
uniformly in $N \ge M \ge M_0$, where $\Dr_{S,N}=\Upsilon_{S,N} - \ld  \ZZ_{W,N}$, which will be proved at the end of the proof. We are now ready to put everything together.  It follows from \eqref{CSY2}, \eqref{CUS0}, \eqref{prob}, \eqref{CUS1}, \eqref{YSN000}, \eqref{YWN000}, and \eqref{CNUV0} that we have
\begin{align}
-& \log \E_{\muu} \Bigg[ \exp\bigg\{\Big(\delta  \min\big(\| (\rho_\eps \ast u_N, \rho_\eps \ast w_N)     \|_{\vec \A}^{20},L\big)  \notag \\
&\hphantom{XXXXXXXXXX}-\dl \| (u_N,w_N) \|_{\vec \A}^{20}  - \mathcal{H}_N^{\dia}(u,w)   \Big) \ind_{ \{|\int_{\T^3} :  u_N^2 :   dx| \le K\}}  \bigg\} \Bigg] \notag \\
&\hphantom{XX}\le -  C_1 |\ld| M^3 + C_2 M^3 +C(\dl, \ld)
\label{CNUV2}
\end{align}

\noi
for some constants  $C_1, C_2 > 0 $, provided that $L \gg M^5 \gg 1$ and $\eps = \eps(M) > 0$ sufficiently small. By taking the limits in $N$ and $L$, we obtain from \eqref{CNUV2} that
\begin{align*}
\limsup_{L \to \infty} & \lim_{N \to \infty}
\E_{\muu} \Bigg[ \exp\bigg\{\Big(\delta  \min\big(\| (\rho_\eps \ast u_N, \rho_\eps \ast w_N)     \|_{\vec \A}^{20},L\big)  \\
&\hphantom{XXXXXXXXXX}-\dl \| (u_N,w_N) \|_{\vec \A}^{20}  - \mathcal{H}_N^{\dia}(u,w)   \Big) \ind_{ \{|\int_{\T^3} :  u_N^2 :   dx| \le K\}}  \bigg\} \Bigg] \notag\\
&\ge \exp \Big(  C_1 |\ld| M^3 - C_2 M^3  - C_0(\ld) \Big)
\too \infty, 
\end{align*}

\noi
as $M \to \infty$,
provided that $|\ld|$ is sufficiently large. Hence, we prove \eqref{div0} and so conclude the proof of Proposition  \ref{PROP:Nonnor} once we prove \eqref{prob}.

It remains to prove \eqref{prob} for any $K>0$ and small $\dl_1>0$. We first estimate the following second moment
\begin{align}
\begin{split}
&\E \bigg[ \Big|\int_{\T^3} \Big(  \<2>_{S,N} + 2 \<1>_{S,N} \Dr_{S,N} + (\Dr_{S,N})^2 \Big) dx \Big|^2 \bigg],
\end{split}
\end{align}

\noi
where $\<2>_{S,N}=\<1>_{S,N}^2-\<tadpole>_{S,N}$ with $\<tadpole>_{S,N}=\E\big[ \<1>_{S,N}^2\big]$ as in \eqref{tadpole1}. Recalling $\Dr_{S,N}=\Upsilon_{S,N} - \ld  \ZZ_{W,N}$, we have 
\begin{align}
&\int_{\T^3}  \Big( \<2>_{S,N} + 2 \<1>_{S,N} \Dr_{S,N} + (\Dr_{S,N})^2 \Big) dx \notag \\
&=\int_{\T^3}  \<2>_{S,N} dx +\int_{\T^3} 2\<1>_{S,N}(\Upsilon_{S,N}-\ld \ZZ_{W,N})dx+\int_{\T^3} (\Upsilon_{S,N}^2-2\ld \Upsilon_{S,N} \ZZ_{W,N} +\ld^2 \ZZ_{W,N}^2 ) dx \notag \\
&=\text{I}+\II+\III 
\label{Wick2d}
\end{align}

\noi
for any $N \ge M \gg 1$. Noting that $\Upsilon^N_S= - Z_M - \sgn(\ld) \sqrt{\al_{M,N} } f_M$ in \eqref{UNS}, we obtain
\begin{align*}
\II&=-2\int_{\T^3} \<1>_{S,N} \big(Z_M+\sgn(\ld)\sqrt{\al_{M,N}} f_M+\ld \ZZ_{W,N} \big)dx\\
\III&=\int_{\T^3} \big(Z_M^2 +2\sgn(\ld)\sqrt{\al_{M, N}} Z_M f_M+\al_{M, N}f_M^2\big) dx\\
&\hphantom{X} +2\ld \int_{\T^3} \ZZ_{W,N}\big(Z_M+\sgn(\ld) \sqrt{\al_{M, N}}f_M\big) dx  +  \ld^2\int_{\T^3}\ZZ_{W,N}^2.
\end{align*}

\noi
By adding both $\II$ and $\III$, we have
\begin{align}
\II+\III&=\int_{\T^3} Z_M^2 dx +\int_{\T^3}\al_{M, N}f_M^2 dx-2\int_{\T^3} \<1>_{S,N} Z_M dx-\ld^2\int_{\T^3} \ZZ_{W,N}^2 dx \notag \\ 
&\hphantom{X} -2\sgn(\ld)\sqrt{\al_{M,N}}\int_{\T^3}  (A_N-Z_M)f_M dx \notag \\
&\hphantom{X}- 2\ld\int_{\T^3} (A_N-Z_M)\ZZ_{W,N}dx
\label{IIplusIII}
\end{align}

\noi
where $A_N=\<1>_{S,N}-\ld \ZZ_{W,N}$ in \eqref{UFR0}. From the definition of $A_N$\footnote{
For the convenience of the presentation, we are using the following short-hand notations
$A_N=A_N(1)$, $\<1>_{S,N}=\<1>_{S,N}(1)$, and $\ZZ_{W,N}=\ZZ_{W,N}(1)$ as above.}, we write
\begin{align}
\int_{\T^3} \<1>_{S,N}^2 dx&=\int_{\T^3}A_N^2 dx+2\ld \int_{\T^3} A_N \ZZ_{W,N} dx+\ld^2\int_{\T^3} \ZZ_{W,N}^2  dx \label{UFraun}\\
-2\int_{\T^3} \<1>_{S,N} Z_M dx&=-2\int_{\T^3} A_N Z_M dx -2\ld \int_{\T^3} \ZZ_{W,N} Z_M dx \label{UFraun1}.
\end{align}

\noi
By combining \eqref{IIplusIII}, \eqref{UFraun}, and \eqref{UFraun1}, we obtain
\begin{align}
\int_{\T^3} \<1>_{S,N}^2dx +\II+\III&=\int_{\T^3} (A_N-Z_M)^2 dx+\al_{M, N}\int_{\T^3}f_M^2 dx \notag \\
&\hphantom{X}-2\sgn(\ld)\sqrt{\al_{M,N}}\int_{\T^3}  (A_N-Z_M)f_M dx.
\label{MENELy}
\end{align}

\noi
By recalling the definition of $A_N=\<1>_{S,N}-\ld \ZZ_{W,N}$, we write
\begin{align}
\<tadpole>_{S,N}=\E\Big[\<1>_{S,N}^2\Big]=\E\Big[A_N^2+2\ld A_N\ZZ_{W,N}+\ld^2 \ZZ_{W,N}^2 \Big].
\label{TTadple}
\end{align}

\noi
Hence, thanks to the definition of $\al_{M, N}$ in \eqref{almn} and \eqref{TTadple}, we have\footnote{Note that $\<tadpole>_{S,N}$ is stationary i.e.~it does not depend on the spatial variable.}
\begin{align}
\al_{M, N}\int_{\T^3} f_M^2 dx-\int_{\T^3}\<tadpole>_{S,N} dx&=-\int_{\T^3} \E\big[(A_N-Z_M)^2 \big] dx.
\label{MENly0}
\end{align}

\noi
By combining \eqref{Wick2d}, \eqref{MENELy}, and \eqref{MENly0}, we have
\begin{align}
\int_{\T^3} & \Big( \<2>_{S,N} + 2 \<1>_{S,N} \Dr_{S,N} + (\Dr_{S,N})^2 \Big) dx=\text{I}+\II+\III \notag \\
&=\int_{\T^3} :\! ( A_N-Z_M)^2 \!: dx-2\sgn(\ld)\sqrt{\al_{M,N}}\int_{\T^3}  (A_N-Z_M)f_M dx, 
\label{lake01}
\end{align}

\noi

\noi
where 
\begin{align*}
:\! ( A_N-Z_M)^2 \!: =(A_N-Z_M)^2 -\E \big[(A_N-Z_M)^2 \big]. 
\end{align*}

We are now ready to prove \eqref{prob}. It follows from \eqref{CORN1} that 
\begin{align}
\E \bigg[  \Big|   \int_{\T^3} :\! (A_N-Z_M)^2 \!: dx  \Big|^2      \bigg] \les M^{-1}.
\label{pr01}
\end{align}
	
\noi
We next turn to the second term in \eqref{lake01}. Thanks to \eqref{logM} and \eqref{CORN2} in Lemma \ref{LEM:approx}, we have
\begin{align}
\E \bigg[ \Big| \sqrt{ \al_{M, N}}  \int_{\T^3} (A_N-Z_M)f_M dx   \Big|^2   \bigg] \les M^{-1}.
\label{pr02}
\end{align}

\noi
Hence, from \eqref{lake01}, \eqref{pr01}, and \eqref{pr02}, we obtain
\begin{align*}
&\E \bigg[ \Big|\int_{\T^3} \Big(  \<2>_{S,N} + 2 \<1>_{S,N} \Dr_{S,N} + (\Dr_{S,N})^2 \Big) dx \Big|^2 \bigg] \les M^{-1}.
\end{align*}

\noi
Therefore, by Chebyshev's inequality, given any $K > 0$ and $\dl_1$, there exists $M_0 = M_0(K, \dl_1) \geq 1$ such that 
\begin{align*}
\PP\bigg( \Big| \int_{\T^3} \Big(  \<2>_{S,N} + 2 \<1>_{S,N} \Dr_{S,N} + (\Dr_{S,N})^2 \Big) dx    \Big| > K \bigg) &\le C\frac{M^{-1} }{K^2} < \dl_1
\end{align*}

\noi
for any $M \ge M_0 (K,\dl_1)$. This proves  \eqref{prob}.

\end{proof}

\subsection{Proofs of the auxiliary lemmas}
\label{SUBSEC:approx}
In this subsection, we present the proof of the approximation lemma (Lemma~\ref{LEM:approx}).

We first briefly review the theory of stochastic integration with respect to continuous semimartingales. Let $X_t$ be a  semimartingale. The quadratic variation of the process $X_t$ is denoted by $\big[X \big]_t$ and defined as follows 
\begin{align*}
\big[X \big]_t:=\lim_{\|P \| \to 0} \sum_{k=1}^n(X_{t_k}-X_{t_{k-1}})^2
\end{align*}

\noi
where $P$ ranges over partitions of the interval $[0,t]$ and the norm of the partition 
$P$ refers to the mesh. If this limit exists, it is defined in terms of convergence in probability. More generally, the covariation of two semimartingales $X$ and $Y$ is
\begin{align*}
\big[X,Y\big]_t:=\lim_{\|P \| \to 0} \sum_{k=1}^n(X_{t_k}-X_{t_{k-1}})(Y_{t_k}-Y_{t_{k-1}}).
\end{align*}

\noi
Given continuous semimartingales $X_t$ and $Y_t$, they can be decomposed as follows
\begin{align*}
X_t&=M_t^X+A_t^X\\
Y_t&=M_t^Y+A_t^Y
\end{align*}

\noi
where $M_t^X$ and $M_t^Y$ are the continuous local martingale parts, and $A_t^X$ and $A_t^Y$ are the continuous finite variation processes. Notice that the covariation of a continuous finite variation process with any other process is zero
\begin{align}
[M^X,A^Y]&=0 \label{Cov1}\\
[A^X,M^Y]&=0 \label{Cov2}\\
[A^X,A^Y]&=0 \label{Cov3}.
\end{align}

\noi
This implies that the covariation of $X$ and $Y$ is determined by the covariation of their continuous local martingale parts
\begin{align*}
[X,Y]_t=[M^X, M^Y]_t.
\end{align*} 

\noi
Given two processes $H(s)$ and $K(t)$ that are predictable and square-integrable with respect to the continuous local martingales $M^X$ and $M^Y$, respectively, 
\begin{align*}
\E\int_0^t |H(s)|^2 dM^X(s)<\infty \\
\E\int_0^t |K(s)|^2 dM^Y(s)<\infty,
\end{align*}

\noi 
and integrable with respect to the finite variation parts $A^X$ and $A^Y$, respectively,
\begin{align*}
\E\int_0^t |H(s)| |dA^X(s)|<\infty \\
\E\int_0^t |K(s)| |dA^Y(s)|<\infty,
\end{align*}

\noi
we can write
\begin{align}
\int_0^t H(s) dX(s) =\int_0^t H(s) dM^X(s)+\int_0^t H(s) dA^X(s) \notag \\
\int_0^t K(s) dY(s) =\int_0^t H(s) dM^Y(s)+\int_0^t H(s) dA^Y(s).
\label{dlcfv0}
\end{align}

\noi 
When taking the expectation of the product of two It\^o integrals, the finite variation parts disappear because they do not contribute to the covariation. See \eqref{Cov1}, \eqref{Cov2}, and \eqref{Cov3}.  Hence, we have
\begin{align}
\E\bigg[\int_0^t H(s) dX(s) \int_0^t K(s) dY(s)   \bigg]=\E\bigg[\int_0^t H(s) K(s)  d\big[M^X, M^Y\big]_s \bigg].
\label{ITOMAR}
\end{align}


\noi 
We are now ready to present the proof of Lemma \ref{LEM:approx}.
\begin{proof}[Proof of Lemma \ref{LEM:approx}]
We set
\begin{align}
X_n(t):=\ft A_N(n, t)- \ft Z_{M}(n, t), 
\quad |n|\le M.
\label{ZZ1} 
\end{align}

\noi
It follows from \eqref{ZZZ} that $X_n(t)$ satisfies  the following stochastic differential equation:
\begin{align}
\begin{cases}
dX_n(t)=-\jb{n}^{-1}M X_n(t) dt +d\ft A_N(n,t)\\
X_n(0)=0
\label{SDEMAR}
\end{cases}
\end{align}	

\noi
for $|n|\le M$, where $\ft A_N(n,t)$ in \eqref{ZZ1} is an It\^o process 
\begin{align}
d\ft A_N(n,t)&=-\ld d\ft \ZZ_{W,N}(n,t)+d\ft {\<1>}_{S,N}(n,t) \notag \\
&=-\frac{\ld}{\jb{n}^2} \sum_{\substack{n_1+n_2=n\\ |n_1|\le N, |n_2|\le N}} \frac{B_{n_1}^S(t) B_{n_2}^W(t)   }{\jb{n_1} \jb{n_2}}dt +\frac 1{\jb{n}} dB^S_n(t).
\label{diff0}
\end{align}
 
\noi
Notice that the It\^o process  $\ft A_N(n,t)$ is not a martingale due to the drift term, but it is a continuous semimartingale since $\ft A_N(n,t)$ can be decomposed into the sum of a continuous local martingale $M_N(n,t)$ and a continuous finite variation process $F_N(n,t)$  
\begin{align}
\ft A_N(n,t)=M_N(n,t)+F_N(n,t)
\label{dlcfv}
\end{align}

\noi
where
\begin{align}
M_N(n,t)&=\int_0^t  \frac{1}{\jb{n}} dB_n^S(u) \\
F_N(n,t)&=-\int_0^t \frac{\ld}{\jb{n}^2} \sum_{\substack{n_1+n_2=n\\ |n_1|\le N, |n_2|\le N}} \frac{B_{n_1}^S(u) B_{n_2}^W(u)   }{\jb{n_1} \jb{n_2}}du \label{Fin1}
\end{align}

\noi 
Therefore, one can define the stochastic differential equation \eqref{SDEMAR} in the It\^o sense. By solving the stochastic differential equation \eqref{SDEMAR} with  It\^o's formula, we have
\begin{align}
X_n(t)&=\int_0^t e^{-\jb{n}^{-1}M(t-s)}d\ft A_N(n,s). 
\label{ZZ2}
\end{align}

\noi
From \eqref{ZZ1} and \eqref{ZZ2}, we obtain 
\begin{align}
\ft Z_{M}(n, t)= \ft A_N(n, t)-\int_0^t e^{-\jb{n}^{-1}M(t-s)}d\ft A_N(n,s).
\label{SDE1}
\end{align}

\noi
for $|n|\le M$.  As mentioned above, the continuous finite variation parts $F(n,t)$ and $F(m,t)$ in \eqref{Fin1} do not contribute to the covariation. See \eqref{Cov1}, \eqref{Cov2}, and \eqref{Cov3}. Therefore, the covariation of two processes $\ft A_N(n,t)$ and $\ft A_N(m,t)$ are given by   
\begin{align}
\big[&\ft A_N(n,s), \cj{\ft A}_N(m,s)\big]_t= \big[ M_N(n,s) \cj M_N(m,s) \big]  =0 \label{ZEROMP}
\end{align}

\noi
for any $n \neq m$, where in \eqref{ZEROMP} we used the fact that $B_n^S$ and $B_m^S$ are independent. 
Therefore, it follows from \eqref{ITOMAR} and \eqref{ZEROMP} that
\begin{align}
\E\bigg[ \int_0^t   H(s) d\ft A_N(n, s)  \cj{ \int_0^t  K(s) d\ft A_N(m, s) }   \bigg]&=0, 
\label{ASG2}
\end{align}

\noi
where two processes $H(s)$ and $K(t)$ that are predictable, square-integrable with respect to the continuous local martingale $M_N$, and  integrable with respect to the finite variation part $F_N$. Thanks to \eqref{SDE1} and \eqref{ASG2}, we get 
\begin{align}
\E \big[ |Z_M(x)|^2 \big]&=\sum_{|n| \le M}\sum_{|m| \le M}\E\big[\ft Z_M(n) \cj {\ft Z_M}(m) \big]  e^{i(n-m)x} \notag \\
&=\sum_{|n|\le M} \Bigg(\E\big[ |\ft A_N(n,t) |^2 \big]-2 \Re \E\bigg[ \cj {\ft A_N}(n,t) \int_0^t e^{-\jb{n}^{-1}M(t-s)}d\ft A_N(n,s)  \bigg]  \notag \\
&\hphantom{XX}+\E  \bigg|\int_0^t e^{-\jb{n}^{-1}M(t-s)}d\ft A_N(n,s)  \bigg|^2   \Bigg)
\label{ZZ3}
\end{align}

\noi
for any $M\gg 1$. Thanks to \eqref{Cov1}, \eqref{Cov2}, and \eqref{Cov3}, the quadratic variation of $\ft A_N(n,t)$ is given by
\begin{align}
\big[\ft A_N(n,s)\big]_t=\big[ M_N(n,s)\big]_t =\frac t{\jb{n}^2},\label{QUAMP}
\end{align}

\noi
where $M_N(n,t)$ is the continuous local martingale part in \eqref{dlcfv}. Combining \eqref{ITOMAR} and \eqref{QUAMP} gives  
\begin{align}
\E\bigg[ \int_0^t  H(s) d\ft A_N(n, s)  \cj{ \int_0^t  K(s) d\ft A_N(n, s) }  \bigg]&=\E\bigg[\int_0^t H(s) \cj K(s)  d\big[ M_N\big]_s \bigg] \notag \\
&=\int_0^t H(s) \cj K(s)  \frac{ds}{\jb{n}^2}. 
\label{ASG1}
\end{align}

\noi
where two processes $H(s)$ and $K(t)$ that are predictable, square-integrable with respect to the continuous local martingale $M_N$, and  integrable with respect to the finite variation part $F_N$. Thanks to \eqref{ASG1}, we have
\begin{align*}
\eqref{ZZ3}&=\sum_{|n| \le M} \bigg( \frac{1}{\jb{n}^2}
-\frac 2{\jb{n}^2}\int_0^1 e^{-\jb{n}^{-1}M (1-s)}ds \\
& \hphantom{XXXXX}
+\frac{1}{\jb{n}^2}\int_0^1 e^{-2\jb{n}^{-1}M (1-s)}ds \bigg)\\
& \sim M + O\Big(\sum_{|n|\le M }\frac 1{\jb{n} }\cdot  \frac{1}{M}\Big) \sim M.
\end{align*}

\noi 
This proves \eqref{CORN0}.

Regarding \eqref{CORN4}, thanks to the Schauder estimate \eqref{Schau}, Minkowski's integral inequality, and the Wiener chaos estimate (Lemma \ref{LEM:hyp}) with $Z_M\in \H_{\le 2}$, we have
\begin{align*}
\E_\PP \Big[ \|  Z_M\|_{\A}^p \Big] 
& \les \E_\PP \Big[ \| Z_M \|_{W^{-\frac 34, 3}}^p \Big] 
\les \Big\| \| \jb{\nb}^{-\frac{3}{4}} Z_M (x)\|_{L^p(d\PP)}\Big\|_{L^3_x}^p \\
& \le p \Big\| \| \jb{\nb}^{-\frac{3}{4}} Z_M (x)\|_{L^2(d\PP)}\Big\|_{L^3_x}^p \les_p 1,
\end{align*}

\noi
where the last step follows from proceeding as in \eqref{ZZ3} with the additional gain $\jb{n}^{-\frac 32}$. This proves \eqref{CORN4}.

By Parseval's theorem, \eqref{SDE1}, \eqref{CORN0}, and  
proceeding as in \eqref{ZZ3}, we have
\begin{align*}
\E\bigg[  2 \int_{\T^3} A_N & Z_M dx - \int_{\T^3} Z_M^2 dx   \bigg]
=\E \bigg[ 2 \sum_{|n| \le M  }\ft A_N(n)  \cj{ \ft Z_M(n)} -\sum_{|n |\le M }|  \ft Z_M(n) |^2    \bigg]\\
&=\E \bigg[ \sum_{|n | \le M } | \ft Z_M(n)   |^2+\sum_{|n|\leq M } \bigg( \int_0^1 e^{-\jb{n}^{-1}M (1-s) }d\ft A_n(s) \bigg)  \cj{\ft Z_M(n)}   \bigg]\\
&\sim M + O\Big(\sum_{|n| \le M } \frac 1{\jb{n}}\cdot \frac 1{M}\Big)\\
& \sim M
\end{align*} 

\noi 
for any $N\ge M\gg 1$.
This proves  \eqref{CORN02}.

Noting that $\ft A_N, \ft Z_M  \in \H_{\leq 2}$, it follows from the Wiener chaos estimate (Lemma \ref{LEM:hyp}) and It\^o's isometry \eqref{ITOMAR} with the quadratic variation \eqref{QUAMP} that we have 
\begin{align}
\E \bigg[ \Big( |\ft A_N(n) |^2- \E \big[ | \ft A_N(n)|^2 \big] \Big)^2 \bigg]\les \Big(\E\big[ |\ft A_N(n)|^2 \big] \Big)^2\sim \frac 1{\jb{n}^4}.
\label{FPKC1}
\end{align}

\noi
and
\begin{align}
\begin{split}
\E \bigg[ \Big( |\ft A_N(n)- & \ft Z_M(n)|^2-\E\big[ |\ft A_N(n)-\ft Z_M(n)|^2 \big] \Big)^2  \bigg]
 \les 
\Big(\E\big[ |\ft A_N(n)-\ft Z_M(n)|^2 \big] \Big)^2  \\
& = \frac 1 {\jb{n}^{4} } \bigg(\int_0^1 e^{-2\jb{n}^{-1}M^{1}(1-s)  } ds \bigg)^2 
\sim \frac 1 {\jb{n}^{2} } \cdot \frac 1 {M^2}.
\end{split}
\label{FPKC2}
\end{align}

We next turn to cross terms. Notice that the zero covariation $\big[\ft A_N(n,s), \cj{\ft A}_N(m,s)\big]_t=0$ with $n \neq m$ does not imply that they are independent process since $\ft A_N(n,s)$ is not a Gaussian process. Hence, as for the cross terms below, we perform the direct calculations about correlations. Note that 
\begin{align}
&\E \bigg[ \Big( |\ft A_N(n) |^2- \E \big[ | \ft A_N(n)|^2 \big] \Big)\Big( |\ft A_N(m) |^2- \E \big[ | \ft A_N(m)|^2 \big] \Big) \bigg] \notag \\
&=\E\Big[| \ft A_N(n)|^2| \ft A_N(m)|^2\Big]-\E\Big[ | \ft A_N(n)|^2 \Big]\E\Big[ | \ft A_N(m)|^2 \Big].
\label{LIZ1}
\end{align}

\noi
For any fixed $n\neq m$, we set 
\begin{align}
\text{I}:&=-\frac{\ld}{\jb{n}^2} \int_0^1 \sum_{\substack{n_1+n_2=n\\ |n_1|\le N, |n_2|\le N}} \frac{B_{n_1}^S(t) B_{n_2}^W(t)  }{\jb{n_1} \jb{n_2}} dt+\frac {1}{\jb{n}} \int_0^1 dB_n^S(t) \label{MANG0}\\
\text{II}:&=-\frac{\ld}{\jb{m}^2} \int_0^1 \sum_{\substack{m_1+m_2=m\\ |m_1|\le N, |m_2|\le N}} \frac{B_{m_1}^S(t) B_{m_2}^W(t)  }{\jb{m_1} \jb{m_2}} dt+\frac {1}{\jb{m}} \int_0^1 dB_m^S(t) \label{MANG1}
\end{align}

\noi
Thanks to \eqref{diff0} and Wick's theorem  (Lemma \ref{LEM:Wick}) with the condition $n \neq m$, we have 
\begin{align}
&\E\Big[| \ft A_N(n)|^2| \ft A_N(m)|^2\Big]=\E\Big[ \text{I} \cdot \overline{\text{I}} \cdot \II \cdot \cj{\II}  \Big] \notag \\
&=\bigg(\frac {\ld^2}{\jb{n}^4}  \sum_{\substack{n_1+n_2=n\\ |n_1|\le N, |n_2|\le N}} \frac{1  }{\jb{n_1}^2 \jb{n_2}^2}\bigg)\cdot
\bigg(\frac {\ld^2}{\jb{m}^4}  \sum_{\substack{m_1+m_2=m\\ |m_1|\le N, |m_2|\le N}} \frac{1  }{\jb{m_1}^2 \jb{m_2}^2}\bigg) \notag \\
&\hphantom{X}+\frac {\ld^2}{\jb{m}^2 \jb{n}^4 }  \sum_{\substack{n_1+n_2=n\\ |n_1|\le N, |n_2|\le N}} \frac{1  }{\jb{n_1}^2 \jb{n_2}^2}+\frac {\ld^2}{\jb{n}^2\jb{m}^4 }  \sum_{\substack{m_1+m_2=m\\ |m_1|\le N, |m_2|\le N}} \frac{1  }{\jb{m_1}^2 \jb{m_2}^2} \notag \\
&\hphantom{X}+\frac {1}{\jb{n}^2\jb{m}^2}
\label{LIZ2}
\end{align}

\noi
and
\begin{align}
&\E\Big[ | \ft A_N(n)|^2 \Big]\E\Big[ | \ft A_N(m)|^2 \Big] \notag \\
&=\bigg(\frac {\ld^2}{\jb{n}^4}  \sum_{\substack{n_1+n_2=n\\ |n_1|\le N, |n_2|\le N}} \frac{1  }{\jb{n_1}^2 \jb{n_2}^2}+\frac 1{\jb{n}^2}\bigg)\bigg(\frac {\ld^2}{\jb{m}^4}  \sum_{\substack{m_1+m_2=m\\ |m_1|\le N, |m_2|\le N}} \frac{1  }{\jb{m_1}^2 \jb{m_2}^2}+\frac 1{\jb{m}^2}\bigg).
\label{LIZ3}
\end{align}

\noi
Hene, by combining \eqref{LIZ1}, \eqref{LIZ2}, and \eqref{LIZ3}, we have
\begin{align}
\E \bigg[ \Big( |\ft A_N(n) |^2- \E \big[ | \ft A_N(n)|^2 \big] \Big)\Big( |\ft A_N(m) |^2- \E \big[ | \ft A_N(m)|^2 \big] \Big) \bigg]=0 
\label{LIZ30}
\end{align}

\noi
for any $n \neq m$.

We consider the next cross term. By \eqref{ZZ1}, we have
\begin{align*}
|\ft A_N(n)-  \ft Z_M(n)|^2-\E\big[ |\ft A_N(n)-\ft Z_M(n)|^2 \big]=|X_n|^2- \E\big[|X_n|^2\big]
\end{align*}

\noi
Hence, it suffices to show
\begin{align}
&\E\big[ \big( |X_n|^2- \E\big[|X_n|^2\big] \big)   \big(|X_m|^2- \E\big[|X_m|^2\big] \big)     \big] \notag \\
&=\E\big[| X_n|^2|  X_m|^2\big]-\E\big[ |X_n|^2 \big]\E\big[ | X_m|^2 \big]=0
\label{LIZ40}
\end{align}

\noi
for any $n \neq m$.
Thanks to \eqref{dlcfv0} and \eqref{diff0}, we obtain
\begin{align}
X_n(t)&=\int_0^t e^{-\jb{n}^{-1}M(t-s)}d\ft A_N(n,s) \notag \\
&=-\frac{\ld}{\jb{n}^2}\sum_{\substack{n_1+n_2=n\\ |n_1|\le N, |n_2|\le N}}\int_0^t e^{-\jb{n}^{-1}M(t-s)} \frac{B_{n_1}^S(s) B_{n_2}^W(s)   }{\jb{n_1} \jb{n_2}}ds+\frac 1{\jb{n}}\int_0^t e^{-\jb{n}^{-1}M(t-s)}  dB^S_n(s).
\label{LIZ4}
\end{align}

\noi
Hence, by proceeding as in \eqref{LIZ2} and \eqref{LIZ3} with \eqref{LIZ4} instead of \eqref{MANG0} and \eqref{MANG1}, we can obtain the desired result \eqref{LIZ40}. Similarly, by repeating the argument in  \eqref{LIZ2} and \eqref{LIZ3}  with \eqref{diff0}, \eqref{LIZ4}, Wick's theorem  (Lemma \ref{LEM:Wick}), and the condition $n \neq m$, we obtain 
\begin{align}
\E \bigg[ \Big( |\ft A_N(n) |^2- \E \big[ | \ft A_N(n)|^2 \big] \Big) \Big(|\ft A_N(m)-  \ft Z_M(m)|^2-\E\big[ |\ft A_N(m)-\ft Z_M(m)|^2 \big]   \Big)  \bigg]=0 
\label{LIZZ4}
\end{align}

\noi
for any $n\neq m$.

Hence, it follows from Plancherel's theorem, \eqref{ZZZ2}, the uncorrelatedness \eqref{LIZ30}, \eqref{LIZ40}, and \eqref{LIZZ4} of  $\big\{ |\ft A_N(n) |^2- \E \big[ | \ft A_N(n)|^2 \big]\big\}_{M < |n|\le N}$ and
\begin{align*}
\big\{ |\ft A_N(n)-\ft Z_M(n)|^2-\E\big[ |\ft A_N(n)-\ft Z_M(n)|^2\big]\big\}_{|n|\le M},
\end{align*}

\noi 
that we have
\begin{align*}
\E \bigg[  & \Big|   \int_{\T^3} :\! ( A_N-Z_M)^2 \!: dx  \Big|^2      \bigg] \notag \\
&= \sum_{M<|n|\le N } \E \bigg[ \Big( |\ft A_N(n) |^2- \E \big[ | \ft A_N(n)|^2 \big] \Big)^2 \bigg]\notag \\
&\hphantom{XX}+ \sum_{|n|\le M} \E \bigg[ \Big( |\ft A_N(n)-\ft Z_M(n)|^2-\E\big[ |\ft A_N(n)-\ft Z_M(n)|^2 \big] \Big)^2  \bigg]\notag  \\
&\les \sum_{M<|n|\le N}\frac 1{\jb{n}^{4}}
+\sum_{|n|\le M }
\frac 1{\jb{n}^{2}}
\frac1 {M^2}
\les M^{-1},
\end{align*}

\noi
where we used \eqref{FPKC1} and \eqref{FPKC2} in the second part. This proves \eqref{CORN1}.

As for \eqref{CORN2}, by exploiting \eqref{ITOMAR} with zero covariation \eqref{ZEROMP} and  the quadratic variation  \eqref{QUAMP}, and \eqref{JJWW03}, we have
\begin{align}
\begin{split}
\E\bigg[ \Big( \int_{\T^3} A_N  f_M dx  \Big)^2\bigg]
&= \E \bigg[ \Big( \sum_{|n| \le M}  \ft A_N(n) \cj{ \ft f_M(n)} \Big)^2 \bigg]
\sim \sum_{|n| \le M} \frac 1{\jb{n}^2} |\ft f_M(n)|^2 \\
&\le \int_{\T^3} \big(\jb{\nb}^{-1} f_M (x)\big)^2 dx
\les M^{-2}.
\end{split}
\label{app4}
\end{align}

\noi
By proceeding as in \eqref{app4} with \eqref{ZZ2}, It\^o's isometry \eqref{ITOMAR} with zero covariation \eqref{ZEROMP} and  the quadratic variation  \eqref{QUAMP}, and \eqref{JJWW03}, we have 
\begin{align}
\begin{split}
 \E \bigg[ \Big( \sum_{|n| \le M} X_n(1) \cj{ \ft f_M(n)} \Big)^2 \bigg]
&=\E \Bigg[ \bigg|\sum_{|n|\leq M} \bigg(   \int_0^1 e^{-\jb{n}^{-1} M(1-s) } d\ft A_n(s) \bigg)  \ft f_M(n)     \bigg|^2     \Bigg]\\
&=\sum_{|n| \le M}  \frac{ | \ft f_M(n)|^2}{\jb{n}^2}\int_0^1 e^{-2\jb{n}^{-1}M (1-s)}ds \\
&\les M^{-1}\sum_{|n| \leq M} \frac 1{\jb{n} }| \ft f_M(n)|^2\\
& \les M^{-2}.
\end{split}
\label{app5}
\end{align}

\noi
Hence, \eqref{CORN2} follows from \eqref{app4} and \eqref{app5} with \eqref{SDE1}.

Lastly, thanks to \eqref{ZZZ}, \eqref{ZZ1}, \eqref{ZZ2}, and It\^o's isometry \eqref{ITOMAR} with the quadratic variation \eqref{QUAMP}, we have 
\begin{align*}
\E\bigg[\int_0^1 \Big\| \frac d {ds} Z_M(s) \Big\|^2_{H^1_x}ds\bigg] &= M^2 \E\bigg[\int_0^1 \Big\| \pi_M(A_N(s)) - Z_M(s) \Big\|^2_{L^2}ds\bigg] \\
&= M^2 \E\bigg[ \int_0^1 \Big(\sum_{|n| \le M} |X_n(s)|^2\Big)  ds\bigg]\\
&=M^2 \sum_{|n| \le M} \frac 1 {\jb{n}^{2}} \int_0^1 \int_0^s e^{-2\jb{n}^{-1}M(s-s')} d s' ds \\
& \les M^2 \sum_{|n| \le M } \frac 1{\jb{n}}\cdot \frac 1{M}\\
&  \les M^3, 
\end{align*}

\noi
yielding  \eqref{Lavo}. This  completes the proof of Lemma~\ref{LEM:approx}.

\end{proof}

\subsection{Non-convergence of the truncated Gibbs measures}
\label{SUBSEC:nonconv}
In this subsection, we prove Proposition \ref{PROP:result2} on non-convergence of the truncated Gibbs measures $\{\rhoo_N\}_{N \in \N}$, even up to a subsequence. 

For the proof, we first define a slightly different  tamed version $\{\nuu_{\delta}^{(N)} \}$ of the truncated Gibbs measure $\{\rhoo_N\}$ by setting
\begin{align}
d \nuu_{\delta}^{(N)} (u,w)=(Z_{\delta}^{(N)})^{-1} \exp \Big( -\delta \| (u,w) \|_{\vec \A}^{20} -\mathcal{H}_N^{\dia}(u,w) \Big) \ind_{\{\int_{\T^3} \; :|u_N|^2: \;  dx \, \leq K\}}d \muu (u,w)
\label{nutame1}
\end{align}

\noi
for $N \in \N$, $K>0$, and $\dl>0$, where the $\vec \A$-norm and $\mathcal{H}_N^\dia$ are as in \eqref{vecA} and \eqref{addcount}, respectively,  and the partition function $Z_{\delta}^{(N)}$ is given by
\begin{align*}
Z_{\delta}^{(N)}
= \int \exp \Big( -\delta \| (u,w) \|_{\vec \A}^{20} - \mathcal{H}_N^{\dia}(u,w) \Big) \ind_{\{\int_{\T^3} \; :|u_N|^2: \;  dx \, \leq K\}} d \muu (u,w).
\end{align*}

\noi
Compared to the tamed version $\nuu_{N,\delta}$ in \eqref{nutame}, the frequency cutoff $\pi_N$ does not exist in the taming $-\delta \| (u,w) \|_{\vec \A}^{20}$ in \eqref{nutame1}. As a corollary to the proof of Proposition \ref{PROP:nutame}, one can obtain the following convergence result for $\nuu_\dl^{(N)}$.

\begin{lemma}\label{LEM:conv1}
Let  $\dl > 0$ and $K>0$. Then, both the subsequence $\{\nuu_{N_k,\dl} \}$ established in Proposition \ref{PROP:nutame} and a subsequence $\{ \nuu_{\delta}^{(N_k)} \}_{k \in \N}$ of \eqref{nutame1} converge weakly to the same limiting measure $\nuu_\dl$ constructed in Proposition \ref{PROP:nutame}.
\end{lemma}

Before we present the proof of Lemma \ref{LEM:conv1}, we introduce another lemma  
whose proof will be presented after proving Lemma \ref{LEM:conv1}.

\begin{lemma}\label{LEM:expintr}
Let $\dl>0$ and $K>0$. Then, we have
\begin{equation}
\int \exp \big(   \| (u_N,w_N) \|_{\vec W^{-\frac 58, 3}}  \big) d \nuu_{N,\delta} \le C_{\dl} < \infty, 
\label{expbound}
\end{equation}

\noi
uniformly in $N \in \N$.
\end{lemma}

By taking Lemma \ref{LEM:expintr} for granted, we first present the proof of Lemma \ref{LEM:conv1}.

\begin{proof}[Proof of Lemma \ref{LEM:conv1}]

Thanks to Proposition \ref{PROP:nutame}, we obtain a subsequence $\{\nuu_{N_k,\dl}\}_{k\in \N}$ which converges weakly to $\nuu_\dl$. Hence, to prove Lemma \ref{LEM:conv1}, it suffices to prove 
\begin{align*}
\lim_{N\to \infty} \bigg\{ \int F(u,w) d \nuu_{\delta}^{(N)}(u,w) -\int F(u,w) d\nuu_{N,\dl}(u,w) \bigg\}=0
\end{align*}

\noi
for any bounded continuous function  $F: \vec{\mathcal C}^{-\frac 12-\eps}(\T^3) \to \R$. From the definitions \eqref{nutame} and \eqref{nutame1} of $\nuu_{N, \dl}$ and $\nuu_\dl^{(N)}$, in the following, we prove 
\begin{align*}
\lim_{N \to \infty}
\bigg\{& \int F(u,w) \exp \Big( -\delta \| (u,w) \|_{\vec \A}^{20}  - \mathcal{H}_N^{\dia}(u,w) \Big) \ind_{\{\int_{\T^3} \; :|u_N|^2: \;  dx \, \leq K\}}  d \muu (u,w)\\
& - \int F(u,w)\exp \Big(  -\delta \| (u_N, w_N) \|_{\vec \A}^{20} - \mathcal{H}_N^{\dia}(u,w) \Big) \ind_{\{\int_{\T^3} \; :|u_N|^2: \;  dx \, \leq K\}}  d\muu (u,w) \bigg\} = 0,
\end{align*}

\noi
which shows that it suffices to prove
\begin{align*}
\lim_{N\to \infty}&  \int \Big| \exp \Big( -\delta \| (u,w) \|_{\vec \A}^{20} - \mathcal{H}_N^{\dia}(u,w) \Big)- \exp \Big( -\delta \| (u_N, w_N) \|_{\vec \A}^{20} - \mathcal{H}_N^{\dia}(u,w) \Big)\Big| \notag \\
&\hphantom{XXXXXXXXXXXXXXXXXXX} \times \ind_{\{\int_{\T^3} \; :|u_N|^2: \;  dx \, \leq K\}} d \muu (u,w)= 0.
\end{align*}

\noi
Thanks to the uniform boundedness of the frequency projector $\pi_N$ on $\vec \A$, we have  
\begin{align}
\| (u_N,w_N) \|_{\vec \A} \les \| (u,w) \|_{\vec \A},
\label{un00} 
\end{align}

\noi
uniformly in $N \in \N$. Then, it follows from the mean-value theorem and \eqref{un00} that there exists $c_0 > 0$ so that we have
\begin{align}
&\int   \Big| \exp \Big( -\delta \| (u,w) \|_{\vec \A}^{20} - \mathcal{H}_N^{\dia}(u,w) \Big)  - \exp \Big( -\delta \| (u_N,w_N) \|_{\vec \A}^{20} - \mathcal{H}_N^{\dia}(u,w) \Big)\Big| \, \notag \\
&\hphantom{XXXXXXXXXXXXXXXXXXX} \times \ind_{\{\int_{\T^3} \; :|u_N|^2: \;  dx \, \leq K\}} d \muu (u,w) \notag \\
& \les  \dl \int \exp \Big( -\delta \min\big(\| (u,w) \|_{\vec \A}^{20}, \| (u_N,w_N)\|_{\vec \A}^{20}\big) - \mathcal{H}_N^{\dia}(u,w) \Big) \notag  \\
&\hphantom{XXXXXXXXXXXX}\times \big|\|(u,w)\|_{\vec \A}^{20}-\|(u_N,w_N)\|_{\vec \A}^{20}\big| \, \ind_{\{\int_{\T^3} \; :|u_N|^2: \;  dx \, \leq K\}} d \muu (u,w) \notag \\
& \les \dl  \int \exp \Big( -\delta c_0 \| (u_N,w_N)\|_{\vec \A}^{20} - \mathcal{H}_N^{\dia}(u,w) \Big) \|(u,w)-(u_N,w_N)\|_{\vec \A}\|(u,w)\|_{\vec \A}^{19} \notag \\
&\hphantom{XXXXXXXXXXXXXXXXXXX}  \times  \ind_{\{\int_{\T^3} \; :|u_N|^2: \;  dx \, \leq K\}}  d \muu (u,w) \notag \\
& \les  \dl \int \exp \Big( -\delta c_0 \| (u_N, w_N)\|_{\vec \A}^{20} - \mathcal{H}_N^{\dia}(u,w) \Big) N^{-\frac 18} \|(u,w)\|_{\W^{-\frac 58, 3}}^{20} \notag \\
&\hphantom{XXXXXXXXXXXXXXXXXXX}\times \ind_{\{\int_{\T^3} \; :|u_N|^2: \;  dx \, \leq K\}} d \muu (u,w) ,
\label{INTK0}
\end{align}

\noi
where in the last step, we exploited the following estimate
\begin{align}
\| (u- u_N, w-w_N) \|_{\vec \A} \les \|(\pi_N^\perp u, \pi_N^\perp w)\|_{\W^{-\frac 34, 3}} \les N^{-\frac 18} \| (u,w) \|_{\W^{-\frac 58, 3}} 
\label{highfreq0}.
\end{align}

\noi
Here, \eqref{highfreq0} follows from \eqref{vecA}, the Schauder estimate \eqref{Schau}, and the fact that $\pi_N^\perp u = u - u_N$ has the frequency support $\{|n|\ges N\}$.  
Hence, thanks to \eqref{INTK0}, the definition of $\{\nuu_{N,\dl}\}$ in \eqref{nutame}, and
the uniform exponential integrability in \eqref{DB14}, we have
\begin{align}
\limsup_{N\to \infty}&  \int \Big| \exp \Big( -\delta \| (u,w) \|_{\vec \A}^{20} - \mathcal{H}_N^{\dia}(u,w) \Big) - \exp \Big( -\delta \| (u_N, w_N) \|_{\vec \A}^{20} - \mathcal{H}_N^{\dia}(u,w) \Big)\Big|  \notag \\
&\hphantom{XXXXXXXXXXXXXXXXXXX} \times \ind_{\{\int_{\T^3} \; :|u_N|^2: \;  dx \, \leq K\}} d \muu (u,w)  \notag \\
&\les \dl\lim_{N\to \infty} \int \exp \Big( -\delta c_0 \| (u_N,w_N) \|_{\vec \A}^{20} - \mathcal{H}_N^{\dia}(u,w) \Big) N^{-\frac 18} \|(u,w)\|_{\vec W^{-\frac 58, 3}}^{20}  \notag \\
&\hphantom{XXXXXXXXXXXXXXXXXXX} \times \ind_{\{\int_{\T^3} \; :|u_N|^2: \;  dx \, \leq K\}} d\muu (u,w) \notag \\
&= \dl \lim_{N\to \infty} N^{-\frac 18} Z_{N, c_0\dl} \int \|(u,w)\|_{\vec W^{-\frac 58, 3}}^{20}  d \nuu_{N, c_0\dl} \notag \\
&\les  \dl \lim_{N\to \infty} N^{-\frac 18}  \int \|(u,w)\|_{\vec W^{-\frac 58, 3}}^{20}  d \nuu_{N, c_0\dl} \label{J945}.
\end{align}

\noi
Thanks to the Fatou's lemma, \eqref{BY02}, and \eqref{BY03}, we have
\begin{align}
\int \|(u,w)\|_{\vec W^{-\frac 58, 3}}^{20} d \nuu_{N,c_0\delta}
&\le \liminf_{\eps \to 0} \int \|(\rho_\eps \ast u, \rho_\eps \ast w) \|_{\vec W^{-\frac 58, 3} }^{20}  d \nuu_{N,c_0\delta} \notag  \\
&= \liminf_{\eps \to 0}  \int \|(\rho_\eps \ast u_N, \rho_\eps \ast w_N) \|_{\vec W^{-\frac 58, 3}}^{20}  d \nuu_{N,c_0\delta} \notag \\
&\les \int \|(u_N, w_N) \|_{\vec W^{-\frac 58, 3}}^{20}  d \nuu_{N,c_0\delta} \label{J944}.  
\end{align}

\noi
Therefore, by combining \eqref{J945}, \eqref{J944}, and Lemma \ref{LEM:expintr}, we obtain
\begin{align*}
\lim_{N\to \infty}& \bigg\{ \int F(u,w) d \nuu_{\delta}^{(N)}(u,w) -\int F(u,w) d\nuu_{N,\dl}(u,w) \bigg\}\\
&\les \dl \lim_{N\to \infty} N^{-\frac 18}  \int \|(u_N,w_N)\|_{\vec W^{-\frac 58, 3}}^{20}  d \nuu_{N, c_0\dl} \\
&\les \dl \lim_{N\to \infty} N^{-\frac 18}=0.
\end{align*}

\noi
Therefore, we conclude the proof of Lemma \ref{LEM:conv1}.

\end{proof}

Before we present the main result of this subsection (i.e.~Proposition \ref{PROP:result2}),
we give the proof of Lemma \ref{LEM:expintr}.

\begin{proof}[Proof of Lemma \ref{LEM:expintr} ]
It suffices to prove the following uniform bound 
\begin{equation}
\int \exp \big(   \| (u_N,w_N) \|_{\vec W^{-\frac 58, 3}}  \big) d \nuu_{N,\delta} \le C_{\dl} < \infty, 
\label{expbound}
\end{equation}

\noi
uniformly in $N \in \N$.
Thanks to the Bou\'e-Dupuis variational formula (Lemma \ref{LEM:BoueDupu}) and \eqref{cut000}, we have  
\begin{align*}
-  & \log   \bigg(\int \exp \big(  \| (u_N,w_N)\|_{\vec W^{-\frac 58, 3}} \big) d \nuu_{N,\delta} \bigg) \\
&\ges  \inf_{(\dot \Upsilon^{N}_S, \dot \Upsilon^{N}_W)\in   \vec{\mathbb H}_a^1} \E
\bigg[  \delta \| (\<1>_{S,N} + \Dr_{S,N}, \<1>_{W,N}+\Dr_{W,N}) \|_{\vec \A}^{20}- \| (\<1>_{S,N} + \Dr_{S,N}, \<1>_{W,N}+\Dr_{W,N}) \|_{\vec W^{-\frac 58,3}}\\
&\hphantom{XXXXXXXXXX}+\ld\int_{\T^3}  \<1>_{S,N} \Dr_{S,N} \Dr_{W,N} dx+\frac \ld2\int_{\T^3}  \<1>_{W,N} \Dr_{S,N}^2 dx + \frac \ld2\int_{\T^3} \Dr_{S,N}^2 \Dr_{W,N} dx \\
&\hphantom{XXXXXXXXXX} + A \bigg| \int_{\T^3} \Big(  \<1>_{S,N}  + 2 \<1>_{S,N} \Dr_{S,N} + \Dr_{S,N}^2 \Big) dx \bigg|^3 \\
&\hphantom{XXXXXXXXXX} + \frac{1}{2}  \int_0^1 \|\big( \dot \Upsilon^N_{S} (t), \dot \Upsilon^N_{W} (t) \big) \|_{\vec H^1_x}^2 dt 
 \bigg]+ \log Z_{N,\delta},
\end{align*}

\noi
where  $\Dr_{S,N}= \Upsilon_{S,N} - \ld \ZZ_{W,N}$ and $\Dr_{W,N}= \Upsilon_{W,N} -  \ld  \ZZ_{S,N}$.
Thanks to Lemma \ref{LEM:Cor00} and \eqref{estiZZ}, we have that for any finite $p \geq 1$,  
\begin{align}
\E\Big[\| (\<1>_{S,N}, \<1>_{W,N} ) \|_{\vec W^{-\frac 58, 3}}^p
+ \| (\ZZ_{S,N}, \ZZ_{W,N}) \|_{\vec W^{-\frac 58, 3}}^p\Big] < \infty,
\label{BW004}
\end{align}

\noi
uniformly in $N \in \N$. 
Then, by repeating arguments in  \eqref{PARIS008} and \eqref{BS09} with  
Sobolev's inequality, Young's inequality, and \eqref{BW004}, we have
\begin{align*}
-  & \log   \bigg(\int \exp \big(  \| (u_N,w_N)\|_{\vec W^{-\frac 58, 3}} \big) d \nuu_{N,\delta} \bigg) \\
&\ges \inf_{(\dot \Upsilon^{N}_S, \dot \Upsilon^{N}_W)\in   \vec{\mathbb H}_a^1} \E\bigg[-\|(\Upsilon_{S,N}, \Upsilon_{W,N}) \|_{\vec W^{-\frac 58,3}} 
+ C_0 \big(  \| \Upsilon_{S,N} \|_{L^2}^6 + \| \Upsilon_{S,N} \|_{H^1}^2+\| \Upsilon_{W,N} \|_{H^1}^2\big)\\
&\hphantom{XXXXXXXXXXXXX} + \frac {\dl}4 \| (\Upsilon_{S,N}, \Upsilon_{W,N})  \|_{\vec \A}^{20} \bigg]
 - C_{C_0,\delta}\\
&\ges \inf_{(\dot \Upsilon^{N}_S, \dot \Upsilon^{N}_W)\in   \vec{\mathbb H}_a^1} \E\bigg[
 \frac {C_0}{10} \big(  \| \Upsilon_{S,N} \|_{L^2}^6 + \| \Upsilon_{S,N} \|_{H^1}^2+\| \Upsilon_{W,N} \|_{H^1}^2\big)+ \frac {\delta}4 \| (\Upsilon_{S,N}, \Upsilon_{W,N})  \|_{\vec \A}^{20} \bigg]
 - C_{C_0,\delta}\\
&\gtrsim -1,
\end{align*}

\noi
which proves \eqref{expbound}. 
\end{proof}

We are now ready to present the proof of Proposition \ref{PROP:result2}.

\begin{proof}[Proof of Proposition \ref{PROP:result2}]

To obtain a contradiction, suppose that there exists a subsequence $\{\rhoo_{N_k}\}_{k \in \N}$ which has a weak limit $\ld$ as probability measures on  $\vec \A$. 
Given any $\delta > 0$, thanks to Proposition \ref{PROP:nutame}, one can obtain a reference measure $\nuu_\dl$ as a weak limit of the tamed version $\nuu_{N_k,\dl}$ of the truncated Gibbs measure, allowing a subsequence of $\{ \nuu_{N_k,\dl} \}$. Then, it follows from Lemma \ref{LEM:conv1} with \eqref{nutame1} and  \eqref{truGibbsN}, we have 
\begin{align}
\int F(u,w) d\nuu_\delta (u,w)&=\lim_{k\to \infty} \int F(u,w) d \nuu_{\delta}^{(N_k)}(u,w) \notag \\
&=\lim_{k\to \infty}\int F(u,w) \frac{e^{-\dl \| (u,w) \|_{\vec \A}^{20} - \H_{N_k}^{\dia}(u,w) }  }
{\int e^{ -\delta \| (\varphi,\psi) \|_{\vec \A}^{20} - \H_{N_k}^{\dia}(\varphi,\psi)} d \muu_K (\varphi,\psi)}  d\muu_K(u,w) \notag \\
&=\lim_{k\to \infty } \int F(u,w)  \frac{e^{-\dl \| (u,w) \|_{\vec \A}^{20} }}
{\int e^{ -\delta \| (\varphi,\psi) \|_{\vec \A}^{20}} d \rhoo_{N_k} (\varphi,\psi)} d\rhoo_{N_k}(u,w) \notag \\
&=\int F(u,w) \frac{e^{-\dl \| (u,w) \|_{\vec \A}^{20} }}
{\int e^{ -\delta \| (\varphi,\psi) \|_{\vec \A}^{20}} d\ld (\varphi,\psi)} d\ld(u,w)
\label{SBB01}
\end{align} 

\noi
for any bounded continuous function  $F:  \vec{\A}(\T^3) \to \R$, where
\begin{align*}
d\muu_K(u,w):= \ind_{\{\int_{\T^3} \; :|u_N|^2: \;  dx \, \leq K\}} d \muu (u,w). 
\end{align*}

\noi
Notice that in the last step of \eqref{SBB01}, we used the weak convergence in $\vec \A$ of the truncated Gibbs measures $\rhoo_{N_k}$, since $\exp ( -\delta \| (u,w) \|_{\vec \A}^{20} )$ is continuous on $\vec \A$. Then, thanks to \eqref{SBB01} and \eqref{sigmafine}, we obtain
\begin{align}
 d\ld(u,w)
= \bigg(\int \exp \big( -\delta \| (\varphi,\psi) \|_{\vec \A}^{20} \big) d \ld (\varphi,\psi)\bigg) 
d \rhoo_\dl (u,w).
\label{nutame3}
\end{align}

\noi
Noticing that $\ld$ is a probability measure on $\vec \A$ by the assumption, we have 
$\| (u,w) \|_{\vec \A} < \infty$, $\ld$-almost surely. By combining the fact that  $\ld$ is a probability measure, \eqref{nutame3}, and Proposition \ref{PROP:Nonnor}, we have
\begin{align*}
1 &= \int 1 \, d \ld\\
&= \int \exp \Big( -\delta \| (u,w) \|_{\vec \A}^{20} \Big) d \ld (u,w) 
\int 1 \,  d  \rhoo_\dl (u,w)\\
&= \infty.
\end{align*}

\noi
This shows a contradiction. Hence, no subsequence of the truncated Gibbs measures $\rhoo_N$ has a weak limit as probability measures on  $\vec \A$.
\end{proof}

\end{document}